\documentclass[reqno, 11pt]{amsart} 
\usepackage{amsfonts, amsmath, amssymb, amsthm}
\usepackage{mathrsfs}

\usepackage{amsfonts,amsmath,latexsym,verbatim,amscd,mathrsfs,color,array}
\usepackage[colorlinks=true]{hyperref}
\usepackage{amsmath,amssymb,amsthm,amsfonts,graphicx,color}
\usepackage{amssymb}
\usepackage{pdfsync}
\usepackage{epstopdf}
\usepackage{cite}
\usepackage{graphicx}
\usepackage{amsmath}
\usepackage{amsfonts}%
\providecommand{\U}[1]{\protect\rule{.1in}{.1in}}

\numberwithin{equation}{section}

\newtheorem{theorem}{Theorem}[section]

\newtheorem{proposition}[theorem]{Proposition}

\newtheorem{remark}[theorem]{Remark}

\newtheorem{lemma}[theorem]{Lemma}

\begin{document}

	\title[Concentrated solutions to fractional]{Concentrated solutions to fractional Schr\"{o}dinger-Poisson system with non-homogeneous potentials}
	
	\author{Lintao Liu}
	\address[Lintao Liu]{Department of Mathematics, North University of China, Taiyuan 030051, Shanxi, PR China}
	\email{liulintao\_math@nuc.edu.cn}

	\author{Haidong Yang}
	\address[Haidong Yang]{School of Mathematics and Statistics, Central South
		University, Changsha, Hunan, PR China}
	\email{hdyang@csu.edu.cn}

	\maketitle

	\begin{abstract}
		This paper mainly investigates several limit properties of normalized solutions for the fractional Schr\"{o}dinger-Poisson system, including existence, concentration behaviors and local uniqueness. It is worth noting that our results on the existence and asymptotic behaviors of normalized solutions are obtained in a doubly nonlocal setting and without assuming homogeneity of the potential, which generalize the results in \cite{GDCDS} in several aspects and improve our previous work in \cite{LIUYANG}. Meanwhile, some precise properties of solution sequence such as energy estimates, decay estimates and uniform regularity are also established by introducing some new techniques.
		
	\end{abstract}
	

	\section{Introduction}
	In this paper, we study the solutions with prescribed mass for the following fractional Schr\"{o}dinger-Poisson system
	\begin{equation}\label{equ1-1}
		\left\{
		\begin{array}{ll}
			i\frac{\partial\psi }{\partial t}=(-\Delta)^{s}\psi+V(x)\psi-\phi \psi-a^{p-2}|\psi|^{p-2}\psi & \hbox{in $\mathbb{R}^{3}$},\\
			(-\Delta)^{s}\phi=\psi ^{2}& \hbox{in $\mathbb{R}^{3}$},
		\end{array}
		\right.
	\end{equation}
	where $\psi: \mathbb{R}^{3}\times\mathbb{R}\rightarrow\mathbb{C}$  is the time dependent wave function, $s\in (\frac{3}{4}, 1)$, $a>0$ and $2<p<2+\frac{4s}{3}$. We seek standing wave solutions to \eqref{equ1-1}, which are given by $\left(\psi(x, t)=e^{-i\mu t}u(x),  \phi(x)\right) $ for $\mu\in \mathbb{R}$. The pair $(u(x), \phi (x))$ thus satisfies the following equation
	\begin{equation}\label{equ1}
		\left\{
		\begin{array}{ll}
			(-\Delta)^{s}u+V(x)u=\mu u+\phi u+a^{p-2}|u|^{p-2}u& \hbox{in $\mathbb{R}^{3}$},\\
			(-\Delta)^{s}\phi=u^{2}& \hbox{in $\mathbb{R}^{3}$},
		\end{array}
		\right.
	\end{equation}
	where $\mu\in \mathbb{R}$ is the corresponding Lagrange multiplier. $(-\Delta)^{s}$ is a nonlocal operator defined by
	\begin{equation*}
		(-\Delta)^{s}u(x):=C_{s}P.V.\int_{\mathbb{R}^{3}}\frac{u(x)-u(y)}{|x-y|^{3+2s}}dy,
	\end{equation*}
	and the symbol $P.V.$ stands for the Cauchy principal value of the singular integral, $C_{s}$ is a suitable normalization constant. Laskin \cite{LP,LPR} originally proposed a method that generalizes the Feynman path integral from Brownian-like quantum paths to L\'{e}vy-like ones, and this generalization has practical implications for the energy spectrum of hydrogen-like atoms, such as the fractional Bohr atom. It is a well-established fact that system \eqref{equ1} holds profound physical significance, as it arises naturally in the context of semiconductor theory. Specifically, in \eqref{equ1}, the first equation is a nonlinear stationary equation, whose nonlinear term models the inter-particle interactions of many-body systems, and this equation is coupled with the Poisson equation and must be satisfied by the function $\phi$. This dependency implies that the potential is determined by the charge density of the wave function \cite{M}.
	
	Over the past few years, numerous scholars have extensively studied the existence and multiplicity of positive solutions, ground state solutions, sign-changing solutions and semiclassical states for the fractional Schr\"{o}dinger-Poisson system \eqref{equ1}, and have developed several effective methods for handling equations or systems containing local terms, such as the perturbation approach, quantitative deformation lemma, global compactness lemma, Pohozaev–Nehari manifold, approximation method and Ljusternik–Schnirelmann theory. We refer the readers to see \cite{DH,JC,MM,TTT,TTTT,TZ} and the references therein for more details. 
	
	It is noted that the research findings of scholars in the aforementioned studies are mainly obtained without imposing any constraints on the $L^{2}$-norm. However, in this paper, we focus more on the existence, asymptotic behavior and uniqueness of the normalized solutions for the fractional Schr\"{o}dinger-Poisson system \eqref{equ1}. Based on the current literature, only a few studies have addressed this type of problem. For instance, Meng and He \cite{MHCV} studied the existence and properties of ground states for the fractional
	Schr\"{o}dinger-Poisson system with combined power nonlinearities
	\begin{equation*}
		\left\{
		\begin{array}{ll}
			(-\Delta)^{s}u-\phi |u|^{2_{s}^{\ast}-3}=\lambda u+\mu |u|^{q-2}u+|u|^{2_{s}^{\ast}-2}u& \hbox{in $\mathbb{R}^{3}$},\\
			(-\Delta)^{s}\phi=u^{2_{s}^{\ast}-1}& \hbox{in $\mathbb{R}^{3}$},
		\end{array}
		\right.
	\end{equation*}
	where $\mu >0$ is a parameter, $2<q<2_{s}^{\ast}:=\frac{6}{3-2s}$ and $\lambda\in \mathbb{R}$ appears as a Lagrange multiplier. For the three perturbation scenarios of $L^{2}$-subcritical, $L^{2}$-critical and $L^{2}$-supercritical, the authors  proved several existence and non-existence results. Furthermore, they also studied the qualitative behavior of the ground states as $\mu\rightarrow0^{+}$. Later, He et al. \cite{HCV,HJGA} considered the existence of normalized solutions for the fractional critical Schr\"{o}dinger-Poisson system
	\begin{equation*}
		\left\{
		\begin{array}{ll}
			(-\Delta)^{s}u+\lambda\phi u=\alpha u+\mu |u|^{q-2}u+|u|^{2_{s}^{\ast}-2}u& \hbox{in $\mathbb{R}^{3}$},\\
			(-\Delta)^{t}\phi=u^{2}& \hbox{in $\mathbb{R}^{3}$},
		\end{array}
		\right.
	\end{equation*}
	where $s, t\in (0, 1)$ satisfy $2s+2t>3$, $q\in (2, 2_{s}^{\ast})$, $\lambda, \mu>0$ are parameters and $\alpha\in \mathbb{R}$ is an undetermined parameter. Under the subcritical perturbation,  the authors obtained the existence of multiple normalized solutions by means of the truncation technique, concentration-compactness principle and the genus theory. In the critical regime, they proved the existence results by applying the Jeanjean theory, Pohozaev manifold method and Brezis-Nirenberg technique. In the supercritical perturbation, they proved two different results for normalized solutions when parameters satisfy different assumptions, by employing the constrained variational methods and the mountain pass theorem.
	These results comprehensively address the existence of normalized solutions for the fractional Schr\"{o}dinger-Poisson system in the autonomous case, with no existing results on the concentration behavior and local uniqueness of normalized solutions in the non-autonomous case. 
	
	When $\phi(x)=0$ and $s\in (0, 1)$, system \eqref{equ1} reduces to a fractional Schr\"{o}dinger equation. In \cite{H}, He and Long studied the following fractional Schr\"{o}dinger equation
	\begin{equation*}
		(-\Delta)^{s}u+V(x)u=a|u|^{\frac{4s}{N}}u\quad \hbox{in $\mathbb{R}^{N}$},
	\end{equation*}
	where $N\geq1$, $a\in\mathbb{R}$ and $V(x)$ is a measurable function. The authors proved the existence or nonexistence of ground states under the $L^{2}$-constraint and certain assumptions on $V(x)$ and $a$. In addition, they analyzed the  behavior of ground states as $a$ tends to a fixed constant. Du et al. \cite{D} studied the existence, nonexistence and mass concentration of normalized solutions for nonlinear fractional Schr\"{o}dinger equations
	\begin{equation*}
		(-\Delta)^{s}u+V(x)u=\mu u+f(u)\quad \hbox{in $\mathbb{R}^{N}$},
	\end{equation*}
	where $N\geq2$, $V(x)$ is an external potential function,  $\mu\in\mathbb{R}$ and $a>0$ are parameters, and $f$ is a subcritical nonlinearity. By deriving the exact expression of the optimal embedding constant in the fractional Gagliardo–Nirenberg–Sobolev inequality, they established the existence and nonexistence of normalized solutions to the equation. Furthermore, the authors obtained the concentration behavior of normalized solutions under the mass critical case by using direct energy estimates. Li et al. \cite{LLZ} investigated the existence and blow-up profile of normalized solutions to the fractional nonlinear Schr\"{o}dinger equation
	\begin{equation*}
		(-\Delta)^{s}u+V(x)u+\lambda u=|u|^{\frac{4s}{N}}u\quad \hbox{in $\mathbb{R}^{N}$},
	\end{equation*}
	with $N\geq2$, $\lambda\in \mathbb{R}$ and negative potentials $V(x)$. The authors proved the existence and nonexistence of normalized solutions under negative potentials $V(x)$. Moreover, they also obtained nonexistence results and analyzed the asymptotic behavior of minimizers under two types of potentials: one is a bounded potential, and the other is a singular potential. Their study provides precise estimates of the energy during the convergence of minimizers and sharp decay rates of blow-up solutions. It is worth noting that the results mentioned above are primarily extensions and generalizations derived from the findings in \cite{GLMP,GAH}.
	
	When $\phi(x)=0$ and $s=1$, system \eqref{equ1} reduces to the Schr\"{o}dinger equation. The concentration behavior and uniqueness of the normalized solutions to such problems have been thoroughly investigated, with a wealth of results documented in the literature.  For instance, Guo and his collaborators have studied the following Gross-Pitaevskii equation
	\begin{equation}\label{equ01}
		-\Delta u+V(x)u=\mu u+af(x, u)\quad \hbox{in $\mathbb{R}^{N}$},
	\end{equation}
	where $a>0$ describes the strength of the attractive interactions. In \cite{GLMP,GSIM}, Guo et al. first studied the concentration behavior and local uniqueness of normalized solutions to equation \eqref{equ01} when $N=2$ and $f(x, u)=|u|^{2}u$ by applying energy estimates and blow-up analysis. Additionally, they studied the concentration and blow-up phenomena of normalized solutions to \eqref{equ01} under the conditions of ring-shaped trapping potentials \cite{GAH} and multi-well potentials \cite{GNA}. When $N\geq1$ and $f(x, u)=|u|^{p-1}u$ wtith $p\in(1, 1+\frac{4}{N})$, Li and Zhu \cite{LZ} proved the mass concentration and local uniqueness of ground states for problem \eqref{equ01}, which extends the concentration results shown in \cite{GLMP,GSIM}. Combining blow-up and the constraint variational method, Deng et al. \cite{DC,DJ} proved the existence and asymptotical behavior of ground state to equation \eqref{equ01} when $N=2$ and $f(x, u)=m(x)|u|^{2}u$ with $0<m(x)\leq1$. For the asymptotic behavior of solutions, we also refer the reader to \cite{CWJMPA,LWZ,WWMZ}. Recent studies have yielded several results on the planar Schr\"{o}dinger-Poisson system. In \cite{GJDE}, Guo et al. considered the following planar Schr\"{o}dinger-Poisson system
	\begin{equation}\label{equ02}
		-\Delta u+(V(x)-\mu) u+\gamma\left[ \Phi_{N}\ast u^{2}\right]u=|u|^{p-2}u\quad \hbox{in $\mathbb{R}^{N}$},
	\end{equation}
	where
	\begin{equation*}
		\Phi_{N}(x):=
		\left\{
		\begin{array}{ll}
			\frac{1}{N(2-N)\mathit{w}_{N}} |x|^{2-N}&\quad \hbox{$N\geq3$},\\
			\frac{1}{2\pi}\ln |x|&\quad \hbox{$N=2$},
		\end{array}
		\right.
	\end{equation*}
	and $\mathit{w}_{N}>0$ denotes the volume of the unit ball in $\mathbb{R}^{N}$. By overcoming the sign-changing property of the logarithmic convolution potential and the non-invariance under translations of the logarithmic external potential, the authors proved the existence and uniqueness of constraint minimizers for system \eqref{equ02} with $V(x)=\ln (1+|x|^{2})$, $N=2$, $p=4$ and $\gamma=2\pi$. Following this, Wang and Zhang \cite{WZDCDS} generalize their results to the mass subcritical regime. Liu and Zhang \cite{LZAFA} studied the limiting behavior and local uniqueness of normalized solutions to \eqref{equ02} when $V(x)=|x|^{2}$, $N=2$ and $p=4$ by overcoming the non-invariance under translations of the harmonic potential.

	Following a literature review, the existence, concentration behavior and local uniqueness of normalized solutions for non-autonomous fractional Schr\"{o}dinger-Poisson systems have received little attention in existing literature. Even for the existing results on classical elliptic equations, they are mainly obtained under the condition of unbounded harmonic potentials.
	{\bf Therefore, it is nature to ask whether the normalized solutions of \eqref{equ1} admit similar limit properties when potential function is bounded and lack of homogeneity.} Using the Lax-Milgram theorem, the fractional Schr\"{o}dinger-Poisson system \eqref{equ1} can be transformed into a single fractional Schr\"{o}dinger equation with a nonlocal term. Indeed, given any fixed $u\in H^{s}(\mathbb{R}^{3})$, the Poisson equation $(-\Delta)^{s}\phi=u^{2}$ possesses a unique weak solution $\phi_{u}\in \mathcal{D}^{s, 2}(\mathbb{R}^{3})$ and $\phi_{u}$ can be expressed as
	\begin{equation}\label{equ03}
		\phi_{u}(x)=c_{s}\int_{\mathbb{R}^{3}}\frac{|u(y)|^{2}}{|x-y|^{3-2s}}dy\quad \hbox{with}\quad c_{s}=\pi^{-\frac{3}{2}}2^{-2s}\frac{\Gamma\left( \frac{3-2s}{2}\right) }{\Gamma(s)}.
	\end{equation}
	In the sequel, we often omit the constant $c_{s}$ for convenience. Substituting \eqref{equ03} into \eqref{equ1}, it reduces to a single fractional Schr\"{o}dinger equation
	\begin{equation}\label{equ04}
		(-\Delta)^{s}u+V(x)u=\mu u+\phi_{u} u+a^{p-2}|u|^{p-2}u\quad \hbox{in $\mathbb{R}^{3}$},
	\end{equation}
	where $s\in (\frac{3}{4}, 1)$, $a>0$ and $2<p<2+\frac{4s}{3}$. The main purpose of this paper is to focus on the existence and asymptotic behavior of normalized solutions to problem \eqref{equ04}. Concretely speaking, for a properly chosen Lagrange multiplier $\mu$, the ground states of \eqref{equ04} are equivalently the minimizers of the following minimization problem
	\begin{equation}\label{equ2}
		e_{m}(a):=\inf_{u\in S_{m}}E_{a}(u) \quad \hbox{$a>0$},
	\end{equation}
	where $S_{m}:=\{u\in H^{s}(\mathbb{R}^{3}): \|u\|_{2}^{2}=m\}$ and  energy functional $E_{a}(u)$ is given by
	\begin{equation}\label{equ3}
		E_{a}(u):=\int_{\mathbb{R}^{3}}(|(-\Delta)^{\frac{s}{2}} u|^{2}+V(x)|u|^{2})dx-\frac{1}{2}\int_{\mathbb{R}^{3}}\int_{\mathbb{R}^{3}}\frac{|u(x)|^{2}|u(y)|^{2}}{|x-y|^{3-2s}}dxdy-\frac{2a^{p-2}}{p}\int_{\mathbb{R}^{3}}|u|^{p}dx.
	\end{equation}	
Now we summarize the first main result of this paper. Assume that\\

$\mathbf{(V_1)}$\quad $V(x)\in L^{\infty}(\mathbb{R}^{3})\cap C^{\alpha}(\mathbb{R}^{3})$ with $\alpha\in (0, 1)$, and $V_{\infty}:=\lim\limits_{|x|\rightarrow\infty}V(x)=\sup\limits_{x \in \mathbb{R}^3}V(x)$.
	\begin{theorem}\label{the1}
		Suppose $V(x)$ satisfies that $\mathbf{(V_1)}$,
	then $e_{m}(a)$ admits at least one minimiser for any $m>0$.
	\end{theorem}
	\begin{remark}
	\rm $(1)$ In order to prove Theorem \ref{the1}, we first recall a fractional Gagliardo-Nirenberg inequality (cf. \cite[Lemma 2.1]{H})
	\begin{equation}\label{equ4}
		\int_{\mathbb{R}^{3}}|u|^{p}dx\leq C_{opt}\left(\int_{\mathbb{R}^{3}}|(-\Delta)^{\frac{s}{2}} u|^{2}dx\right)^{\frac{3(p-2)}{4s}}\left(\int_{\mathbb{R}^{3}}|u|^{2}dx\right)^{\frac{p}{2}-\frac{3p-6}{4s}}
	\end{equation}
	with
	\begin{equation*}
		C_{opt}=\frac{2sp}{6-p(3-2s)}\left( \frac{6-p(3-2s)}{3p-6}\right)^{\frac{3(p-2)}{4s}} \frac{1}{\|Q\|_{2}^{p-2}},
	\end{equation*}
	where $Q(x)=Q(|x|)$ is the unique (up to translations) positive solution of the following nonlinear equation
	\begin{equation}\label{equ5}
		(-\Delta)^{s}u+u-u^{p-1}=0 \quad \hbox{in $\mathbb{R}^3$}.
	\end{equation}	
	$(2)$ If $u$ is a minimizer of \eqref{equ2}, then we can assume that $u$ is nonnegative, due to the fact that $E_{a}(u)\geq E_{a}(|u|)$. Consequently, without loss of generality, we can restrict the minimization to nonnegative functions. Unlike the fractional Schr\"{o}dinger-Poisson system studied in \cite{MHCV}, the potential $V(x)$ in problem \eqref{equ04} seems to render the existence of normalized solutions more intricate. In contrast to the coercive potentials considered in \cite{D,GLMP,GSIM,H}, which ensure the validity of the embedding, this property, combined with constrained variational methods, directly yields the existence of normalized solutions. Conversely, for general bounded potentials, the compact embedding property is not guaranteed, and thus the concentration-compactness lemma must be invoked in this scenario.
\end{remark}
	
	We next focus on the concentration behaviour of nonnegative
	minimizers for \eqref{equ2} as $a\rightarrow\infty$. For simplicity, we assume $m=1$. Inspired by \cite{GSIM,LZ}, we assume that \\
	
	$\mathbf{(V_2)}$\quad $V(x)$ is almost homogeneous of some degree $r_{i}>0$ around $x_{i}$, where $x_{i}$ is the global minimum points of $V(x)$ satisfying
	\begin{equation*}
		Z:=\left\lbrace x\in\mathbb{R}^{3}: V(x)=0\right\rbrace=\left\lbrace x_{1}, x_{2},...,x_{n}\right\rbrace \quad \hbox{$n\geq1$},
	\end{equation*}
	and there exist $V_{i}(x)\in C_{loc}^{2}(\mathbb{R}^{3})$ is homogeneous of degree $r_{i}>0$, such that
	\begin{equation}\label{equ9}
		\frac{V(x+x_{i})}{V_{i}(x)}\rightarrow1 \quad \hbox{as $|x|\rightarrow0$ for $i=1, 2,...,n $}.
	\end{equation}
Moreover, we define
	\begin{equation}\label{equ10}
		H_{i}(y):=	\int_{\mathbb{R}^{3}}V_{i}(x+y)Q^{2}(x)dx\quad \hbox{$i=1, 2,...,n$}.
	\end{equation}
	Let
	\begin{equation}\label{equ11}
		r:=\max_{1\leq i\leq m}r_{i}, \quad  \overline{Z}:=\left\lbrace x_{i}\in Z: r_{i}=r\right\rbrace \subset Z,
	\end{equation}
	and
	\begin{equation}\label{equ12}
		\bar{\lambda}_{0}:=\min_{i\in \Lambda}\bar{\lambda}_{i} \quad \mathrm {where} \quad \bar{\lambda}_{i}:=\min_{y\in\mathbb{R}^{3}}H_{i}(y)\quad \mathrm {and} \quad\Lambda:=\left\lbrace i: x_{i}\in\overline{Z}\right\rbrace.
	\end{equation}
	To further streamline the presentation of our subsequent results, we now introduce some useful notation.
	\begin{equation}\label{equ13}
		H(y):=	\int_{\mathbb{R}^{3}}V_{0}(x+y)Q^{2}(x)dx\quad \mathrm {where} \quad V_{0}(x):=V_{i}(x) \quad \mathrm {and }~~i~~\mathrm {satisfying }~~ \bar{\lambda}_{i}=\bar{\lambda}_{0},
	\end{equation}
	and
	\begin{equation}\label{equ14}
		Z_{0}:=\left\lbrace x_{i}\in\overline{Z}: \bar{\lambda}_{i}=\bar{\lambda}_{0}\right\rbrace,\quad K_{0}:=\left\lbrace y: H(y)= \bar{\lambda}_{i}=\bar{\lambda}_{0}\right\rbrace.
	\end{equation}
	The result of the concentration of minimizers as $a\rightarrow\infty$ is stated as follows. 
	
	\begin{theorem}\label{the2}
		Assume that $V(x)$ satisfies $\mathbf{(V_1)}$ and $\mathbf{(V_2)}$, let $u_{k}$ be a nonnegative minimizer of $e_{1}(a_{k})$ with $a_{k}\rightarrow\infty$ as $k\rightarrow\infty$. Then there exists a subsequence of $\{u_{k}\}$, still denoted by $\{u_{k}\}$, such that $u_{k}$ satisfies
		\begin{equation}\label{equ15}
			\lim_{k\rightarrow\infty}\varepsilon_{k}^{\frac{3}{2}}u_{k}(\varepsilon_{k}x+x_{k})=\frac{Q(x)}{\sqrt{a^{\ast}}}\quad \hbox{in $H^{s}(\mathbb{R}^{3})\cap L^{\infty}(\mathbb{R}^{3})$},
		\end{equation}
		where $\varepsilon_{k}:=\left(\frac{a_{k}}{\sqrt{a^{\ast}}}\right)^{\frac{2p-4}{3p-6-4s}}$, $a^{\ast}:=\|Q\|_{2}^{2}$ and $Q$ is the unique positive solution of \eqref{equ5}. Moreover, $x_{k}$ is the unique global maximal point of $u_{k}$ satisfying
		\begin{equation}\label{equ16}
			\lim\limits_{k\rightarrow \infty}x_{k}=x_{0}\quad \hbox{with $V(x_{0})=0$},
		\end{equation}
		and
		\begin{equation}\label{equ17}
			\lim\limits_{k\rightarrow \infty}\frac{x_{k}-x_{0}}{\varepsilon_{k}}=y_{0} \quad \hbox{for some $x_{0}\in Z_{0}$ and $y_{0}\in K_{0}$}.
		\end{equation}
	\end{theorem}
	
	We now proceed to address the local uniqueness of minimizers for $e_{1}(a)$ under the following non-degenerate condition on the potential $V(x)$. Assume that\\
	
	$\mathbf{(V_3)}$\quad $Z_{0}$ contains only one element and $y_0$ is the unique and non-degenerate critical point of $K_{0}$, where $Z_{0}$ and $K_{0}$ are given in \eqref{equ14}.

	\begin{theorem}\label{the3}
		Assume that $V(x)$ satisfies $\mathbf{(V_1)}$-$\mathbf{(V_3)}$, then there exists a unique nonnegative minimizer of $e_{1}(a_{k})$ as $k\rightarrow\infty$.
	\end{theorem}
	\begin{remark}
	\rm 	 Compared with the study of the concentration and uniqueness of normalized solutions in classical elliptic problems \cite{GSIM,GJDE,LZ,WZDCDS}, the fractional Laplacian operator introduces some additional difficulties, which require the development of new analytical methods.\\
	$(1)$ Compared with our previous work \cite{LIUYANG}, we obtain the asymptotic behaviors and local uniqueness without conditions of homogeneity and unboundedness of the potential, while the potential is $|x|^2$ in \cite{LIUYANG}. Our results relax the global condition on the potential function to the local condition, which only depends on the asymptotic behaviors of the potential function near its minimum point. Similar enhancements have also been widely applied in the study of classical elliptic equations, we refer to \cite{DWY,LiWeiYang}. This improvement enables our results to be applied more widely.\\ 
	$(2)$ The proof of Theorem \ref{the2} follows from optimal energy estimates 
	\begin{equation*}
		\lim_{a\rightarrow\infty}\frac{e_{1}(a)}{\left(\frac{a}{\sqrt{a^{\ast}}}\right)^{\frac{4s(p-2)}{4s+6-3p}}}=\frac{3p-6-4s}{6-(3-2s)p},
	\end{equation*}
	where $a^{\ast}:=\|Q\|_{2}^{2}$ and $Q>0$ is the unique positive solution of \eqref{equ5}. However, when obtaining a lower bound estimate for $e_{1}(a)$, the $L^{2}$-subcritical nonlinearity term renders the fractional Gagliardo-Nirenberg inequality inapplicable directly. To address this, we have to employ the fact that
	\begin{equation*}
		\frac{e_{1}(a)}{\left(\frac{a}{\sqrt{a^{\ast}}}\right)^{\frac{4s(p-2)}{4s+6-3p}}}\geq\tilde{e}_{1}(\sqrt{a^{\ast}})+o(1)\quad \hbox{as $a\rightarrow\infty$},
	\end{equation*} 
	where $\tilde{e}_{1}(\sqrt{a^{\ast}})$ is a new auxiliary minimization problem defined in \eqref{equ3-2-1}, see Lemma \ref{lem3-1}.\\
	$(3)$ The inapplicability of classical local elliptic estimates leads to a significant increase in the difficulty of deriving the decay estimate for $|\nabla\hat{u}_{ik}(x)|$, where $\hat{u}_{ik}$ be as in \eqref{equ4-4}. Inspired by \cite[Theorem 12.2.4]{C}, we obtain the upper bound estimate for $\|\hat{u}_{ik}(y)\|_{C^{2s}(B_{1}(x))}$ by combining the $L^{\infty}$-estimate and the nonlocal term estimate, thereby deriving the decay property of $|\nabla\hat{u}_{ik}(x)|$, see Lemma \ref{lem5-1}. Moreover, the decay estimate of $\hat{\eta}_{k}$ serves as key information for establishing the local uniqueness of the constrained minimizers via contradiction, where $\hat{\eta}_{k}$ is defined in \eqref{equ4-5}. However, due to the complexity of the equation satisfied by $\hat{\eta}_{k}$, classical methods for handling fractional elliptic equations, such as Nash-Moser methods and Bessel kernel properties, are no longer applicable. To address this challenge, we have to develop a new comparison principle, see Lemma \ref{lem4-2} and Lemma \ref{lem5-2}.\\
	$(4)$ The Pohozaev-type identity plays an important role in proving the local uniqueness of constrained minimizers. Nevertheless, the non-locality of the fractional Laplacian operator makes it impossible to apply integration by parts and other conventional methods directly. To overcome this difficulty, we employ the extension method proposed by Caffarelli and Silvestre\cite{CC}, which transforms the nonlocal problem into a local one in higher dimensions. Through precise integral estimates, we establish the Pohozaev-type identity, see Lemma \ref{lem4-3}, Lemma \ref{lem5-3} and Lemma \ref{lem5-4}.
		\end{remark}

	This paper is organized as follows. Section \ref{sec 2} focuses on proving Theorem \ref{the1}, which is concerned with the existence of minimisers for $e_{m}(a)$. In Section \ref{sec 3}, we shall prove Theorem \ref{the2} on the refined limiting profiles of nonnegative minimizers	for $e_{1}(a)$ as $a\rightarrow\infty$.  The proof of Theorem \ref{the3} is given in Section \ref{sec 4}, which centers on the local uniqueness of nonnegative minimizers. Finally, Appendix A contains the detailed proofs of some results used in the proof of Theorem \ref{the3}.

	\section{Existence of minimizers}\label{sec 2}
	In this section, we mainly prove Theorem \ref{the1} on the existence of minimizers for $e_{m}(a)$. As shown in \cite[Theorem 1.1]{D}, $Q(x)$ exhibits the following decay behavior
	\begin{equation}\label{equ6}
		\frac{C_{1}}{1+|x|^{3+2s}}\leq Q(x)\leq\frac{C_{2}}{1+|x|^{3+2s}}\quad \mathrm {and} \quad	|\nabla Q(x)|\leq\frac{C_{3}}{1+|x|^{3+2s}}.
	\end{equation}
	where $Q$ is the unique positive solution of \eqref{equ5} and $C_{1}$, $C_{2}$ and $C_{3}$ are positive constants. Using  Pohozaev identity, one knows that
	\begin{equation}\label{equ7}
		\int_{\mathbb{R}^{3}}|(-\Delta)^{\frac{s}{2}}Q|^{2}dx=\frac{3(p-2)}{2sp}\int_{\mathbb{R}^{3}}|Q|^{p}dx=\frac{3(p-2)}{6-(3-2s)p}\int_{\mathbb{R}^{3}}|Q|^{2}dx.
	\end{equation}
	Recall the Hardy–Littewood–Sobolev inequality (cf. \cite[Theorem 4.3]{L}), one has
	\begin{equation}\label{equ8}
		\int_{\mathbb{R}^{3}}\int_{\mathbb{R}^{3}}\frac{|u(x)||v(y)|}{|x-y|^{3-2s}}dxdy\leq C\|u\|_{\frac{6}{3+2s}}\|v\|_{\frac{6}{3+2s}}\quad \hbox{ $u, v\in L^{\frac{6}{3+2s}}(\mathbb{R}^{3})$}.
	\end{equation}
We now consider the existence of constrained minimizers for $e_{m}(a)$ when $V(x)\equiv constant$. Without loss of generality, we may assume that $V(x)\equiv0$. For this purpose, we introduce the following constrained minimization problem
	\begin{equation*}
		e^{0}_{m}(a):=\inf_{u\in S_{m}}E^{0}_{a}(u),
	\end{equation*}
	where energy functional $E^{0}_{a}(u)$ is given by
	\begin{equation*}
		E^{0}_{a}(u):=\int_{\mathbb{R}^{3}}|(-\Delta)^{\frac{s}{2}} u|^{2}dx-\frac{1}{2}\int_{\mathbb{R}^{3}}\int_{\mathbb{R}^{3}}\frac{|u(x)|^{2}|u(y)|^{2}}{|x-y|^{3-2s}}dxdy-\frac{2a^{p-2}}{p}\int_{\mathbb{R}^{3}}|u|^{p}dx.
	\end{equation*}

	\begin{lemma}\label{lem2-1}
		The map $m\mapsto e^{0}_{m}(a)$ is a continuous function on $\left( 0, \infty\right) $.
	\end{lemma}
	\begin{proof}
		Let $u\in H^{s}(\mathbb{R}^{3})$ with $\|u\|_{2}^{2}=m>0$, applying \eqref{equ4} and \eqref{equ8}, one can see that
		\begin{equation}\label{equ2-1}
			\int_{\mathbb{R}^{3}}\int_{\mathbb{R}^{3}}\frac{u^{2}(x)u^{2}(y)}{|x-y|^{3-2s}}dxdy\leq C\left(\int_{\mathbb{R}^{3}}|u|^{\frac{12}{3+2s}}dx\right)^{\frac{3+2s}{3}}\leq C\left(\int_{\mathbb{R}^{3}}|(-\Delta)^{\frac{s}{2}} u|^{2}dx\right)^{\frac{6-4s}{4s}}.
		\end{equation} 
		Using the Young’s inequality, we deduce from \eqref{equ4} and \eqref{equ2-1} that for any $2<p<2+\frac{4s}{3}$
		\begin{equation}\label{equ2-2}
			\begin{split}
				E^{0}_{a}(u)&\geq\int_{\mathbb{R}^{3}}|(-\Delta)^{\frac{s}{2}} u|^{2}dx-C\left(\int_{\mathbb{R}^{3}}|(-\Delta)^{\frac{s}{2}} u|^{2}dx\right)^{\frac{6-4s}{4s}}-Ca^{p-2}\left(\int_{\mathbb{R}^{3}}|(-\Delta)^{\frac{s}{2}} u|^{2}dx\right)^{\frac{3p-6}{4s}}\\
				&\geq\frac{1}{2} \int_{\mathbb{R}^{3}}|(-\Delta)^{\frac{s}{2}} u|^{2}dx-C_{a},
			\end{split}
		\end{equation}
		which shows that $E^{0}_{a}(u)$ is bounded from below.
		
		Now consider $m\in (0, \infty)$ and a sequence $\left\lbrace m_{n}\right\rbrace \subset\left( 0, \infty\right)$ such that $\lim\limits_{n\rightarrow\infty}m_{n}=m$. There exists a sequence $\left\lbrace u_{n} \right\rbrace \subset S_{m}$ such that $\lim\limits_{n\rightarrow\infty}E^{0}_{a}(u_{n})=e^{0}_{m}(a)$. Moreover, \eqref{equ2-2} shows that $\left\lbrace u_{n}\right\rbrace $ is uniformly bounded in $H^{s}(\mathbb{R}^{3})$. Let $w_{n}:=\sqrt{\frac{m_{n}}{m}}u_{n}$, one has $w_{n}\in S_{m_{n}}$ and
		\begin{equation*}
			\begin{split}
				E^{0}_{a}(w_{n})=&\frac{m_{n}}{m}E^{0}_{a}(u_{n})+\frac{1}{2}\frac{m_{n}}{m}\left( 1-\frac{m_{n}}{m}\right) \int_{\mathbb{R}^{3}}\int_{\mathbb{R}^{3}}\frac{|u_{n}(x)|^{2}|u_{n}(y)|^{2}}{|x-y|^{3-2s}}dxdy\\
				&-\frac{2a^{p-2}}{p}\frac{m_{n}}{m}\left( 1-\left( \frac{m_{n}}{m}\right) ^{\frac{p-2}{2}}\right)\int_{\mathbb{R}^{3}}|u_{n}|^{p}dx.
			\end{split}
		\end{equation*} 
		Them we have
		\begin{equation}\label{equ2-3}
			\begin{split}
				\left| E^{0}_{a}(w_{n})-E^{0}_{a}(u_{n})\right| \leq&\left| \frac{m_{n}}{m}-1\right| \left| E^{0}_{a}(u_{n})\right| +C\left|1- \frac{m_{n}}{m}\right|\left(\int_{\mathbb{R}^{3}}|u_{n}|^{\frac{12}{3+2s}}dx\right)^{\frac{3+2s}{3}}\\
				&+C\left|  1-\left( \frac{m_{n}}{m}\right) ^{\frac{p-2}{2}}\right| \int_{\mathbb{R}^{3}}|u_{n}|^{p}dx\\
				\leq&C\left| \frac{m_{n}}{m}-1\right|+C\left| \left( \frac{m_{n}}{m}\right) ^{\frac{p-2}{2}}-1\right|.
			\end{split}
		\end{equation}
		Thus
		\begin{equation}\label{equ2-4}
			e^{0}_{m_{n}}(a)\leq E^{0}_{a}(w_{n})\leq E^{0}_{a}(u_{n})+C\left| \frac{m_{n}}{m}-1\right|+C\left| \left( \frac{m_{n}}{m}\right) ^{\frac{p-2}{2}}-1\right|,
		\end{equation}
		which implies that
		\begin{equation}\label{equ2-5}
			\limsup_{n\rightarrow\infty}e^{0}_{m_{n}}(a)\leq\lim_{n\rightarrow\infty}E^{0}_{a}(u_{n})=e^{0}_{m}(a).
		\end{equation} 
		Further, for all $n$, there exists $v_{n}\in S_{m_{n}}$ such that
		\begin{equation*}
			e^{0}_{m_{n}}(a)\leq E^{0}_{a}(v_{n})\leq e^{0}_{m_{n}}(a)+\frac{1}{n},
		\end{equation*} 
		from which and \eqref{equ2-4} we deduce that $E^{0}_{a}(v_{n})$ is uniformly bounded. Then \eqref{equ2-2} shows that $\left\lbrace v_{n}\right\rbrace $ is uniformly bounded in $H^{s}(\mathbb{R}^{3})$. Let $h_{n}:=\sqrt{\frac{m}{m_{n}}}v_{n}$, similar to \eqref{equ2-3}, one derive that
		\begin{equation*}
			\left| E^{0}_{a}(h_{n})-E^{0}_{a}(v_{n})\right| \leq C\left| \frac{m}{m_{n}}-1\right|+C\left| \left( \frac{m}{m_{n}}\right) ^{\frac{p-2}{2}}-1\right|.
		\end{equation*} 
		Thus we have
		\begin{equation*}
			\begin{aligned}
				e^{0}_{m_{n}}(a)&\geq E^{0}_{a}(v_{n})-\frac{1}{n}\\
				&\geq E^{0}_{a}(h_{n})-C\left| \frac{m}{m_{n}}-1\right|-C\left| \left( \frac{m}{m_{n}}\right) ^{\frac{p-2}{2}}-1\right|-\frac{1}{n}\\
				&\geq e^{0}_{m}(a)-C\left| \frac{m}{m_{n}}-1\right|-C\left| \left( \frac{m}{m_{n}}\right) ^{\frac{p-2}{2}}-1\right|-\frac{1}{n}
			\end{aligned}
		\end{equation*}
		which means that 
		\begin{equation*}
			\liminf_{n\rightarrow\infty}e^{0}_{m_{n}}(a)\geq e^{0}_{m}(a).
		\end{equation*}
		Combining with \eqref{equ2-5}, we obtain that $\lim\limits_{n\rightarrow\infty}e^{0}_{m_{n}}(a)= e^{0}_{m}(a)$
		as $m_{n}\rightarrow m$, this completes the proof of Lemma \ref{lem2-1}.
	\end{proof}
	
	Let $\left\lbrace u_{n}\right\rbrace $ be a minimizing sequence for $e^{0}_{m}(a)$. By virtue of the concentration compactness lemma from \cite[Lemma 2.4]{W} (see also \cite[Lemma 2.4]{F}), there exists a subsequence, still denoted by $\left\lbrace u_{n}\right\rbrace $ for simplicity, such that vanishing, dichotomy or compactness is satisfied for this sequence. We next prove two key claims to rule out the vanishing case and dichotomy case.
	
	\noindent
	{\bf  Claim I} The vanishing case does not occur. If not, one has $\lim\limits_{n\rightarrow\infty}\sup\limits_{y\in \mathbb{R}^{3}}\int_{B_{R}(y)}|u_{n}|^{2}dx=0$ for all $R>0$. Using the vanishing lemma in \cite{S}, it follows that
	\begin{equation*}
		\lim_{n\rightarrow\infty}\int_{\mathbb{R}^{3}}|u_{n}|^{r}dx=0\quad \hbox{for any $r\in (2, \frac{6}{3-2s})$},
	\end{equation*}
	from which and \eqref{equ8} we deduce that
	\begin{equation*}
		\int_{\mathbb{R}^{3}}\int_{\mathbb{R}^{3}}\frac{|u_{n}(x)|^{2}|u_{n}(y)|^{2}}{|x-y|^{3-2s}}dxdy\leq C\left(\int_{\mathbb{R}^{3}}|u_{n}|^{\frac{12}{3+2s}}dx\right)^{\frac{3+2s}{3}}\rightarrow0\quad \hbox{as $n\rightarrow\infty$}.
	\end{equation*}
	Thus we have
	\begin{equation}\label{equ2-6}
		e^{0}_{m}(a)=\lim _{n\rightarrow\infty}E^{0}_{a}(u_{n})=\lim _{n\rightarrow\infty}\int_{\mathbb{R}^{3}}|(-\Delta)^{\frac{s}{2}} u_{n}|^{2}dx\geq0.
	\end{equation} 
	Now we set $u_{t}(x)=\frac{t^{\frac{3}{2}}\sqrt{m}}{\|Q\|_{2}}Q(tx)$ with $t>0$ is sufficiently small. It then follows from \eqref{equ7} that
	\begin{equation}\label{equ2-7}
		\begin{split}
			e^{0}_{m}(a)\leq E^{0}_{a}(u_{t})&\leq\frac{(3p-6)mt^{2s}}{6-(3-2s)p}-\frac{4spm^{\frac{p}{2}}}{6-(3-2s)p}\frac{t^{\frac{3p-6}{2}}}{\|Q\|_{2}^{p-2}}\\
			&<0\quad \hbox{as $t$ is sufficiently small},
		\end{split}
	\end{equation}
	where we have used the fact that $\frac{3p-6}{2}<2s$. This contradicts \eqref{equ2-6}, and the vanishing case is ruled out.
	
	\noindent
	{\bf  Claim II} The dichotomy case does not occur. For any $\lambda\in (0, m)$ and $\theta>1$, we define
	\begin{equation*}
		e^{0}_{\lambda}(a):=\inf_{u\in S_{\lambda}}E^{0}_{a}(u),
	\end{equation*}
	and
	\begin{equation*}
		e^{0}_{\theta\lambda}(a):=\inf_{u\in S_{\theta\lambda}}E^{0}_{a}(u)=\inf_{v\in S_{\lambda}}E^{0}_{a}(\sqrt{\theta}v).
	\end{equation*}
	Similar to \eqref{equ2-1}-\eqref{equ2-2}, one can derive that there exists $\{v_{n}\}\subset H^{s}(\mathbb{R}^{3})$ such that
	\begin{equation*}
		e^{0}_{\theta\lambda}(a)=\lim _{n\rightarrow\infty}E^{0}_{a}(\sqrt{\theta}v_{n}),\quad  e^{0}_{\lambda}(a)=\lim _{n\rightarrow\infty}E^{0}_{a}(v_{n})\quad \hbox{and}\quad \|v_{n}\|_{2}^{2}=\lambda.
	\end{equation*}
	We prove that there exists some $\delta>0$ such that
	\begin{equation*}
		\liminf_{n\rightarrow\infty}\int_{\mathbb{R}^{3}}\int_{\mathbb{R}^{3}}\frac{|v_{n}(x)|^{2}|v_{n}(y)|^{2}}{|x-y|^{3-2s}}dxdy\geq \delta>0\quad \hbox{or}\quad	\liminf_{n\rightarrow\infty}\int_{\mathbb{R}^{3}}|v_{n}|^{p}dx\geq \delta>0.
	\end{equation*}
	If not, one has
	\begin{equation*}
		e^{0}_{\lambda}(a)=\lim _{n\rightarrow\infty}E^{0}_{a}(v_{n})=\lim _{n\rightarrow\infty}\int_{\mathbb{R}^{3}}|(-\Delta)^{\frac{s}{2}} v_{n}|^{2}dx\geq0,
	\end{equation*}
	which contradicts \eqref{equ2-7}. Thus we have
	\begin{equation}\label{equ2-8}
		\begin{split}
			&e^{0}_{\theta\lambda}(a)=\lim_{n\rightarrow\infty}E^{0}_{a}(\sqrt{\theta}v_{n})\\
			=&\lim_{n\rightarrow\infty}\left( \theta\int_{\mathbb{R}^{3}}|(-\Delta)^{\frac{s}{2}} v_{n}|^{2}dx-\frac{\theta^{2}}{2}\int_{\mathbb{R}^{3}}\int_{\mathbb{R}^{3}}\frac{|v_{n}(x)|^{2}|v_{n}(y)|^{2}}{|x-y|^{3-2s}}dxdy-\frac{2a^{p-2}}{p}\theta^{\frac{p}{2}}\int_{\mathbb{R}^{3}}|v_{n}|^{p}dx\right) \\
			=&\theta e^{0}_{\lambda}(a)+\frac{\theta(1-\theta)}{2}\lim_{n\rightarrow\infty}\int_{\mathbb{R}^{3}}\int_{\mathbb{R}^{3}}\frac{|v_{n}(x)|^{2}|v_{n}(y)|^{2}}{|x-y|^{3-2s}}dxdy+\frac{2a^{p-2}}{p}\theta(1-\theta^{\frac{p-2}{2}})\lim_{n\rightarrow\infty}\int_{\mathbb{R}^{3}}|v_{n}|^{p}dx\\
			<&\theta e^{0}_{\lambda}(a).
		\end{split}
	\end{equation}
	It then follows from \eqref{equ2-8} that
	\begin{equation}\label{equ2-9}
		e^{0}_{m}(a)=\frac{m-\lambda}{m}e^{0}_{\frac{m}{m-\lambda}(m-\lambda)}(a)+\frac{\lambda}{m}e^{0}_{\frac{m}{\lambda}\lambda}(a)<e^{0}_{m-\lambda}(a)+e^{0}_{\lambda}(a).
	\end{equation}
	Suppose the dichotomy case occurs. Then the concentration compactness lemma (cf. \cite[Lemma 2.4]{W}) shows that there exist some $\alpha\in (0, m)$ and two bounded sequences $u_{1n}, u_{2n}\in H^{s}(\mathbb{R}^{3})$ such that for any $\varepsilon>0$, there holds that
	\begin{equation*}
		\left|\int_{\mathbb{R}^{3}}\left| u_{1n}\right|^{2}dx-\alpha \right|\leq\varepsilon, \quad \left|\int_{\mathbb{R}^{3}}\left| u_{2n}\right|^{2}dx-\left(m-\alpha  \right) \right|\leq\varepsilon,
	\end{equation*}
	\begin{equation*}
		\|u_{n}-u_{1n}-u_{2n}\|_{r}\rightarrow0\quad \hbox{for $r\in[2, 2_{s}^{\ast})$}.
	\end{equation*}
	\begin{equation*}
		\liminf_{n\rightarrow\infty} \int_{\mathbb{R}^{3}}\left( |(-\Delta)^{\frac{s}{2}} u_{n}|^{2}-|(-\Delta)^{\frac{s}{2}} u_{1n}|^{2}-|(-\Delta)^{\frac{s}{2}} u_{2n}|^{2}\right)dx\geq0.
	\end{equation*}
	Furthermore, there exist $\left\lbrace y_{n}\right\rbrace \subset\mathbb{R}^{N}$ and $\left\lbrace R_{n}\right\rbrace \subset(0, \infty)$ with $\lim\limits_{n\rightarrow\infty}R_{n}=\infty$, such that
	
	\begin{equation*}
		\begin{cases}
			u_{1n}=u_{n}
			\quad &\text{if } |x-y_{n}|\leq R_{0},\\
			u_{1n}\leq|u_{n}|&\text{if } R_{0}\leq|x-y_{n}|\leq 2R_{0},\\
			u_{1n}=0&\text{if } |x-y_{n}|\geq 2R_{0},
		\end{cases}\quad 
		\begin{cases}
			u_{2n}=0
			\quad &\text{if } |x-y_{n}|\leq R_{n},\\
			u_{2n}\leq|u_{n}|&\text{if } R_{n}\leq|x-y_{n}|\leq 2R_{n},\\
			u_{2n}=u_{n}&\text{if } |x-y_{n}|\geq 2R_{n}.
		\end{cases}
	\end{equation*} 
	Some calculations yield
	\begin{equation*}
		\int_{\mathbb{R}^{3}}\int_{\mathbb{R}^{3}} \frac{|u_{1n}(x)|^{2}|u_{2n}(y)|^{2}}{|x-y|^{3-2s}}dxdy\leq\frac{1}{2^{3-2s}(R_{n}-R_{0})^{3-2s}}\| u_{1n}\|_{2}^{2}\| u_{2n}\|_{2}^{2}\rightarrow0\quad \hbox{as $n\rightarrow\infty$},
	\end{equation*}
	where we have used the fact that $|x-y_{n}|\leq2R_{0}$ for $x\in {\rm supp}~u_{1n}$, and $|y-y_{n}|\geq2R_{n}$ for $y\in {\rm supp}~u_{2n}$. Similarly, one can see that
	\begin{equation*}
		\int_{\mathbb{R}^{3}}\int_{\mathbb{R}^{3}} \frac{|u_{1n}(x)|^{2}u_{1n}(y)u_{2n}(y)}{|x-y|^{3-2s}}dxdy\rightarrow0\quad \hbox{as $n\rightarrow\infty$}.
	\end{equation*}
	Then we have
	\begin{equation*}
		\begin{split}
			&	\int_{\mathbb{R}^{3}}\int_{\mathbb{R}^{3}}\left( \frac{|u_{n}(x)|^{2}|u_{n}(y)|^{2}}{|x-y|^{3-2s}}-\frac{|u_{1n}(x)|^{2}|u_{1n}(y)|^{2}}{|x-y|^{3-2s}}-\frac{|u_{2n}(x)|^{2}|u_{2n}(y)|^{2}}{|x-y|^{3-2s}}\right) dxdy\\
			=&\int_{\mathbb{R}^{3}}\int_{\mathbb{R}^{3}}\left( \frac{|u_{n}(x)|^{2}|u_{n}(y)|^{2}}{|x-y|^{3-2s}}-\frac{|u_{1n}(x)+u_{2n}(x)|^{2}|u_{1n}(y)+u_{2n}(y)|^{2}}{|x-y|^{3-2s}}\right) dxdy+o(1)\\
			=&\int_{\mathbb{R}^{3}}\int_{\mathbb{R}^{3}} \frac{|u_{n}(y)|^{2}\left(|u_{n}(x)|^{2}-|u_{1n}(x)+u_{2n}(x)|^{2} \right) }{|x-y|^{3-2s}}dxdy\\
			&+\int_{\mathbb{R}^{3}}\int_{\mathbb{R}^{3}} \frac{|u_{1n}(x)+u_{2n}(x)|^{2}\left(|u_{n}(y)|^{2}-|u_{1n}(y)+u_{2n}(y)|^{2} \right) }{|x-y|^{3-2s}}dxdy+o(1).
		\end{split}
	\end{equation*}
	Applying the H\"{o}lder's inequality, one can deduce from \eqref{equ8} that
	\begin{equation*}
		\begin{split}
			&	\int_{\mathbb{R}^{3}}\int_{\mathbb{R}^{3}} \frac{|u_{n}(y)|^{2}\left(|u_{n}(x)|^{2}-|u_{1n}(x)+u_{2n}(x)|^{2} \right) }{|x-y|^{3-2s}}dxdy\\
			\leq&C \left(	\int_{\mathbb{R}^{3}}|u_{n}(y)|^{\frac{12}{3+2s}} dy\right)^{\frac{3+2s}{6}}\left( \int_{\mathbb{R}^{3}}\left| |u_{n}(x)|^{2}-|u_{1n}(x)+u_{2n}(x)|^{2}\right|^{\frac{6}{3+2s}}dx\right)^{\frac{3+2s}{6}}\\
			\leq &C\| u_{n}(x)+u_{1n}(x)+u_{2n}(x)\|_{\frac{12}{3+2s}}\| u_{n}(x)-u_{1n}(x)-u_{2n}(x)\|_{\frac{12}{3+2s}}\leq C\varepsilon,
		\end{split}
	\end{equation*}
	and
	\begin{equation*}
		\int_{\mathbb{R}^{3}}\int_{\mathbb{R}^{3}} \frac{|u_{1n}(x)+u_{2n}(x)|^{2}\left(|u_{n}(y)|^{2}-|u_{1n}(y)+u_{2n}(y)|^{2} \right) }{|x-y|^{3-2s}}dxdy\leq C\varepsilon.
	\end{equation*}
	Thus we obtain that
	\begin{equation}\label{equ2-10}
		\begin{split}
			&\lim_{n\rightarrow\infty}\int_{\mathbb{R}^{3}}\int_{\mathbb{R}^{3}}\frac{|u_{n}(x)|^{2}|u_{n}(y)|^{2}}{|x-y|^{3-2s}}dxdy\\
			\leq&\lim_{n\rightarrow\infty}\left( \int_{\mathbb{R}^{3}}\int_{\mathbb{R}^{3}}\frac{|u_{1n}(x)|^{2}|u_{1n}(y)|^{2}}{|x-y|^{3-2s}}dxdy+\int_{\mathbb{R}^{3}}\int_{\mathbb{R}^{3}}\frac{|u_{2n}(x)|^{2}|u_{2n}(y)|^{2}}{|x-y|^{3-2s}}dxdy\right) 
		\end{split}
	\end{equation}
	By the facts $\left|\left| u_{n}\right| ^{p} -\left|u_{1n}+u_{2n} \right| ^{p}\right| \leq C\left| u_{n}-(u_{1n}+u_{2n})\right| \left(\left| u_{n}\right| ^{p-1}+\left|u_{1n}+u_{2n} \right| ^{p-1} \right) $ and ${\rm supp}u_{1n}\cap{\rm supp}u_{2n}=\emptyset$, we have
	\begin{equation*}
		\begin{split}
			&\int_{\mathbb{R}^{3}}\left(\left| u_{n}\right| ^{p} -\left| u_{1n}\right| ^{p} -\left| u_{2n}\right| ^{p}  \right) 	dx=\int_{\mathbb{R}^{3}}\left( \left| u_{n}\right| ^{p}-\left| u_{1n}+u_{2n}\right| ^{p}\right) dx\\
			\leq &C\int_{\mathbb{R}^{3}}\left| u_{n}-(u_{1n}+u_{2n})\right| \left| \left| u_{n}\right| ^{p-1}+\left|u_{1n}+u_{2n} \right| ^{p-1} \right| dx\\
			\leq&C\| u_{n}-u_{1n}-u_{2n}\|_{2}\leq C\varepsilon,
		\end{split}
	\end{equation*}
	which shows that
	\begin{equation}\label{equ2-11}
		\begin{split}
			\lim_{n\rightarrow\infty}\int_{\mathbb{R}^{3}}\left| u_{n}\right| ^{p}dx\leq\lim_{n\rightarrow\infty}\left(\int_{\mathbb{R}^{3}}\left| u_{1n}\right| ^{p}dx+\int_{\mathbb{R}^{3}}\left| u_{2n}\right| ^{p}dx \right).
		\end{split}
	\end{equation}
	Using \eqref{equ2-10}, \eqref{equ2-11} and Lemma \ref{lem2-1}, one can check that
	\begin{equation*}
		\begin{split}
			e^{0}_{m}(a)=\lim _{n\rightarrow\infty}E^{0}_{a}(u_{n})\geq\lim _{n\rightarrow\infty}E^{0}_{a}(u_{1n})+\lim _{n\rightarrow\infty}E^{0}_{a}(u_{2n})\geq e^{0}_{m-\alpha}(a)+e^{0}_{\alpha}(a),
		\end{split}
	\end{equation*}
	which contradicts \eqref{equ2-9}. This  implies that the dichotomy case cannot occur.
	
	\begin{proposition}\label{lem2-2}
		For any $m>0$, then $e^{0}_{m}(a)$ admits at least one minimiser.
	\end{proposition}
	\begin{proof}
		In view of {\bf  Claim I} and {\bf  Claim II}, there exist a sequence $\left\lbrace x_{n}\right\rbrace \subset\mathbb{R}^{3}$ and $u_{0}\in H^{s}(\mathbb{R}^{3})$ such that
		\begin{equation*}
			\int_{B_{R_{\varepsilon}}(x_{n})}|u_{n}|^{2}dx\geq m-\varepsilon\quad \hbox{for all $n\in\mathbb{N}$}.
		\end{equation*}
		Define $v_{n}(x):=u_{n}(x+x_{n})$, then we have
		\begin{equation}\label{equ2-12}
			\int_{B_{R_{\varepsilon}}(0)}|v_{n}|^{2}dx\geq m-\varepsilon.
		\end{equation}
		Since $\left\lbrace v_{n}\right\rbrace $ is uniformly bounded in $H^{s}(\mathbb{R}^{3})$, there exists $v_{0}\in H^{s}(\mathbb{R}^{3})$ such that
		\begin{equation*}
			v_{n}\rightharpoonup v_{0} \quad  \hbox{in $H^{s}(\mathbb{R}^{3})$ and $v_{n}\rightarrow v_{0}$ in $L^{2}_{loc}(\mathbb{R}^{3})$ as $n\rightarrow\infty$}.
		\end{equation*}
		Thus \eqref{equ2-9} shows that $\| v_{0}\|_{2}^{2}\geq m-\varepsilon$. This means that $\| v_{0}\|_{2}^{2}=m$ and $v_{n}\rightarrow v_{0}$ in $L^{2}(\mathbb{R}^{3})$ as $n\rightarrow\infty$. Consequently, 
		\begin{equation*}
			v_{n}\rightarrow v_{0}\quad  \hbox{in $L^{r}(\mathbb{R}^{3})$ for $r\in \left[ 2, \frac{6}{3-2s}\right) $ as $n\rightarrow\infty$},
		\end{equation*}
		from which and \eqref{equ8} we deduce that
		\begin{equation*}
			\begin{split}
				&	\left| \int_{\mathbb{R}^{3}}\int_{\mathbb{R}^{3}}\frac{|v_{n}(x)|^{2}|v_{n}(y)|^{2}}{|x-y|^{3-2s}}dxdy-\int_{\mathbb{R}^{3}}\int_{\mathbb{R}^{3}}\frac{|v_{0}(x)|^{2}|v_{0}(y)|^{2}}{|x-y|^{3-2s}}dxdy\right| \\
				\leq&\left| \int_{\mathbb{R}^{3}}\int_{\mathbb{R}^{3}}\frac{|v_{n}(x)|^{2}\left||v_{n}(y)|^{2}-|v_{0}(y)|^{2} \right| }{|x-y|^{3-2s}}dxdy\right| +\left| \int_{\mathbb{R}^{3}}\int_{\mathbb{R}^{3}}\frac{|v_{0}(y)|^{2}\left||v_{n}(x)|^{2}-|v_{0}(x)|^{2} \right| }{|x-y|^{3-2s}}dxdy\right|\\
				\leq& \| v_{n}\|_{\frac{12}{3+2s}}^{2}\| v_{n}+v_{0}\|_{\frac{12}{3+2s}}\| v_{n}-v_{0}\|_{\frac{12}{3+2s}}+\| v_{0}\|_{\frac{12}{3+2s}}^{2}\| v_{n}+v_{0}\|_{\frac{12}{3+2s}}\| v_{n}-v_{0}\|_{\frac{12}{3+2s}}\\
				\leq& C\| v_{n}-v_{0}\|_{\frac{12}{3+2s}}\rightarrow0\quad  \hbox{ as $n\rightarrow\infty$}.
			\end{split}
		\end{equation*}
		Thus we get 
		\begin{equation*}
			e^{0}_{m}(a)=\lim _{n\rightarrow\infty}E^{0}_{a}(u_{n})=\lim _{n\rightarrow\infty}E^{0}_{a}(v_{n})\geq E^{0}_{a}(v_{0})\geq e^{0}_{m}(a),
		\end{equation*}
		which shows that $v_{0}$ is a minimiser of $e^{0}_{m}(a)$. This completes the proof of proposition \ref{lem2-2}.
	\end{proof}

	Now we are ready to prove Theorem \ref{the1}.
	
	\begin{proof}[Proof of Theorem \ref{the1}] We first introduce the following constrained minimization problem
		\begin{equation*}
			e^{\infty}_{m}(a):=\inf_{u\in S_{m}}E^{\infty}_{a}(u),
		\end{equation*}
		where energy functional $E^{\infty}_{a}(u)$ is given by
		\begin{equation*}
			E^{\infty}_{a}(u):=\int_{\mathbb{R}^{3}}\left( |(-\Delta)^{\frac{s}{2}} u|^{2}+V_{\infty}\left| u\right| ^{2}\right) dx-\frac{1}{2}\int_{\mathbb{R}^{3}}\int_{\mathbb{R}^{3}}\frac{|u(x)|^{2}|u(y)|^{2}}{|x-y|^{3-2s}}dxdy-\frac{2a^{p-2}}{p}\int_{\mathbb{R}^{3}}|u|^{p}dx.
		\end{equation*}	
		By Proposition \ref{lem2-2}, we know that $e^{\infty}_{m}(a)$ admits at least one minimiser for $m>0$. In fact, if $V(x)\equiv constant\neq0$, $\int_{\mathbb{R}^{3}}V(x)\left| u\right| ^{2}dx$ is directly replaceable by the constant $Cm$. Suppose $v_{m}^{\infty}$ is a minimiser of $e^{\infty}_{m}(a)$, one can see that
		\begin{equation}\label{equ2-13}
			e_{m}(a)\leq E_{a}(v_{m}^{\infty})=E^{\infty}_{a}(v_{m}^{\infty})+\int_{\mathbb{R}^{3}}\left( V(x)-V_{\infty}\right) \left| v_{m}^{\infty}\right| ^{2}dx<E^{\infty}_{a}(v_{m}^{\infty})=e^{\infty}_{m}(a).
		\end{equation}
		Similar to \eqref{equ2-1}-\eqref{equ2-2},  there exists a minimising sequence $\left\lbrace u_{n}\right\rbrace \subset S_{m}$ such that
		\begin{equation*}
			\| u_{n}\|_{H^{s}(\mathbb{R}^{3})}\leq C,	\quad  \lim _{n\rightarrow\infty}E_{a}(u_{n})=e_{m}(a)\quad \hbox{and}\quad \|u_{n}\|_{2}^{2}=m.
		\end{equation*}
		Applying the concentration compactness lemma, there exists a subsequence, still denoted by $\left\lbrace u_{n}\right\rbrace $ for simplicity, such that vanishing, dichotomy or compactness is satisfied for this sequence.
		
		Suppose the vanishing case occurs, for any $\varepsilon>0$, there exists $R>0$ large enough, such that
		\begin{equation*}
			\max _{|x|\geq R}\left| V(x)-V_{\infty}\right| <\varepsilon\quad \hbox{and}\quad \int_{B_{R}(0)}|u_{n}|^{2}dx<C\varepsilon\quad  \hbox{as $n\rightarrow\infty$}.
		\end{equation*}
		Hence we have
		\begin{equation*}
			\int_{\mathbb{R}^{3}}\left|  V(x)-V_{\infty}\right|  \left| u_{n}\right| ^{2}dx\leq C\int_{B_{R}(0)}|u_{n}|^{2}dx+\varepsilon\int_{B^{c}_{R}(0)}|u_{n}|^{2}dx<C\varepsilon\rightarrow0\quad  \hbox{as $n\rightarrow\infty$}.
		\end{equation*}
		This implies that
		\begin{equation*}
			e_{m}(a)=\lim _{n\rightarrow\infty}E_{a}(u_{n})=\lim _{n\rightarrow\infty}E^{\infty}_{a}(u_{n})\geq e^{\infty}_{m}(a),
		\end{equation*}
		which contradicts \eqref{equ2-13}, and the vanishing case does not occur.
		
		Now we claim that there exists some $\delta>0$ such that
		\begin{equation*}
			\liminf_{n\rightarrow\infty}\int_{\mathbb{R}^{3}}|u_{n}|^{p}dx\geq \delta>0.
		\end{equation*}
		Indeed, up to a subsequence of $\left\lbrace u_{n}\right\rbrace $, if $\int_{\mathbb{R}^{3}}|u_{n}|^{p}dx\rightarrow0$ as $n\rightarrow\infty$, then there exists $R>0$ large enough, such that
		\begin{equation*}
			\begin{split}
				\int_{\mathbb{R}^{3}}\left|  V(x)-V_{\infty}\right|  \left| u_{n}\right| ^{2}dx
				&\leq C\left( \int_{B_{R}(0)}|u_{n}|^{p}dx\right) ^{\frac{2}{p}}+\varepsilon\int_{B^{c}_{R}(0)}|u_{n}|^{2}dx\\
				&\leq C\left( \int_{\mathbb{R}^{3}}|u_{n}|^{p}dx\right) ^{\frac{2}{P}}+C\varepsilon\rightarrow0\quad  \hbox{as $n\rightarrow\infty$}.
			\end{split}
		\end{equation*}
		This implies that
		\begin{equation*}
			e_{m}(a)=\lim _{n\rightarrow\infty}E_{a}(u_{n})=\lim _{n\rightarrow\infty}E^{\infty}_{a}(u_{n})\geq e^{\infty}_{m}(a),
		\end{equation*}
		which contradicts \eqref{equ2-13}. Thus, via the same argument as in {\bf  Claim II}, we deduce that the dichotomy case cannot occur.
		
		Similar to the proof of Proposition \ref{lem2-2}, there exist a sequence $\left\lbrace x_{n}\right\rbrace \subset\mathbb{R}^{3}$ and $u_{0}\in H^{s}(\mathbb{R}^{3})$ such that
		\begin{equation*}
			u_{n}(x+x_{n})\rightarrow u_{0}(x)\quad  \hbox{in $L^{r}(\mathbb{R}^{3})$ for $r\in \left[ 2, \frac{6}{3-2s}\right) $ as $n\rightarrow\infty$}.
		\end{equation*}
		We now prove that $\left\lbrace x_{n}\right\rbrace $ is uniformly bounded in $\mathbb{R}^{3}$. Assume that $|x_{n}| \rightarrow\infty$ as $n\rightarrow\infty$, set $\check{u}_{n}(x):=u_{n}(x+x_{n})$, one can check that
		\begin{equation*}
			\int_{\mathbb{R}^{3}}V(x) \left| u_{n}\right| ^{2}dx=\int_{\mathbb{R}^{3}}V(x+x_{n}) |\check{u}_{n}|^{2}dx\rightarrow V_{\infty}	\int_{\mathbb{R}^{3}} \left| u_{0}\right| ^{2}dx\quad  \hbox{as $n\rightarrow\infty$}.
		\end{equation*}
		Thus we obtain that
		\begin{equation*}
			e_{m}(a)=\lim _{n\rightarrow\infty}E_{a}(u_{n})=\lim _{n\rightarrow\infty}E^{\infty}_{a}(u_{0})\geq e^{\infty}_{m}(a),
		\end{equation*}
		which contradicts \eqref{equ2-13}, and $\left\lbrace x_{n}\right\rbrace $ is uniformly bounded in $\mathbb{R}^{3}$. Suppose $x_{n}\rightarrow x_{0}$ as $n\rightarrow\infty$. Let $y_{n}=x_{n}-x_{0}$ and $\dot{u}_{n}(x):=u_{n}(x+y_{n})$, then there exists $\dot{u}_{0}\in H^{s}(\mathbb{R}^{3})$ such that
		\begin{equation*}
			\dot{u}_{n}\rightarrow \dot{u}_{0}\quad  \hbox{in $L^{r}(\mathbb{R}^{3})$ for $r\in \left[ 2, \frac{6}{3-2s}\right) $ as $n\rightarrow\infty$}.
		\end{equation*}
		Furthermore, one can see that
		\begin{equation*}
			e_{m}(a)=\lim _{n\rightarrow\infty}E_{a}(u_{n})=\lim _{n\rightarrow\infty}E_{a}(	\dot{u}_{n})\geq E_{a}(	\dot{u}_{0}) \geq e_{m}(a),
		\end{equation*}
		which shows that $\dot{u}_{0}$ is a minimiser of $e_{m}(a)$. This completes the proof of Theorem \ref{the1}.
	\end{proof}
	
	\section{Asymptotic behavior }\label{sec 3}
	
	We investigate in detail the asymptotic behavior of the minimizers for $e_{1}(a)$ as $a\rightarrow\infty$ in the present section. First, we establish following energy estimate.
	
	\begin{lemma}\label{lem3-1}
		Supppose $V(x)$ satisfies $(V_{1})$, then we have
		\begin{equation}\label{equ3-1}
			\lim_{a\rightarrow\infty}\frac{e_{1}(a)}{\left(\frac{a}{\sqrt{a^{\ast}}}\right)^{\frac{4s(p-2)}{4s+6-3p}}}=\frac{3p-6-4s}{6-(3-2s)p},
		\end{equation}
		where $a^{\ast}:=\|Q\|_{2}^{2}$ and $Q>0$ is the unique positive solution of \eqref{equ5}.
	\end{lemma}
	\begin{proof}
		Set a cut-off function $\varphi\in C_{0}^{\infty}(\mathbb{R}^{3})$ such that $\varphi(x)=1$ for $|x|\leq1$, $\varphi(x)=0$ for $|x|\geq2$, $0\leq\varphi\leq1$ and $|\nabla\varphi|\leq2$. Define
		\begin{equation*}
			u_{\tau}(x):=A_{\tau}\frac{\tau^{\frac{3}{2}}}{\|Q\|_{2}}\varphi(x)Q(\tau x),
		\end{equation*}
		where $\tau>0$ and $A_{\tau}>0$ is chosen so that $\|u_{\tau}\|_{2}^{2}=1$. Using \eqref{equ6}, one has $A_{\tau}\rightarrow1$ as $\tau\rightarrow\infty$. By \eqref{equ6}, \eqref{equ7} and \cite[Lemma 3.2]{D}, it follows that
		\begin{equation*}
			E_{a}(u_{\tau})\leq\frac{3(p-2)}{6-(3-2s)p}\tau^{2s}+V(0)-\frac{4s}{6-(3-2s)p}\left(\frac{a}{\sqrt{a^{\ast}}}\right)^{p-2}\tau^{\frac{3p-6}{2}}+o(1) \quad \hbox{as $\tau\rightarrow\infty$}.
		\end{equation*}
		Taking the infimum with respect to $\tau>0$ yields that
		\begin{equation}\label{equ3-2}
			e_{1}(a)\leq	E_{a}(u_{\tau})\leq\frac{3p-6-4s}{6-(3-2s)p}\left(\frac{a}{\sqrt{a^{\ast}}}\right)^{\frac{4s(p-2)}{4s+6-3p}}(1+o(1))\quad \hbox{as $\tau\rightarrow\infty$}.
		\end{equation}
		To prove the lower bound of $e_{1}(a)$, we introduce the following minimization problem
		\begin{equation}\label{equ3-2-1}
			\tilde{e}_{1}(a):=\inf_{u\in S_{1}}\tilde{E}_{a}(u),
		\end{equation}
		where $\tilde{E}_{a}(u)$ is defined by
		\begin{equation*}
			\tilde{E}_{a}(u):=\int_{\mathbb{R}^{3}}|(-\Delta)^{\frac{s}{2}}u|^{2}dx-\frac{2a^{p-2}}{p}\int_{\mathbb{R}^{3}}|u|^{p}dx.
		\end{equation*}
		Similar to the proof of Proposition \ref{lem2-2}, one can derive that $\tilde{e}_{1}(a)$ admits at least one minimiser for $a>0$. Suppose $\tilde{u}_{a}$ is a nonnegative minimizer of $\tilde{e}_{1}(a)$, let $\tilde{\alpha}_{a}:=\left(\frac{a}{\sqrt{a^{\ast}}}\right)^{\frac{2(p-2)}{4s-3(p-2)}}$, we first claim that
		\begin{equation}\label{equ3-3}
			\tilde{e}_{1}(a)=\frac{3p-6-4s}{6-(3-2s)p}\left(\frac{a}{\sqrt{a^{\ast}}}\right)^{\frac{4s(p-2)}{4s+6-3p}}\quad \hbox{and}\quad \tilde{u}_{a}(x)=\frac{1}{\sqrt{a^{\ast}}}\tilde{\alpha}_{a}^{\frac{3}{2}}Q(\tilde{\alpha}_{a}x).
		\end{equation}
		Indeed, we shall prove that
		\begin{equation}\label{equ3-4}
			\tilde{e}_{1}(a)=a^{\frac{4s(p-2)}{4s-3(p-2)}}\tilde{e}_{1}(1),\quad \tilde{u}_{a}(x)=\alpha_{a}^{\frac{3}{2}}\tilde{u}_{1}(\alpha_{a}x)\quad \hbox{with}\quad \alpha_{a}=a^{\frac{2(p-2)}{4s-3(p-2)}}.
		\end{equation}
		Define $\tilde{v}_{1}(x):=\alpha_{a}^{-\frac{3}{2}}\tilde{u}_{a}(\alpha_{a}^{-1}x)$, one can check that
		\begin{equation*}
			\tilde{e}_{1}(a)=\tilde{E}_{a}(\tilde{u}_{a})=a^{\frac{4s(p-2)}{4s-3(p-2)}}\left[\int_{\mathbb{R}^{3}}|(-\Delta)^{\frac{s}{2}}\tilde{v}_{1}|^{2}dx-\frac{2}{p}\int_{\mathbb{R}^{3}}|\tilde{v}_{1}|^{p}dx\right]\geq a^{\frac{4s(p-2)}{4s-3(p-2)}}\tilde{e}_{1}(1).
		\end{equation*}
		Define $\tilde{v}_{a}(x):=\alpha_{a}^{\frac{3}{2}}\tilde{u}_{1}(\alpha_{a}x)$, we also have
		\begin{equation*}
			\tilde{e}_{1}(a)\leq\tilde{E}_{a}(\tilde{v}_{a})=a^{\frac{4s(p-2)}{4s-3(p-2)}}\tilde{e}_{1}(1).
		\end{equation*}
		Thus \eqref{equ3-4} is proved. We next prove that
		\begin{equation}\label{equ3-5}
			\tilde{e}_{1}(1)=\frac{3p-6-4s}{6-(3-2s)p}(\sqrt{a^{\ast}})^{-\frac{4s(p-2)}{4s-3(p-2)}}
		\end{equation}
		and $\tilde{u}_{1}$ satisfies
		\begin{equation}\label{equ3-6}
			\tilde{u}_{1}(x)=(\sqrt{a^{\ast}})^{-\frac{4s}{4s-3(p-2)}}Q\left[(\sqrt{a^{\ast}})^{-\frac{2(p-2)}{4s-3(p-2)}}x\right].
		\end{equation}
		Note that
		\begin{equation*}
			(-\Delta)^{s}\tilde{u}_{1}=\tilde{\mu}_{1}\tilde{u}_{1}+\tilde{u}_{1}^{p-1}\quad\hbox{in $\mathbb{R}^3$},
		\end{equation*}
		where $\tilde{\mu}_{1}\in\mathbb{R}$ is a Lagrange multiplier. One can obtain that
		\begin{equation*}
			\tilde{\mu}_{1}=\tilde{e}_{1}(1)-\frac{p-2}{p}\int_{\mathbb{R}^{3}}|\tilde{u}_{1}|^{p}dx<0.
		\end{equation*}
		Similar to the proof of \cite[Proposition 4.1]{T}, we get that $\tilde{u}_{1}>0$. This implies that
		\begin{equation}\label{equ3-7}
			\tilde{u}_{1}(x)=(-\tilde{\mu}_{1})^{\frac{1}{p-2}}Q\left[(-\tilde{\mu}_{1})^{\frac{1}{2s}}x\right],
		\end{equation}
		due to the fact that $Q>0$ is the unique positive solution of \eqref{equ5}. Applying $\|\tilde{u}_{1}\|_{2}^{2}=1$, one can deduce that $\tilde{\mu}_{1}=-(\sqrt{a^{\ast}})^{-\frac{4s(p-2)}{4s-3(p-2)}}$, from which and \eqref{equ3-7} we obtain that \eqref{equ3-6} holds. Moreover, \eqref{equ7} and \eqref{equ3-7} show that \eqref{equ3-5} holds. In view of \eqref{equ3-4}-\eqref{equ3-6}, one can conclude that \eqref{equ3-3} holds, the claim is thus established.
		
		We now establish the lower bound estimate of $e_{1}(a)$. Let $u_{a}$ be a nonnegative minimizer of $e_{1}(a)$. Define
		\begin{equation}\label{equ3-8}
			v_{a}(x):=\varepsilon_{a}^{\frac{3}{2}}u_{a}(\varepsilon_{a}x)\quad \hbox{with}\quad \varepsilon_{a}:=\left(\frac{a}{\sqrt{a^{\ast}}}\right)^{\frac{2p-4}{3p-6-4s}}.
		\end{equation}
		It then follows from \eqref{equ4} and \eqref{equ8} that
		\begin{equation}\label{equ3-9}
			\begin{split}
				\varepsilon_{a}^{2s}e_{1}(a)=&\varepsilon_{a}^{2s} E_{a}(u_{a})\\
				\geq&\int_{\mathbb{R}^{3}}|(-\Delta)^{\frac{s}{2}} v_{a}|^{2}dx-\frac{\varepsilon_{a}^{4s-3}}{2}\int_{\mathbb{R}^{3}}\int_{\mathbb{R}^{3}}\frac{v_{a}^{2}(x)v_{a}^{2}(y)}{|x-y|^{3-2s}}dxdy-\frac{2(\sqrt{a^{\ast}})^{p-2}}{p}\int_{\mathbb{R}^{3}}|v_{a}|^{p}dx\\
				\geq&\int_{\mathbb{R}^{3}}|(-\Delta)^{\frac{s}{2}} v_{a}|^{2}dx-C\varepsilon_{a}^{4s-3}\left( \int_{\mathbb{R}^{3}}|(-\Delta)^{\frac{s}{2}} v_{a}|^{2}dx\right)^{\frac{6-4s}{4s}}\\
				&-C\left( \int_{\mathbb{R}^{3}}|(-\Delta)^{\frac{s}{2}} v_{a}|^{2}dx\right)^{\frac{3(p-2)}{4}},
			\end{split}
		\end{equation}
		from which and \eqref{equ3-2} we obtain that
		\begin{equation}\label{equ3-10}
			\int_{\mathbb{R}^{3}}|(-\Delta)^{\frac{s}{2}} v_{a}|^{2}dx\leq C.
		\end{equation}
		Combining with \eqref{equ8}, one can see that
		\begin{equation}\label{equ3-11}
			\int_{\mathbb{R}^{3}}\int_{\mathbb{R}^{3}}\frac{v_{a}^{2}(x)v_{a}^{2}(y)}{|x-y|^{3-2s}}dxdy\leq C\left( \int_{\mathbb{R}^{3}}|(-\Delta)^{\frac{s}{2}}  v_{a}|^{2}dx\right)^{\frac{6-4s}{4s}}\leq C.
		\end{equation}
		Using \eqref{equ3-9} and \eqref{equ3-11}, one can check that
		\begin{equation}\label{equ3-12}
			\begin{split}
				\varepsilon_{a}^{2s}e_{1}(a)&\geq\int_{\mathbb{R}^{3}}|(-\Delta)^{\frac{s}{2}}  v_{a}|^{2}dx-\frac{2(\sqrt{a^{\ast}})^{p-2}}{p}\int_{\mathbb{R}^{3}}|v_{a}|^{p}dx+o(1)\\
				&\geq\tilde{e}_{1}(\sqrt{a^{\ast}})+o(1)\quad \hbox{as $a\rightarrow\infty$},
			\end{split}
		\end{equation}
		Thus \eqref{equ3-3} yields that
		\begin{equation}\label{equ3-13}
			\frac{e_{1}(a)}{\left(\frac{a}{\sqrt{a^{\ast}}}\right)^{\frac{4s(p-2)}{4s+6-3p}}}\geq\frac{3p-6-4s}{6-(3-2s)p}\quad \hbox{as $a\rightarrow\infty$}.
		\end{equation}
		It then yield from \eqref{equ3-2} and \eqref{equ3-13} that \eqref{equ3-1} holds, and this completes the proof.
	\end{proof}
	
	For any given sequence $\{a_{k}\}$, of $\{a_{k}\}$ with $a_{k}\rightarrow\infty$ as $k\rightarrow\infty$, we denote $u_{k}$ be a nonnegative minimizer of $e_{1}(a_{k})$.
	\begin{lemma}\label{lem3-2}
		Let $u_{k}$ be a nonnegative minimizer of $e_{1}(a_{k})$ with $a_{k}\rightarrow\infty$ as $k\rightarrow\infty$, then we have\\
		(1) There exist a sequence $\{y_{\varepsilon_{k}}\}$ and positive constants $R_{0}$ and $\eta$ such that the normalized function
		\begin{equation}\label{equ3-14}
			w_{k}(x):=\varepsilon_{k}^{\frac{3}{2}}u_{k}(\varepsilon_{k}x+\varepsilon_{k}y_{\varepsilon_{k}})\quad \hbox{with}\quad \varepsilon_{k}:=\left(\frac{a_{k}}{\sqrt{a^{\ast}}}\right)^{\frac{2p-4}{3p-6-4s}}
		\end{equation}
		satisfies
		\begin{equation}\label{equ3-15}
			\liminf\limits_{k\rightarrow\infty}\int_{B_{R_{0}}(0)}|w_{k}|^{2}dx\geq\eta>0.
		\end{equation}
		Moreover, up to a subsequence, there holds $z_{k}:=\varepsilon_{k}y_{\varepsilon_{k}}\rightarrow x_{0}$ as $k\rightarrow\infty$, and $x_{0}$ satisfies $V(x_{0})=0$.\\
		(2) Passing to a subsequence if necessary, it holds that
		\begin{equation}\label{equ3-16}
			\lim_{k\rightarrow\infty}w_{k}(x)=\frac{Q(x)}{\sqrt{a^{\ast}}}\quad \hbox{in $H^{s}(\mathbb{R}^{3})$}.
		\end{equation}
		Moreover, we have $\mu_{k}\varepsilon_{k}^{2s}\rightarrow-1$ as $k\rightarrow\infty$.
	\end{lemma}
	\begin{proof}
		(1) Define $v_{k}(x):=\varepsilon_{k}^{\frac{3}{2}}u_{k}(\varepsilon_{k}x)$, we claim that there exist some positive constants $C_{1}, C_{2}, C'_{1}$ and $C'_{2}$ such that
		\begin{equation}\label{equ3-17}
			C_{1}\leq\int_{\mathbb{R}^{3}}|v_{k}|^{p}dx\leq C_{2}\quad \hbox{and}\quad C'_{1}\leq\int_{\mathbb{R}^{3}}|(-\Delta)^{\frac{s}{2}}v_{k}|^{2}dx\leq C'_{2}.
		\end{equation}
		In fact, similar to \eqref{equ3-9}-\eqref{equ3-12}, one can see that
		\begin{equation}\label{equ3-18}
			0>\frac{3p-6-4s}{6-(3-2s)p}\geq\int_{\mathbb{R}^{3}}|(-\Delta)^{\frac{s}{2}}  v_{k}|^{2}dx-\frac{2(\sqrt{a^{\ast}})^{p-2}}{p}\int_{\mathbb{R}^{3}}|v_{k}|^{p}dx\quad \hbox{as $k\rightarrow\infty$}.
		\end{equation}
		It then follows from \eqref{equ4} and \eqref{equ3-10} that
		\begin{equation*}
			\int_{\mathbb{R}^{3}}|v_{k}|^{p}dx\leq C_{2}\quad \hbox{and}\quad \int_{\mathbb{R}^{3}}|(-\Delta)^{\frac{s}{2}}v_{k}|^{2}dx\leq C'_{2},
		\end{equation*}
		from which and \eqref{equ3-18} we deduce that $\int_{\mathbb{R}^{3}}|v_{k}|^{p}dx\geq C_{1}$. Then \eqref{equ4} yields that $\int_{\mathbb{R}^{3}}|(-\Delta)^{\frac{s}{2}}v_{k}|^{2}dx\geq C'_{1}$, and hence \eqref{equ3-17} holds. We now prove that there exist a sequence $\{y_{\varepsilon_{k}}\}\subset\mathbb{R}^{3}$ and $R_{0}, \eta>0$ such that
		\begin{equation}\label{equ3-19}
			\liminf_{k\rightarrow\infty}\int_{B_{R_{0}}(y_{\varepsilon_{k}})}v_{k}^{2}dx\geq\eta>0.
		\end{equation}
		If not, then the Vanishing Lemma shows that $v_{k}\rightarrow 0$ in $L^{p}(\mathbb{R}^{3})$ as $k\rightarrow\infty$, which contradicts \eqref{equ3-17}, \eqref{equ3-15} thus follows immediately from \eqref{equ3-19}. Similar to the proof of \cite[Lemma 5.2]{LL}, one can derive that
		$\int_{\mathbb{R}^{3}}V(x)|u_{k}|^{2}dx\rightarrow0$ as $k\rightarrow\infty$, from which and \eqref{equ3-14} we can deduce that
		\begin{equation*}
			\int_{\mathbb{R}^{3}}V(x)|u_{k}(x)|^{2}dx=	\int_{\mathbb{R}^{3}}V(\varepsilon_{k}x+\varepsilon_{k}y_{\varepsilon_{k}})|w_{k}(x)|^{2}dx \rightarrow0\quad \hbox{as $k\rightarrow\infty$}.
		\end{equation*}
		Then a proof similar to \cite[Lemma 2.5]{G}, one can deduce that $z_{k}:=\varepsilon_{k}y_{\varepsilon_{k}}\rightarrow x_{0}$ as $k\rightarrow\infty$, and $V(x_{0})=0$.
		
		(2) From \eqref{equ3-11}, we have
		\begin{equation}\label{equ3-20}
			\int_{\mathbb{R}^{3}}\int_{\mathbb{R}^{3}}\frac{w_{k}^{2}(x)w_{k}^{2}(y)}{|x-y|^{3-2s}}dxdy\leq  C.
		\end{equation}
		It then follows from \eqref{equ3-1}, \eqref{equ3-3}, \eqref{equ3-9} and \eqref{equ3-20} that
		\begin{equation}\label{equ3-21}
			\begin{split}
				&	\varepsilon_{k}^{2s}\int_{\mathbb{R}^{3}}V(\varepsilon_{k}x+z_{k})w_{k}^{2}dx\\
				=&\varepsilon_{k}^{2s}e_{1}(a_{k})-\int_{\mathbb{R}^{3}}|(-\Delta)^{\frac{s}{2}} w_{k}|^{2}dx+\frac{2(\sqrt{a^{\ast}})^{p-2}}{p}\int_{\mathbb{R}^{3}}|w_{k}|^{p}dx+\frac{\varepsilon_{k}^{4s-3}}{2}\int_{\mathbb{R}^{3}}\int_{\mathbb{R}^{3}}\frac{w_{k}^{2}(x)w_{k}^{2}(y)}{|x-y|^{3-2s}}dxdy\\
				\leq&\varepsilon_{k}^{2s}e_{1}(a_{k})-\tilde{e}_{1}(\sqrt{a^{\ast}})+C\varepsilon_{k}^{4s-3}\rightarrow0\quad \hbox{as $k\rightarrow\infty$}.
			\end{split}
		\end{equation}
		Using \eqref{equ3-1}, \eqref{equ3-20} and \eqref{equ3-21}, it follows that
		\begin{equation}\label{equ3-22}
			\frac{3p-6-4s}{6-(3-2s)p}=\lim_{k\rightarrow\infty}\varepsilon_{k}^{2s}e_{1}(a_{k})=\lim_{k\rightarrow\infty}\left[\int_{\mathbb{R}^{3}}|(-\Delta)^{\frac{s}{2}} w_{k}|^{2}dx-\frac{2(\sqrt{a^{\ast}})^{p-2}}{p}\int_{\mathbb{R}^{3}}|w_{k}|^{p}dx \right].
		\end{equation}
		Up to a subsequence if necessary, there exists a $w_{0}\in H^{s}(\mathbb{R}^{3})$ such that $w_{k}\rightharpoonup w_{0}$ in $H^{s}(\mathbb{R}^{3})$, $w_{k}\rightarrow w_{0}$ in $L^{r}_{loc}(\mathbb{R}^{3})$ \rm for $r\in[2, \frac{6}{3-2s})$ and $w_{k}\rightarrow w_{0}$ a.e. in $\mathbb{R}^{3}$. Moreover,  \eqref{equ3-15} yields that $w_{0}\not\equiv0$. By the Br\'{e}zis-Lieb lemma, one can see that
		\begin{equation}\label{equ3-23}
			\begin{split}
				&	\|w_{k}\|_{q}^{q}=\|w_{0}\|_{q}^{q}+\|w_{k}-w_{0}\|_{q}^{q}+o(1)\quad \hbox{as $k\rightarrow\infty$ for $q\in \left[ 2, \frac{6}{3-2s}\right) $}, \\
				&	\|(-\Delta)^{\frac{s}{2}} w_{k}\|_{2}^{2}=\|(-\Delta)^{\frac{s}{2}} w_{0}\|_{2}^{2}+\|(-\Delta)^{\frac{s}{2}}w_{k}-(-\Delta)^{\frac{s}{2}} w_{0}\|_{2}^{2}+o(1)\quad \hbox{as $k\rightarrow\infty$}.
			\end{split}
		\end{equation}
		We next claim that $\|w_{0}\|_{2}^{2}=1$. If $\|w_{0}\|_{2}^{2}=l\in (0, 1)$, then we define $w_{l}:=\frac{w_{0}}{\sqrt{l}}$ and $w_{1-l}:=\frac{w_{k}-w_{0}}{\sqrt{1-l}}$. Applying \eqref{equ3-3}, \eqref{equ3-22} and \eqref{equ3-23} it holds that
		\begin{equation*}
			\begin{split}
				\frac{3p-6-4s}{6-(3-2s)p}=&\int_{\mathbb{R}^{3}}|(-\Delta)^{\frac{s}{2}} w_{0}|^{2}dx-\frac{2(\sqrt{a^{\ast}})^{p-2}}{p}\int_{\mathbb{R}^{3}}|w_{0}|^{p}dx \\
				&+\lim_{k\rightarrow\infty}\left[\int_{\mathbb{R}^{3}}|(-\Delta)^{\frac{s}{2}} (w_{k}-w_{0})|^{2}dx- \frac{2(\sqrt{a^{\ast}})^{p-2}}{p}\int_{\mathbb{R}^{3}}|w_{k}-w_{0}|^{p}dx\right] \\
				=&l\int_{\mathbb{R}^{3}}|(-\Delta)^{\frac{s}{2}} w_{l}|^{2}dx-\frac{2(\sqrt{a^{\ast}})^{p-2}}{p}(\sqrt{l})^{p}\int_{\mathbb{R}^{3}}|w_{l}|^{p}dx \\
				&+\lim_{k\rightarrow\infty}\left[(1-l)\int_{\mathbb{R}^{3}}|(-\Delta)^{\frac{s}{2}}w_{1-l}|^{2}dx- \frac{2(\sqrt{a^{\ast}})^{p-2}}{p}(\sqrt{1-l})^{p}\int_{\mathbb{R}^{3}}|w_{1-l}|^{p}dx\right] \\
				>&l\left[\int_{\mathbb{R}^{3}}|(-\Delta)^{\frac{s}{2}} w_{l}|^{2}dx-\frac{2(\sqrt{a^{\ast}})^{p-2}}{p}\int_{\mathbb{R}^{3}}|w_{l}|^{p}dx \right] \\
				&+(1-l)\lim_{k\rightarrow\infty}\left[\int_{\mathbb{R}^{3}}|(-\Delta)^{\frac{s}{2}} w_{1-l}|^{2}dx- \frac{2(\sqrt{a^{\ast}})^{p-2}}{p}\int_{\mathbb{R}^{3}}|w_{1-l}|^{p}dx\right] \\
				&\geq l\tilde{e}_{1}(\sqrt{a^{\ast}})+(1-l)\tilde{e}_{1}(\sqrt{a^{\ast}})=\tilde{e}_{1}(\sqrt{a^{\ast}})=\frac{3p-6-4s}{6-(3-2s)p},
			\end{split}
		\end{equation*}
		where we have used the fact that $\int_{\mathbb{R}^{3}}|w_{l}|^{p}dx\geq C$ or $\int_{\mathbb{R}^{3}}|w_{1-l}|^{p}dx\geq C$, in view of \eqref{equ3-22} and \eqref{equ3-23}. This is a contradiction, and thus we have $\|w_{0}\|_{2}^{2}=1$. Using  the interpolation inequality, one can derive that
		\begin{equation}\label{equ3-24}
			w_{k}(x)\rightarrow w_{0}(x)\quad \hbox{in $L^{q}(\mathbb{R}^{3})$ for $q\in[2, \frac{6}{3-2s})$ as $k\rightarrow\infty$}.
		\end{equation}
		It then follows from \eqref{equ3-22} and \eqref{equ3-24} that
		\begin{equation*}
			\begin{split}
				\tilde{e}_{1}(\sqrt{a^{\ast}})&\geq\lim_{k\rightarrow\infty}\left[\int_{\mathbb{R}^{3}}|(-\Delta)^{\frac{s}{2}} w_{k}|^{2}dx-\frac{2(\sqrt{a^{\ast}})^{p-2}}{p}\int_{\mathbb{R}^{3}}|w_{k}|^{p}dx \right]\\
				&\geq\int_{\mathbb{R}^{3}}|(-\Delta)^{\frac{s}{2}} w_{0}|^{2}dx-\frac{2(\sqrt{a^{\ast}})^{p-2}}{p}\int_{\mathbb{R}^{3}}|w_{0}|^{p}dx\geq\tilde{e}_{1}(\sqrt{a^{\ast}}),
			\end{split}
		\end{equation*}
		which yields that $w_{0}$ is a minimizer of $\tilde{e}_{1}(\sqrt{a^{\ast}})$ and $w_{k}(x)\rightarrow w_{0}(x)$ in $H^{s}(\mathbb{R}^{3})$ as $k\rightarrow\infty$. Using \eqref{equ3-3}, it holds that $w_{k}(x)\rightarrow w_{0}(x):=\frac{Q(x)}{\sqrt{a^{\ast}}}$ as $k\rightarrow\infty$. Thus \eqref{equ3-16} is proved.
		
		Finally, one can check that $w_{k}$ satisfies that
		\begin{equation}\label{equ3-25}
			(-\Delta)^{s} w_{k}+\varepsilon_{k}^{2s}V(\varepsilon_{k}x+\varepsilon_{k}y_{\varepsilon_{k}})w_{k}=\mu_{k}\varepsilon_{k}^{2s}w_{k}+\varepsilon_{k}^{4s-3}\int_{\mathbb{R}^{3}}\frac{w^{2}_{k}(y)}{|x-y|^{3-2s}}dyw_{k}+(\sqrt{a^{\ast}})^{p-2}w_{k}^{p-1}.
		\end{equation}
		Applying \eqref{equ7}, \eqref{equ3-1}, \eqref{equ3-14}, \eqref{equ3-16}, \eqref{equ3-20} and \eqref{equ3-25}, it follows that
		\begin{equation}\label{equ3-26}
			\begin{split}
				\mu_{k}\varepsilon_{k}^{2s}&=\varepsilon_{k}^{2s}e_{1}(a_{k})-\frac{\varepsilon_{k}^{4s-3}}{2}\int_{\mathbb{R}^{3}}\int_{\mathbb{R}^{3}}\frac{w_{k}^{2}(x)w_{k}^{2}(y)}{|x-y|^{3-2s}}dxdy-\frac{(p-2)(\sqrt{a^{\ast}})^{p-2}}{p}\int_{\mathbb{R}^{3}}|w_{k}|^{p}dx\\
				&\rightarrow\frac{3p-6-4s}{6-(3-2s)p}-\frac{(p-2)(\sqrt{a^{\ast}})^{p-2}}{p}(\sqrt{a^{\ast}})^{-p}\int_{\mathbb{R}^{3}}|Q|^{p}dx=-1\quad \hbox{as $k\rightarrow\infty$}.
			\end{split}
		\end{equation}
		This completes the proof of Lemma \ref{lem3-2}.
	\end{proof}
	
	\begin{lemma}\label{lem3-3}
		Let $u_{k}$ be a nonnegative minimizer of $e_{1}(a_{k})$ with $a_{k}\rightarrow \infty$ as $k\rightarrow\infty$, then there exists a sequence $\{u_{k}\}$, still denoted by  $\{u_{k}\}$, such that each $u_{k}$ has a unique global maximum point $x_{k}$ satisfying
		\begin{equation}\label{equ3-27}
			\lim\limits_{k\rightarrow \infty} x_{k} =x_{0}\quad\hbox{with $V(x_{0})=0$},
		\end{equation}
		and
		\begin{equation}\label{equ3-28}
			\lim\limits_{k\rightarrow \infty} \varepsilon_{k}^{\frac{3}{2}}u_{k}(\varepsilon_{k}x+x_{k}) =\frac{Q(x)}{\sqrt{a^{\ast}}} \quad\hbox{in $L^{\infty}(\mathbb{R}^{3})$}.
		\end{equation}
	\end{lemma}
	\begin{proof}
		Since $\left\lbrace w_{k}\right\rbrace $ is uniformly bounded in $H^{s}(\mathbb{R}^{3})$, one can derive that
		\begin{equation}\label{equ3-29}
			\begin{split}
				\int_{\mathbb{R}^{3}}\frac{w^{2}_{k}(y)}{|x-y|^{3-2s}}dy\leq& \left(\int_{|x-y|<1}\frac{1}{|x-y|^{\frac{9-6s}{2s}}}dy\right)^{\frac{2s}{3}}\left(\int_{|x-y|<1}|w_{k}|^{\frac{6}{3-2s}}dy\right)^{\frac{3-2s}{3}}\\
				&+\int_{|x-y|\geq1}|w_{k}|^{2}dy\leq C\int_{\mathbb{R}^{3}}\left(|(-\Delta)^{\frac{s}{2}} w_{k}|^{2}+|w_{k}|^{2}\right)dx\leq C.
			\end{split}
		\end{equation}
		Then some arguments similar to \cite[Lemma 2.4]{DD} and \cite[Lemma 5.6]{TT} yield that
		\begin{equation}\label{equ3-30}
			\|w_{k}\|_{\infty}\leq C\quad \mathrm {and} \quad w_{k} \rightarrow 0  \quad\hbox{as $|x| \rightarrow \infty$ uniformly for large $k$ },
		\end{equation}
		and $w_{k}$ satisfies the following decay estimate
		\begin{equation}\label{equ3-31}
			w_{k}(x)\leq\frac{C}{1+|x|^{3+2s}}\quad \hbox{ uniformly for large $k$}.
		\end{equation}
		Using \eqref{equ3-31} and Lemma \ref{lem3-2}, we deduce that for any $\rho>0$ small enough, it holds
		\begin{equation}\label{equ3-32}
			u_{k}(x)=\varepsilon_{k}^{-\frac{3}{2}}w_{k}\left(\frac{x-z_{k}}{\varepsilon_{k}}\right)\rightarrow 0\quad\hbox{as $k\rightarrow\infty$ for $x\in B_{\rho}^{c}(x_{0})$}.
		\end{equation}
		Let $x_{k}$ be a global maximum point of $u_{k}(x)$ and define
		\begin{equation}\label{equ3-33}
			\bar{w}_{k}(x):=\varepsilon_{k}^{\frac{3}{2}}u_{k}(\varepsilon_{k}x+x_{k}).
		\end{equation}
		Note that $0$ is a global maximum point of $\bar{w}_{k}(x)$. Using the proof by contradiction and the vanishing lemma, it then follows from \eqref{equ3-17} that $	\bar{w}_{k}(0)\geq C$
		as $k$ large enough, hence we have $u_{k}(x_{k})\geq C\varepsilon_{k}^{-\frac{3}{2}}$. This together with \eqref{equ3-32} implies that \eqref{equ3-27} holds.
		
		Moreover, similar to the proof of Lemma \ref{lem3-2}, we have
		\begin{equation}\label{equ3-34}
			(-\Delta)^{s} \bar{w}_{k}+\varepsilon_{k}^{2s}V(\varepsilon_{k}x+x_{k})\bar{w}_{k}=\mu_{k}\varepsilon_{k}^{2s}\bar{w}_{k}+\varepsilon_{k}^{4s-3}\int_{\mathbb{R}^{3}}\frac{\bar{w}^{2}_{k}(y)}{|x-y|^{3-2s}}dy\bar{w}_{k}+(\sqrt{a^{\ast}})^{p-2}\bar{w}_{k}^{p-1},
		\end{equation}
		and
		\begin{equation}\label{equ3-35}
			\bar{w}_{k}(x)\rightarrow\bar{w}_{0}(x):=\frac{Q(x)}{\sqrt{a^{\ast}}}\quad \hbox{in $H^{s}(\mathbb{R}^{3})$ as $k\rightarrow\infty$}.
		\end{equation}
		Next, we verify the uniqueness of $x_{k}$ as $k\rightarrow\infty$. Using \eqref{equ3-26}, \eqref{equ3-29} and \eqref{equ3-30}, one can obtain from \cite[Proposition 2.9]{SS} that $\|\bar{w}_{k}\|_{C^{1, \alpha}(\mathbb{R}^{3})}\leq C$ for some $\alpha\in (0, 1)$ as $k\rightarrow\infty$. Then, fix any $x_{1}\neq x_{2}\in \mathbb{R}^{3}$, it follows from \eqref{equ3-31} that
		\begin{equation*}
			\begin{split}
				&\frac{\left| \int_{\mathbb{R}^{3}}\frac{|\bar{w}_{k}(y)|^{2}}{|x_{1}-y|^{3-2s}}dy-\int_{\mathbb{R}^{3}}\frac{|\bar{w}_{k}(y)|^{2}}{|x_{2}-y|^{3-2s}}dy\right|}{|x_{1}-x_{2}|^{\alpha}}\\
				\leq&\int_{\mathbb{R}^{3}}\frac{\left| \bar{w}^{2}_{k}(x_{1}-y)- \bar{w}^{2}_{k}(x_{2}-y)\right| }{|x_{1}-x_{2}|^{\alpha}|y|^{3-2s}}dy	\leq C\|\bar{w}_{k}\|_{C^{\alpha}(\mathbb{R}^{3})}\int_{\mathbb{R}^{3}}\frac{\left| \bar{w}_{k}(x_{1}-y)+\bar{w}_{k}(x_{2}-y)\right| }{|y|^{3-2s}}dy\\
				\leq&C\|\bar{w}_{k}\|^{2}_{C^{\alpha}(\mathbb{R}^{3})}\int_{|y|\leq1}\frac{1}{|y|^{3-2s}}dy+C\|\bar{w}_{k}\|_{C^{\alpha}(\mathbb{R}^{3})}\int_{|y|\geq1}\left|\bar{w}_{k}(x_{1}-y)+\bar{w}_{k}(x_{2}-y)\right|dy \\
				\leq& C+C\int_{\mathbb{R}^{3}}|\bar{w}_{k}|dx\leq C,
			\end{split}
		\end{equation*}
		which shows that
		\begin{equation*}
			\int_{\mathbb{R}^{3}}\frac{|\bar{w}_{k}(y)|^{2}}{|x-y|^{3-2s}}dy\in C^{\alpha}(\mathbb{R}^{3})\quad \hbox{for some $\alpha\in (0, 1)$ as $k\rightarrow\infty$}.
		\end{equation*}
		Sinve $V(x)\in C^{\alpha}(\mathbb{R}^{3})$, applying the
		Schauder estimate \cite[Theorem 2.11]{J}, one can derive that $\{\bar{w}_{k}\}$ is bounded uniformly in $C_{loc}^{2, \alpha}(\mathbb{R}^{3})$ as $k\rightarrow\infty$. Further, \eqref{equ3-35} yields that
		\begin{equation}\label{equ3-36}
			\bar{w}_{k}\rightarrow \bar{w}_{0}\quad\hbox{in $C_{loc}^{2}(\mathbb{R}^{3})$ as $k\rightarrow\infty$.}
		\end{equation}
		Since the origin is the unique global maximum point of $Q(x)$, it follows from \eqref{equ3-35} and \eqref{equ3-36} that all global maximum points of $\bar{w}_{k}$ lie in a small ball $B_{\rho}(0)$ for some $\rho>0$ as $k\rightarrow\infty$. Moreover, the condition $Q''(0)<0$ implies that $Q''(r)<0$ for $r\in[0, \rho)$. Then \cite[Lemma 4.2]{N} implies that the uniqueness of $x_{k}$. Furthermore, from the polynomial decay of $\bar{w}_{k}(x)$ and $Q(x)$, a simple analysis gives that $	\bar{w}_{k}\rightarrow \bar{w}_{0}$ in $L^{\infty}(\mathbb{R}^{3})$ as $k\rightarrow\infty$, and this completes the proof of Lemma \ref{lem3-3}.
	\end{proof}
	
	Now we are ready to prove Theorem \ref{the2}.
	
	\begin{proof}[Proof of Theorem \ref{the2}] Note that
		\begin{equation*}
			e_{1}(a_{k})\leq E_{a_{k}}(u_{k})\leq E_{a_{k}}(u_{k}(x-\varepsilon_{k}y_{0}+x_{k}-x_{0})),
		\end{equation*}
		where $x_{0}\in Z_{0}$ and $y_{0}\in K_{0}$. Direct calculation indicates that
		\begin{equation}\label{equ3-37}
			\begin{split}
				\int_{\mathbb{R}^{3}}V(x)|u_{k}|^{2}dx&\leq\int_{\mathbb{R}^{3}}V(x)|u_{k}(x-\varepsilon_{k}y_{0}+x_{k}-x_{0})|^{2}dx\\
				&=\int_{\mathbb{R}^{3}}V(\varepsilon_{k}x+\varepsilon_{k}y_{0}+x_{0})|\bar{w}_{k}(x)|^{2}dx\\
				&=\int_{\mathbb{R}^{3}}\frac{V(\varepsilon_{k}x+\varepsilon_{k}y_{0}+x_{0})}{V_{0}(\varepsilon_{k}x+\varepsilon_{k}y_{0})}V_{0}(\varepsilon_{k}x+\varepsilon_{k}y_{0})|\bar{w}_{k}(x)|^{2}dx\\
				&=\frac{(1+o(1))\varepsilon_{k}^{r}}{a^{\ast}}\int_{\mathbb{R}^{3}}V_{0}(x+y_{0})|Q(x)|^{2}dx\\&=\frac{(1+o(1))\varepsilon_{k}^{r}}{a^{\ast}}\bar{\lambda}_{0}.
			\end{split}
		\end{equation}
		Moreover, we also have
		\begin{equation}\label{equ3-38}
			\begin{split}
				\int_{\mathbb{R}^{3}}V(x)|u_{k}|^{2}dx&=	\int_{\mathbb{R}^{3}}V(\varepsilon_{k}x+x_{k})|\bar{w}_{k}(x)|^{2}dx\\
				&=\int_{\mathbb{R}^{3}}\frac{V(\varepsilon_{k}x+x_{k}-x_{i}+x_{i})}{V_{i}(\varepsilon_{k}x+x_{k}-x_{i})}V_{i}(\varepsilon_{k}x+x_{k}-x_{i})|\bar{w}_{k}(x)|^{2}dx\\
				&\geq(1+o(1))\int_{B_{\frac{1}{\sqrt{\varepsilon_{k}}}}(x_{i})}V_{i}(\varepsilon_{k}x+x_{k}-x_{i})|\bar{w}_{k}(x)|^{2}dx\\
				&=\frac{(1+o(1))\varepsilon_{k}^{r_{i}}}{a^{\ast}}\int_{B_{\frac{1}{\sqrt{\varepsilon_{k}}}}(x_{i})}V_{i}\left( x+\frac{x_{k}-x_{i}}{\varepsilon_{k}}\right) |Q(x)|^{2}dx.
			\end{split}
		\end{equation}
		In view of \eqref{equ3-37} and \eqref{equ3-38}, it holds that
		\begin{equation*}
			\frac{(1+o(1))\varepsilon_{k}^{r_{i}}}{a^{\ast}}\int_{B_{\frac{1}{\sqrt{\varepsilon_{k}}}}(x_{i})}V_{i}\left( x+\frac{x_{k}-x_{i}}{\varepsilon_{k}}\right) |Q(x)|^{2}dx\leq \frac{(1+o(1))\varepsilon_{k}^{r}}{a^{\ast}}\bar{\lambda}_{0},
		\end{equation*}
		which implies that
		\begin{equation*}
			x_{i}\in Z_{0},\quad r_{i}=r \quad \hbox{and} \quad \left\lbrace\frac{x_{k}-x_{i}}{\varepsilon_{k}} \right\rbrace~  \hbox{bounded},
		\end{equation*}
		due to the fact $V_{i}(x)\rightarrow\infty$ as $|x|\rightarrow\infty$. It can be further verified that
		\begin{equation*}
			\frac{x_{k}-x_{i}}{\varepsilon_{k}} \rightarrow y_{0}\in K_{0}\quad\hbox{as $k\rightarrow\infty$}.
		\end{equation*}
		This therefore completes the proof of Theorem \ref{the2}.
	\end{proof}

	\section{Local Uniqueness}\label{sec 4}
	
	Given the limiting behavior of nonnegative minimizers derived in Section 2, the main result we establish in this section is Theorem \ref{the3}, which addresses the local uniqueness of positive minimizers for $e_{1}(a)$ as $a\rightarrow\infty$. Let $u_{k}$ be a nonnegative minimizer of $e_{1}(a_{k})$, one has
	\begin{equation}\label{equ4-1}
		(-\Delta)^{s} u_{k}+V(x)u_{k}=\mu_{k}u_{k}+a_{k}^{p-2}u_{k}^{p-1}+\int_{\mathbb{R}^{3}}\frac{u^{2}_{k}(y)}{|x-y|^{3-2s}}dyu_{k}\quad \hbox{in $\mathbb{R}^{3}$},
	\end{equation}
	where the associated Lagrange multiplier $\mu_{k}\in\mathbb{R}$ satisfies
	\begin{equation}\label{equ4-2}
		\mu_{k}=e_{1}(a_{k})-\frac{1}{2}\int_{\mathbb{R}^{3}}\int_{\mathbb{R}^{3}}\frac{u^{2}_{k}(x)u^{2}_{k}(y)}{|x-y|^{3-2s}}dxdy-\frac{p-2}{p}a_{k}^{p-2}\int_{\mathbb{R}^{3}}|u_{k}|^{p}dx.
	\end{equation}
	It then follows from \eqref{equ3-26} that
	\begin{equation}\label{equ4-3}
		\mu_{k}\varepsilon^{2s}_{k}\rightarrow-1\quad \hbox{as $k\rightarrow\infty$}.
	\end{equation}
	Argue by contradiction, suppose that $u_{1k}$ and $u_{2k}$  are two different nonnegative minimizers of $e_{1}(a_{k})$, where $a_{k}\rightarrow\infty$ as $k\rightarrow\infty$. Let $x_{1k}$ and $x_{2k}$ be the unique global maximum point of $u_{1k}$ and $u_{2k}$, respectively. Define
	\begin{equation}\label{equ4-4}
		\hat{u}_{ik}(x):=\sqrt{a^{\ast}}\varepsilon^{\frac{3}{2}}_{k}u_{ik}(\varepsilon_{k}x+x_{1k})\quad \hbox{in $\mathbb{R}^{3}$,\quad $i=1, 2$}
	\end{equation}
	and
	\begin{equation}\label{equ4-5}
		\hat{\eta}_{k}(x):=\frac{	\hat{u}_{1k}(x)-	\hat{u}_{2k}(x)}{\|\hat{u}_{1k}(x)-	\hat{u}_{2k}(x)\|_{\infty}}=\frac{	u_{1k}(x)-	u_{2k}(x)}{\|u_{1k}(x)-u_{2k}(x)\|_{\infty}}\quad \hbox{in $\mathbb{R}^{3}$}.
	\end{equation}
	We first establish several properties of $\hat{\eta}_{k}(x)$.
	\begin{lemma}\label{lem4-1}
		Under the assumptions of Theorem \ref{the3}, then there exists a constant $C>0$, independent
		of $k>0$, such that
		\begin{equation}\label{equ4-6}
			\int_{\mathbb{R}^{3}}|(-\Delta)^{\frac{s}{2}}\hat{\eta}_{k}|^{2}dx+\varepsilon_{k}^{2s}\int_{\mathbb{R}^{3}}V(\varepsilon_{k}x+x_{1k})\hat{\eta}_{k}^{2}dx+\int_{\mathbb{R}^{3}}\hat{\eta}_{k}^{2}dx\leq C\quad \hbox{as $k\rightarrow\infty$}.
		\end{equation}
	\end{lemma}
	\begin{proof}
		Using \eqref{equ4-1}-\eqref{equ4-5}, one can derive that
		\begin{equation}\label{equ4-7}
			(-\Delta)^{s} \hat{u}_{ik}+\varepsilon_{k}^{2s}V(\varepsilon_{k}x+x_{1k})\hat{u}_{ik}=\mu_{ik}\varepsilon_{k}^{2s}\hat{u}_{ik}+\frac{\varepsilon_{k}^{4s-3}}{a^{\ast}}\int_{\mathbb{R}^{3}}\frac{\hat{u}_{ik}^{2}(y)}{|x-y|^{3-2s}}dy\hat{u}_{ik}+\hat{u}_{ik}^{p-1}\quad \hbox{in $\mathbb{R}^{3}$},
		\end{equation}
		and
		\begin{equation}\label{equ4-8}
			(-\Delta)^{s}\hat{\eta}_{k}+\varepsilon_{k}^{2s}V(\varepsilon_{k}x+x_{1k})\hat{\eta}_{k}=\mu_{1k}\varepsilon_{k}^{2s}\hat{\eta}_{k}+g_{k}(x)+f_{k}(x) \quad\hbox{in $\mathbb{R}^3$},
		\end{equation}
		where there exists $\theta\in(0, 1)$ such that
		\begin{equation}\label{equ4-9}
			g_{k}(x):=\frac{\varepsilon_{k}^{4s-3}}{a^{\ast}}\int_{\mathbb{R}^{3}}\frac{\hat{u}_{1k}^{2}(y)}{|x-y|^{3-2s}}dy\hat{\eta}_{k}+(p-1)[\theta\hat{u}_{1k}+(1-\theta)\hat{u}_{2k}]^{p-2}\hat{\eta}_{k},
		\end{equation}
		\begin{equation}\label{equ4-10}
			\begin{split}
				f_{k}(x):=&\frac{\varepsilon_{k}^{2s}\hat{u}_{2k}(\mu_{1k}-\mu_{2k})}{\|\hat{u}_{1k}(x)-	\hat{u}_{2k}(x)\|_{\infty}}+\frac{\varepsilon_{k}^{4s-3}}{a^{\ast}}\int_{\mathbb{R}^{3}}\frac{[\hat{u}_{1k}(y)+\hat{u}_{2k}(y)]}{|x-y|^{3-2s}}\hat{\eta}_{k}(y)dy\hat{u}_{2k}\\
				=&-\frac{\varepsilon_{k}^{4s-3}}{2(a^{\ast})^{2}}\int_{\mathbb{R}^{3}}\int_{\mathbb{R}^{3}}\frac{[\hat{u}_{1k}(y)+\hat{u}_{2k}(y)]}{|x-y|^{3-2s}}\hat{\eta}_{k}(y)\hat{u}^{2}_{1k}(x)dxdy\hat{u}_{2k}\\
				&-\frac{\varepsilon_{k}^{4s-3}}{2(a^{\ast})^{2}}\int_{\mathbb{R}^{3}}\int_{\mathbb{R}^{3}}\frac{[\hat{u}_{1k}(x)+\hat{u}_{2k}(x)]}{|x-y|^{3-2s}}\hat{\eta}_{k}(x)\hat{u}^{2}_{2k}(y)dxdy\hat{u}_{2k}\\
				&-\frac{p-2}{a^{\ast}}\int_{\mathbb{R}^{3}}[\theta\hat{u}_{1k}+(1-\theta)\hat{u}_{2k}]^{p-1}\hat{\eta}_{k}dx\hat{u}_{2k}\\
				&+\frac{\varepsilon_{k}^{4s-3}}{a^{\ast}}\int_{\mathbb{R}^{3}}\frac{[\hat{u}_{2k}(y)+\hat{u}_{1k}(y)]}{|x-y|^{3-2s}}\hat{\eta}_{k}(y)dy\hat{u}_{2k}.
			\end{split}
		\end{equation}
		Multiplying \eqref{equ4-8} by $\hat{\eta}_{k}$ and integrating on $\mathbb{R}^{3}$, we have
		\begin{align*} 
			&\int_{\mathbb{R}^{3}}|(-\Delta)^{\frac{s}{2}}\hat{\eta}_{k}|^{2}dx+\varepsilon_{k}^{2s}\int_{\mathbb{R}^{3}}V(\varepsilon_{k}x+x_{1k})\hat{\eta}_{k}^{2}dx-\mu_{1k}\varepsilon_{k}^{2s}\int_{\mathbb{R}^{3}}\hat{\eta}_{k}^{2}dx\\
			=&\frac{\varepsilon_{k}^{4s-3}}{a^{\ast}}\int_{\mathbb{R}^{3}}\int_{\mathbb{R}^{3}}\frac{\hat{u}_{1k}^{2}(y)\hat{\eta}_{k}^{2}(x)}{|x-y|^{3-2s}}dxdy+(p-1)\int_{\mathbb{R}^{3}}[\theta\hat{u}_{1k}+(1-\theta)\hat{u}_{2k}]^{p-2}\hat{\eta}_{k}^{2}dx\\
			&-\frac{\varepsilon_{k}^{4s-3}}{2(a^{\ast})^{2}}\int_{\mathbb{R}^{3}}\int_{\mathbb{R}^{3}}\frac{[\hat{u}_{1k}(y)+\hat{u}_{2k}(y)]}{|x-y|^{3-2s}}\hat{\eta}_{k}(y)\hat{u}^{2}_{1k}(x)dxdy\int_{\mathbb{R}^{3}}\hat{u}_{2k}(x)\hat{\eta}_{k}(x)dx\\
			&-\frac{\varepsilon_{k}^{4s-3}}{2(a^{\ast})^{2}}\int_{\mathbb{R}^{3}}\int_{\mathbb{R}^{3}}\frac{[\hat{u}_{1k}(x)+\hat{u}_{2k}(x)]}{|x-y|^{3-2s}}\hat{\eta}_{k}(x)\hat{u}^{2}_{2k}(y)dxdy\int_{\mathbb{R}^{3}}\hat{u}_{2k}(x)\hat{\eta}_{k}(x)dx\\
			&-\frac{p-2}{a^{\ast}}\int_{\mathbb{R}^{3}}[\theta\hat{u}_{1k}+(1-\theta)\hat{u}_{2k}]^{p-1}\hat{\eta}_{k}dx\int_{\mathbb{R}^{3}}\hat{u}_{2k}(x)\hat{\eta}_{k}(x)dx\\
			&+\frac{\varepsilon_{k}^{4s-3}}{a^{\ast}}\int_{\mathbb{R}^{3}}\frac{[\hat{u}_{2k}(y)+\hat{u}_{1k}(y)]}{|x-y|^{3-2s}}\hat{\eta}_{k}(y)dy\int_{\mathbb{R}^{3}}\hat{u}_{2k}(x)\hat{\eta}_{k}(x)dx.
		\end{align*}
		Applying Lemma \ref{lem5-1}, one has
		\begin{equation}\label{equ4-11}
			\begin{split}
				&(p-1)\int_{\mathbb{R}^{3}}[\theta\hat{u}_{1k}+(1-\theta)\hat{u}_{2k}]^{p-2}\hat{\eta}_{k}^{2}dx\\
				&-\frac{p-2}{a^{\ast}}\int_{\mathbb{R}^{3}}[\theta\hat{u}_{1k}+(1-\theta)\hat{u}_{2k}]^{p-1}\hat{\eta}_{k}dx\int_{\mathbb{R}^{3}}\hat{u}_{2k}(x)\hat{\eta}_{k}(x)dx\\
				\leq& \frac{1}{4}\int_{\mathbb{R}^{3}}\hat{\eta}_{k}^{2}dx+C\quad \hbox{as $k\rightarrow\infty$},
			\end{split}
		\end{equation}
		where we have used the fact that $\hat{u}_{ik}(x)\rightarrow0$ as $|x|\rightarrow\infty$. Similar to \eqref{equ3-29}, we can also deduce that
		\begin{equation}\label{equ4-12}
			\int_{\mathbb{R}^{3}}\frac{\hat{u}_{ik}^{2}(y)}{|x-y|^{3-2s}}dy\leq C\|\hat{u}_{ik}\|_{H^{s}(\mathbb{R}^{3})}^{2}\leq C.
		\end{equation}
		Thus
		\begin{equation}\label{equ4-13}
			\begin{split}
				\frac{\varepsilon_{k}^{4s-3}}{a^{\ast}}\int_{\mathbb{R}^{3}}\int_{\mathbb{R}^{3}}\frac{\hat{u}_{1k}^{2}(y)\hat{\eta}_{k}^{2}(x)}{|x-y|^{3-2s}}dxdy&\leq C\varepsilon_{k}^{4s-3}	\int_{\mathbb{R}^{3}}\frac{\hat{u}_{1k}^{2}(y)}{|x-y|^{3-2s}}dy\int_{\mathbb{R}^{3}}\hat{\eta}_{k}^{2}(x)dx\\
				&\leq C\varepsilon_{k}^{4s-3}\int_{\mathbb{R}^{3}}\hat{\eta}_{k}^{2}(x)dx\quad \hbox{as $k\rightarrow\infty$}.
			\end{split}
		\end{equation}
		By the fact that $|\hat{\eta}_{k}|\leq1$, it then follows from \eqref{equ4-12} and Lemma \ref{lem5-1} that
		\begin{equation}\label{equ4-14}
			\begin{split}
				&\frac{\varepsilon_{k}^{4s-3}}{2(a^{\ast})^{2}}\int_{\mathbb{R}^{3}}\int_{\mathbb{R}^{3}}\frac{[\hat{u}_{1k}(y)+\hat{u}_{2k}(y)]}{|x-y|^{3-2s}}\hat{\eta}_{k}(y)\hat{u}^{2}_{1k}(x)dxdy\int_{\mathbb{R}^{3}}\hat{u}_{2k}(x)\hat{\eta}_{k}(x)dx\\
				\leq&C\varepsilon_{k}^{4s-3}\int_{\mathbb{R}^{3}}\frac{\hat{u}_{1k}^{2}(x)}{|x-y|^{3-2s}}dx\int_{\mathbb{R}^{3}}[\hat{u}_{1k}(y)+\hat{u}_{2k}(y)]dy\int_{\mathbb{R}^{3}}\hat{u}_{2k}(x)dx\\
				\leq& C\varepsilon_{k}^{4s-3}\quad \hbox{as $k\rightarrow\infty$}.
			\end{split}
		\end{equation}
		Applying the H\"{o}lder inequality, we obtain that
		\begin{equation}\label{equ4-15}
			\begin{split}
				&\frac{\varepsilon_{k}^{4s-3}}{a^{\ast}}\int_{\mathbb{R}^{3}}\frac{[\hat{u}_{2k}(y)+\hat{u}_{1k}(y)]}{|x-y|^{3-2s}}\hat{\eta}_{k}(y)dy\int_{\mathbb{R}^{3}}\hat{u}_{2k}(x)\hat{\eta}_{k}(x)dx\\
				\leq&C\varepsilon_{k}^{4s-3}\left(\int_{\mathbb{R}^{3}}\frac{|\hat{u}_{2k}(y)+\hat{u}_{1k}(y)|^{2}}{|x-y|^{3-2s}}dy \right)^{\frac{1}{2}}\left(\int_{\mathbb{R}^{3}}\frac{|\hat{\eta}_{k}(y)|^{2}}{|x-y|^{3-2s}}dy \right)^{\frac{1}{2}}\\
				&\times\left(\int_{\mathbb{R}^{3}}|\hat{u}_{2k}(x)|^{2}dx\right)^{\frac{1}{2}}\left(\int_{\mathbb{R}^{3}}|\hat{\eta}_{k}(x)|^{2}dx\right)^{\frac{1}{2}}\\
				\leq&C\varepsilon_{k}^{4s-3}\|\hat{u}_{1k}+\hat{u}_{2k}\|_{H^{s}(\mathbb{R}^{3})}\|\hat{u}_{2k}\|_{H^{s}(\mathbb{R}^{3})}\|\hat{\eta}_{k}\|_{H^{s}(\mathbb{R}^{3})}^{2}\\
				\leq&C\varepsilon_{k}^{4s-3}\int_{\mathbb{R}^{3}}|(-\Delta)^{\frac{s}{2}}\hat{\eta}_{k}|^{2}dx+C\varepsilon_{k}^{4s-3}\int_{\mathbb{R}^{3}}\hat{\eta}_{k}^{2}dx\quad \hbox{as $k\rightarrow\infty$}.
			\end{split}
		\end{equation}
		In view of \eqref{equ4-11}-\eqref{equ4-15}, it follows that
		\begin{align*} 
			&\int_{\mathbb{R}^{3}}|(-\Delta)^{\frac{s}{2}}\hat{\eta}_{k}|^{2}dx+\varepsilon_{k}^{2s}\int_{\mathbb{R}^{3}}V(\varepsilon_{k}x+x_{1k})\hat{\eta}_{k}^{2}dx-\mu_{1k}\varepsilon_{k}^{2s}\int_{\mathbb{R}^{3}}\hat{\eta}_{k}^{2}dx\\
			\leq&C\varepsilon_{k}^{4s-3}\int_{\mathbb{R}^{3}}|(-\Delta)^{\frac{s}{2}}\hat{\eta}_{k}|^{2}dx+C\varepsilon_{k}^{4s-3}\int_{\mathbb{R}^{3}}\hat{\eta}_{k}^{2}dx+\frac{1}{4}\int_{\mathbb{R}^{3}}\hat{\eta}_{k}^{2}dx+C\varepsilon_{k}^{4s-3}+C\quad \hbox{as $k\rightarrow\infty$}.
		\end{align*}
		Combining this with \eqref{equ4-3} gives \eqref{equ4-6}, this therefore completes the proof of Lemma \ref{lem4-1}.
	\end{proof}
	
	\begin{lemma}\label{lem4-2}
		Let $\hat{\eta}_{k}$ be as in \eqref{equ4-5}, we then have the following decays polynomially
		\begin{equation}\label{equ4-16}
			|\hat{\eta}_{k}(x)|\leq\frac{C}{1+|x|^{3+2s}}\quad \hbox{as $k\rightarrow\infty$}.
		\end{equation}
		Moreover, we have $\hat{\eta}_{k}\in C_{loc}^{1, \alpha}(\mathbb{R}^{3})$ for some $\alpha\in(0, 1)$ as $k\rightarrow\infty$. 
	\end{lemma}
	\begin{proof}
		Lemma \ref{lem4-1} shows that
		\begin{equation}\label{equ4-17}
			\|\hat{\eta}_{k}\|_{H^{s}(\mathbb{R}^{3})}^{2}=\int_{\mathbb{R}^{3}}|(-\Delta)^{\frac{s}{2}}\hat{\eta}_{k}|^{2}dx+\int_{\mathbb{R}^{3}}\hat{\eta}_{k}^{2}dx\leq C\quad \hbox{as $k\rightarrow\infty$}.
		\end{equation}
		We rewrite \eqref{equ4-8} as follows
		\begin{equation*}
			(-\Delta)^{s}\hat{\eta}_{k}+\rho_{k}(x)\hat{\eta}_{k}=f_{k}(x) \quad\hbox{in $\mathbb{R}^3$},
		\end{equation*}
		where $\rho_{k}(x):=\varepsilon_{k}^{2s}V(\varepsilon_{k} x+x_{1k})-\varepsilon_{k}^{2s}\mu_{1k}-\frac{\varepsilon_{k}^{4s-3}}{a^{\ast}}\int_{\mathbb{R}^{3}}\frac{\hat{u}_{1k}^{2}(y)}{|x-y|^{3-2s}}dy-(p-1)[\theta\hat{u}_{1k}+(1-\theta)\hat{u}_{2k}]^{p-2}$ and $f_{k}(x)$ be as in \eqref{equ4-10}. By the boundedness of potential, we then deduce from \eqref{equ4-3}, \eqref{equ4-12} and Lemma \ref{lem5-1} that 
		\begin{equation}\label{equ4-18}
			\rho_{k}(x)\geq\rho_{0}>0\quad \hbox{for $|x|\geq R$ large and $k\rightarrow\infty$}.
		\end{equation}
		Using \eqref{equ4-12} and Lemma \ref{lem5-1}, one can derive that
		\begin{equation*}
			\begin{split}
				\Bigg|&-\frac{\varepsilon_{k}^{4s-3}}{2(a^{\ast})^{2}}\int_{\mathbb{R}^{3}}\int_{\mathbb{R}^{3}}\frac{[\hat{u}_{1k}(y)+\hat{u}_{2k}(y)]}{|x-y|^{3-2s}}\hat{\eta}_{k}(y)\hat{u}^{2}_{1k}(x)dxdy\hat{u}_{2k}\\
				&-\frac{\varepsilon_{k}^{4s-3}}{2(a^{\ast})^{2}}\int_{\mathbb{R}^{3}}\int_{\mathbb{R}^{3}}\frac{[\hat{u}_{1k}(x)+\hat{u}_{2k}(x)]}{|x-y|^{3-2s}}\hat{\eta}_{k}(x)\hat{u}^{2}_{2k}(y)dxdy\hat{u}_{2k}\\
				&-\frac{p-2}{a^{\ast}}\int_{\mathbb{R}^{3}}[\theta\hat{u}_{1k}+(1-\theta)\hat{u}_{2k}]^{p-1}\hat{\eta}_{k}dx\hat{u}_{2k}\Bigg|\leq \frac{C}{|x|^{3+2s}}\quad \hbox{as $k\rightarrow\infty$}.
			\end{split}
		\end{equation*}
		By \eqref{equ4-17} and Lemma \ref{lem5-1}, we also have
		\begin{equation*}
			\begin{split}
				&	\left| \frac{\varepsilon_{k}^{4s-3}}{a^{\ast}}\int_{\mathbb{R}^{3}}\frac{[\hat{u}_{2k}(y)+\hat{u}_{1k}(y)]}{|x-y|^{3-2s}}\hat{\eta}_{k}(y)dy\hat{u}_{2k}\right| \\
				\leq&C\left(\int_{\mathbb{R}^{3}}\frac{|\hat{u}_{2k}(y)+\hat{u}_{1k}(y)|^{2}}{|x-y|^{3-2s}}dy \right)^{\frac{1}{2}}\left(\int_{\mathbb{R}^{3}}\frac{|\hat{\eta}_{k}(y)|^{2}}{|x-y|^{3-2s}}dy \right)^{\frac{1}{2}}|\hat{u}_{2k}(x)|	\\
				\leq&C\|\hat{u}_{1k}+\hat{u}_{2k}\|_{H^{s}(\mathbb{R}^{3})}\|\hat{\eta}_{k}\|_{H^{s}(\mathbb{R}^{3})}|\hat{u}_{2k}(x)|\leq \frac{C}{|x|^{3+2s}}\quad \hbox{as $k\rightarrow\infty$}.
			\end{split}
		\end{equation*}
		Thus we obtain
		\begin{equation}\label{equ4-19}
			|f_{k}(x)|\leq \frac{C}{|x|^{3+2s}}\quad \hbox{as $k\rightarrow\infty$}.
		\end{equation}
		In view of \eqref{equ4-18} and \eqref{equ4-19}, it then follows from Lemma \ref{lem5-2} that
		\begin{equation*}
			|\hat{\eta}_{k}(x)|\leq\frac{C}{1+|x|^{3+2s}}\quad \hbox{as $k\rightarrow\infty$},
		\end{equation*}
		and thus  \eqref{equ4-16} is proved. Further, denote
		\begin{align*} 
			\hat{D}_{k}(x):=-\varepsilon_{k}^{2s}V(\varepsilon_{k}x+x_{1k})\hat{\eta}_{k}+\mu_{1k}\varepsilon_{k}^{2s}\hat{\eta}_{k}+g_{k}(x)+f_{k}(x) 
		\end{align*}
		so that $\hat{\eta}_{k}$ satisfies
		\begin{align*} 
			-\Delta\hat{\eta}_{k}=\hat{D}_{k}(x)\quad \hbox{in $\mathbb{R}^{3}$}.
		\end{align*}
		By the fact that $|\hat{\eta}_{k}|\leq1$ and the boundedness of $V(x)$, we derive from \eqref{equ4-3}, \eqref{equ4-12} and Lemma \ref{lem5-1} that $\hat{D}_{k}(x)\in L^{\infty}(\mathbb{R}^{3})$ as $k\rightarrow\infty$. Then \cite[Proposition 2.9]{SS} yields that
		\begin{equation*}
			\left\|\hat{\eta}_{k} \right\| _{C^{1, \alpha}(\mathbb{R}^{3})}\leq C\left( \left\|\hat{\eta}_{k} \right\| _{\infty}+\left\|\hat{D}_{k} \right\| _{\infty}\right) \leq C\quad \hbox{as $k\rightarrow\infty$},
		\end{equation*}
		this therefore completes the proof of Lemma \ref{lem4-2}.
	\end{proof}
	
	Using \eqref{equ4-17}, we may assume that $\hat{\eta}_{k} \rightharpoonup\hat{\eta}_{0}$ in $H^{s}(\mathbb{R}^{3})$ as $k\rightarrow\infty$, similar to \eqref{equ4-11}-\eqref{equ4-15}, by \eqref{equ4-3}, \eqref{equ4-12}, \eqref{equ4-16}, Lemma \ref{lem5-1} and the dominated convergence theorem, a simple analysis yields that $\hat{\eta}_{0}$ is a weak solution of 
	\begin{equation}\label{equ4-20}
		(-\Delta)^{s}\hat{\eta}_{0}+\left[ 1-(p-1)Q^{p-2}\right] \hat{\eta}_{0}=-\frac{(p-2)Q}{a^{\ast}}\int_{\mathbb{R}^{3}}Q^{p-1}\hat{\eta}_{0}dx.
	\end{equation}
	where $Q>0$ is the unique positive solution of \eqref{equ5}. Define the linearized operator
	\begin{equation*}
		L:=	(-\Delta)^{s}+1-(p-1)Q^{p-2}.
	\end{equation*}
	Recall from \cite{FF} that
	\begin{equation*}
		{\rm Ker} L={\rm span}\left\lbrace \frac{\partial Q}{\partial x_{1}},  \frac{\partial Q}{\partial x_{2}},  \frac{\partial Q}{\partial x_{3}}\right\rbrace \quad \hbox{and}\quad L\left(Q+\frac{p-2}{2s}x\cdot\nabla Q\right)=-(p-2)Q.
	\end{equation*}
	Thus, \eqref{equ4-20} and Lemma \ref{lem4-2} show that there exist a subsequence of $\{\hat{\eta}_{k}\}$, such that
	\begin{equation}\label{equ4-21}
		\hat{\eta}_{k}\rightarrow\hat{\eta}_{0}\quad \hbox{in $C_{loc}(\mathbb{R}^{3})$ as $k\rightarrow\infty$},
	\end{equation}
	and
	\begin{equation}\label{equ4-22}
		\hat{\eta}_{0}(x)=b_{0}\left(Q+\frac{p-2}{2s}x\cdot\nabla Q\right)+\sum_{i=1}^{3}b_{i}\frac{\partial Q}{\partial x_{i}},
	\end{equation}
	where $b_{0}, b_{1}, b_{2}$ and $b_{3}$ are some constants.
	
	\begin{lemma}\label{lem4-3}
		Under the assumptions of Theorem \ref{the3}, and assume that $V_{0}(x)$ is defined by \eqref{equ13}, then there holds that
		\begin{equation}\label{equ4-23}
			\frac{b_{0}(p-2)}{2s}\int_{\mathbb{R}^{3}}\frac{\partial V_{0}(x+y_{0})}{\partial x_{j}}(x\cdot\nabla Q^{2})dx-\sum^{3}_{i=1}
			b_{i}\int_{\mathbb{R}^{3}}\frac{\partial^{2}V_{0}(x+y_{0})}{\partial x_{i}\partial x_{j}}Q^{2}dx=0.
		\end{equation}
	\end{lemma}
	\begin{proof}
		Multiplying \eqref{equ4-7} by $\frac{\partial\hat{u}_{ik}}{\partial x_{j}}$ and integrating over $\mathbb{R}^{3}$, one has
		\begin{equation}\label{equ4-24}
			\begin{split}
				&\int_{\mathbb{R}^{3}}(-\Delta)^{s} \hat{u}_{ik}\frac{\partial\hat{u}_{ik}}{\partial x_{j}}dx+\varepsilon_{k}^{2s}\int_{\mathbb{R}^{3}}V(\varepsilon_{k}x+x_{1k})\hat{u}_{ik}\frac{\partial\hat{u}_{ik}}{\partial x_{j}}dx\\
				=&\mu_{ik}\varepsilon_{k}^{2s}\int_{\mathbb{R}^{3}}\hat{u}_{ik}\frac{\partial\hat{u}_{ik}}{\partial x_{j}}dx+\frac{\varepsilon_{k}^{4s-3}}{a^{\ast}}\int_{\mathbb{R}^{3}}\int_{\mathbb{R}^{3}}\frac{\hat{u}_{ik}^{2}(y)}{|x-y|^{3-2s}}\hat{u}_{ik}(x)\frac{\partial\hat{u}_{ik}}{\partial x_{j}}dydx+\int_{\mathbb{R}^{3}}\hat{u}_{ik}^{p-1}\frac{\partial\hat{u}_{ik}}{\partial x_{j}}dx.
			\end{split}
		\end{equation}
		Direct calculation indicates that
		\begin{equation}\label{equ4-25}
			\begin{split}
				&\varepsilon_{k}^{2s}\int_{\mathbb{R}^{3}}V(\varepsilon_{k}x+x_{1k})\hat{u}_{ik}\frac{\partial\hat{u}_{ik}}{\partial x_{j}}dx=\varepsilon_{k}^{2s}\lim_{R\rightarrow\infty}\int_{ B_{R}(0)}V(\varepsilon_{k}x+x_{1k})\hat{u}_{ik}\frac{\partial\hat{u}_{ik}}{\partial x_{j}}dx\\
				=&\frac{\varepsilon_{k}^{2s}}{2}\lim_{R\rightarrow\infty}\int_{ \partial B_{R}(0)}V(\varepsilon_{k}x+x_{1k})\hat{u}_{ik}^{2}\nu_{j} dS-\frac{\varepsilon_{k}^{2s}}{2}\int_{\mathbb{R}^{3}}\frac{\partial V(\varepsilon_{k}x+x_{1k})}{\partial x_{j}}\hat{u}_{ik}^{2}dx\\
				=&-\frac{\varepsilon_{k}^{2s}}{2}\int_{\mathbb{R}^{3}}\frac{\partial V(\varepsilon_{k}x+x_{1k})}{\partial x_{j}}\hat{u}_{ik}^{2}dx,
			\end{split}
		\end{equation}
		where we have used the polynomial decay \eqref{equ5-1}, which shows that
		\begin{equation*}
			\left|\int_{ \partial B_{R}(0)}V(\varepsilon_{k}x+x_{1k})\hat{u}_{ik}^{2}\nu_{j} dS\right| \leq C R^{-6-4s}R^{2}\rightarrow0\quad \hbox{as $R\rightarrow\infty$}.
		\end{equation*}
		Similarly, using Lemma \ref{lem5-1}, we also have
		\begin{equation}\label{equ4-26}
			\begin{split}
				&	\left| \mu_{ik}\varepsilon_{k}^{2s}\int_{\mathbb{R}^{3}}\hat{u}_{ik}\frac{\partial\hat{u}_{ik}}{\partial x_{j}}dx+\int_{\mathbb{R}^{3}}\hat{u}_{ik}^{p-1}\frac{\partial\hat{u}_{ik}}{\partial x_{j}}dx\right| \\
				=&\left|\mu_{ik}\varepsilon_{k}^{2s}\lim_{R\rightarrow\infty}\int_{B_{R}(0)}\hat{u}_{ik}\frac{\partial\hat{u}_{ik}}{\partial x_{j}}dx+\lim_{R\rightarrow\infty}\int_{B_{R}(0)}\hat{u}_{ik}^{p-1}\frac{\partial\hat{u}_{ik}}{\partial x_{j}}dx \right| \\
				=&\left|\frac{\mu_{ik}\varepsilon_{k}^{2s}}{2}\lim_{R\rightarrow\infty}\int_{\partial B_{R}(0)}\hat{u}_{ik}^{2}\nu_{j}dS+\frac{1}{p}\lim_{R\rightarrow\infty}\int_{\partial B_{R}(0)}\hat{u}_{ik}^{p}\nu_{j}dS \right| \\
				\leq&C\lim_{R\rightarrow\infty}\left(R^{-6-4s} R^{2}+R^{-p(3+2s)}R^{2}\right) =0.
			\end{split}
		\end{equation}
		Applying \eqref{equ4-12} and Lemma \ref{lem5-1}, one derive that
		\begin{equation}\label{equ4-27}
			\begin{split}
				&	\left| \int_{\mathbb{R}^{3}}\int_{\mathbb{R}^{3}}\frac{\hat{u}_{ik}^{2}(y)}{|x-y|^{3-2s}}\hat{u}_{ik}(x)\frac{\partial\hat{u}_{ik}}{\partial x_{j}}dydx\right| \\
				=&\left| \frac{1}{2}\lim_{R\rightarrow\infty}\int_{ \partial B_{R}(0)}\int_{\mathbb{R}^{3}}\frac{\hat{u}^{2}_{ik}(x)\hat{u}^{2}_{ik}(y)}{|x-y|^{3-2s}}\nu_{j}dydS+\frac{3-2s}{2}	\int_{\mathbb{R}^{3}}\int_{\mathbb{R}^{3}}\frac{x_{j}-y_{j}}{|x-y|^{5-2s}}\hat{u}_{ik}^{2}(x)\hat{u}_{ik}^{2}(y)dxdy\right| \\
				\leq&C\lim_{R\rightarrow\infty}\int_{\mathbb{R}^{3}}\frac{\hat{u}^{2}_{ik}(y)}{|x-y|^{3-2s}}dy\int_{ \partial B_{R}(0)}\hat{u}^{2}_{ik}(x)\nu_{j}dS\leq C\lim_{R\rightarrow\infty}R^{-6-4s} R^{2}=0,
			\end{split}
		\end{equation}
		where we have used the fact that
		\begin{equation*}
			\int_{\mathbb{R}^{3}}\int_{\mathbb{R}^{3}}\frac{x_{j}-y_{j}}{|x-y|^{5-2s}}\hat{u}_{ik}^{2}(x)\hat{u}_{ik}^{2}(y)dxdy=0,
		\end{equation*}
		By \eqref{equ4-24}-\eqref{equ4-26}, one can deduce from Lemma \ref{lem5-3} that
		\begin{equation}\label{equ4-28}
			\begin{split}
				0&=\varepsilon_{k}^{2s}\int_{\mathbb{R}^{3}}\frac{\partial}{\partial x_{j}}V\left[\varepsilon_{k}(x+y_{0})+x_{0} \right] \hat{u}_{ik}^{2}dx\\
				&=\varepsilon_{k}^{2s+r}\left[ \int_{\mathbb{R}^{3}}\frac{\partial}{\partial x_{j}}V_{0}(x+y_{0})\hat{u}_{ik}^{2}dx+o(1)\right] \quad \hbox{as $k\rightarrow\infty$},
			\end{split}
		\end{equation}
		where $V_{0}(x)$ is defined by \eqref{equ13}, $x_{0}\in Z_{0}$ and $y_{0}\in K_{0}$. In view of \eqref{equ6}, \eqref{equ4-5}, \eqref{equ4-16}, \eqref{equ4-21}, \eqref{equ4-22} and \eqref{equ5-1}, one can obtain from \eqref{equ4-28} that
		\begin{equation*}
			\begin{split}
				0=&\int_{\mathbb{R}^{3}}\frac{\partial V_{0}(x+y_{0})}{\partial x_{j}}(\hat{u}_{1k}+\hat{u}_{2k})\hat{\eta}_{k}dx+o(1)\\
				=&2\int_{\mathbb{R}^{3}}\frac{\partial V_{0}(x+y_{0})}{\partial x_{j}}Q\left[ b_{0}\left(Q+\frac{p-2}{2s}x\cdot\nabla Q\right)+\sum_{i=1}^{3}b_{i}\frac{\partial Q}{\partial x_{i}}\right] dx+o(1)\\
				=&2b_{0}\int_{\mathbb{R}^{3}}\frac{\partial V_{0}(x+y_{0})}{\partial x_{j}}Q^{2}dx+\frac{(p-2)b_{0}}{2s}\int_{\mathbb{R}^{3}}\frac{\partial V_{0}(x+y_{0})}{\partial x_{j}}x\cdot\nabla Q^{2}dx\\
				&+\sum_{i=1}^{3}b_{i}\int_{\mathbb{R}^{3}}\frac{\partial V_{0}(x+y_{0})}{\partial x_{j}}\frac{\partial Q^{2}}{\partial x_{i}}dx+o(1)\\
				=&\frac{b_{0}(p-2)}{2s}\int_{\mathbb{R}^{3}}\frac{\partial V_{0}(x+y_{0})}{\partial x_{j}}(x\cdot\nabla Q^{2})dx-\sum^{3}_{i=1}
				b_{i}\int_{\mathbb{R}^{3}}\frac{\partial^{2}V_{0}(x+y_{0})}{\partial x_{i}\partial x_{j}}Q^{2}dx+o(1)\quad \hbox{as $k\rightarrow\infty$},
			\end{split}
		\end{equation*}
		which shows that \eqref{equ4-23} holds, and we thus complete the proof of Lemma \ref{lem4-3}.
	\end{proof}
	
	\begin{lemma}\label{lem4-4}
		Under the assumptions of Theorem \ref{the3}, and assume that $V_{0}(x)$ is defined by \eqref{equ13}, then we have that $\hat{\eta}_{0}\equiv0$.
	\end{lemma}
	\begin{proof}
		Multiplying \eqref{equ4-7} by $(x-x_{1k})\cdot\nabla\hat{u}_{ik}(x)$ and integrating over $\mathbb{R}^{3}$, it follows that
		\begin{equation}\label{equ4-29}
			\begin{split}
				&\int_{\mathbb{R}^{3}}(-\Delta)^{s}\hat{u}_{ik}[(x-x_{1k})\cdot\nabla\hat{u}_{ik}]dx+\varepsilon_{k}^{2s}\int_{\mathbb{R}^{3}}V(\varepsilon_{k}x+x_{1k})\hat{u}_{ik}[(x-x_{1k})\cdot\nabla\hat{u}_{ik}]dx\\
				=&\mu_{ik}\varepsilon_{k}^{2s}\int_{\mathbb{R}^{3}}\hat{u}_{ik}[(x-x_{1k})\cdot\nabla\hat{u}_{ik}]dx+\int_{\mathbb{R}^{3}}\hat{u}_{ik}^{p-1}[(x-x_{1k})\cdot\nabla\hat{u}_{ik}]dx\\
				&+\frac{\varepsilon_{k}^{4s-3}}{a^{\ast}}\int_{\mathbb{R}^{3}}\int_{\mathbb{R}^{3}}\frac{\hat{u}_{ik}^{2}(y)}{|x-y|^{3-2s}}\hat{u}_{ik}(x)[(x-x_{1k})\cdot\nabla\hat{u}_{ik}]dydx.
			\end{split}
		\end{equation}
		Applying \eqref{equ4-7} and Lemma \ref{lem5-4}, one deduce that
		\begin{equation}\label{equ4-30}
			\begin{split}
				&\int_{\mathbb{R}^{3}}(-\Delta)^{s}\hat{u}_{ik}[(x-x_{1k})\cdot\nabla\hat{u}_{ik}]dx=\frac{2s-3}{2} \int_{\mathbb{R}^3} \left| (-\Delta)^{\frac{s}{2}}\hat{u}_{ik}\right| ^2dx\\
				=&-\frac{2s-3}{2}\varepsilon_{k}^{2s}\int_{\mathbb{R}^{3}} V(\varepsilon_{k}x+x_{1k})\hat{u}_{ik}^{2}dx+ \frac{2s-3}{2}\mu_{ik}\varepsilon_{k}^{2s}\int_{\mathbb{R}^{3}}\hat{u}_{ik}^{2}dx\\
				&+\frac{2s-3}{2}\int_{\mathbb{R}^{3}}\hat{u}_{ik}^{p}dx+\frac{2s-3}{2}\frac{\varepsilon_{k}^{4s-3}}{a^{\ast}}\int_{\mathbb{R}^{3}}\int_{\mathbb{R}^{3}}\frac{\hat{u}_{ik}^{2}(y)\hat{u}_{ik}^{2}(x)}{|x-y|^{3-2s}}dydx.
			\end{split}
		\end{equation}
		Integrating by parts, one can see that
		\begin{equation}\label{equ4-31}
			\begin{split}
				&\varepsilon_{k}^{2s}\int_{\mathbb{R}^{3}}V(\varepsilon_{k}x+x_{1k})\hat{u}_{ik}[(x-x_{1k})\cdot\nabla\hat{u}_{ik}]dx\\
				=&\varepsilon_{k}^{2s}\lim_{R\rightarrow\infty}\int_{ B_{R}(0)}V(\varepsilon_{k}x+x_{1k})\hat{u}_{ik}[(x-x_{1k})\cdot\nabla\hat{u}_{ik}]dx\\
				=&\frac{\varepsilon_{k}^{2s}}{2}\lim_{R\rightarrow\infty}\int_{\partial B_{R}(0)}V(\varepsilon_{k}x+x_{1k})(x-x_{1k})\cdot\nu\hat{u}_{ik}^{2}dS\\
				&-\frac{\varepsilon_{k}^{2s}}{2}\int_{\mathbb{R}^{3}}\left[3V(\varepsilon_{k}x+x_{1k})+(x-x_{1k})\cdot\nabla V(\varepsilon_{k}x+x_{1k})\right] \hat{u}_{ik}^{2}dx\\
				=&-\frac{\varepsilon_{k}^{2s}}{2}\int_{\mathbb{R}^{3}}\left[3V(\varepsilon_{k}x+x_{1k})+(x-x_{1k})\cdot\nabla V(\varepsilon_{k}x+x_{1k})\right] \hat{u}_{ik}^{2}dx,
			\end{split}
		\end{equation}
		\begin{equation}\label{equ4-32}
			\begin{split}
				&\mu_{ik}\varepsilon_{k}^{2s}\int_{\mathbb{R}^{3}}\hat{u}_{ik}[(x-x_{1k})\cdot\nabla\hat{u}_{ik}]dx=\mu_{ik}\varepsilon_{k}^{2s}\lim_{R\rightarrow\infty}\int_{ B_{R}(0)}\hat{u}_{ik}[(x-x_{1k})\cdot\nabla\hat{u}_{ik}]dx\\
				=&\frac{\mu_{ik}\varepsilon_{k}^{2s}}{2}\lim_{R\rightarrow\infty}\int_{\partial B_{R}(0)}(x-x_{1k})\cdot\nu\hat{u}_{ik}^{2}dS-\frac{3\mu_{ik}\varepsilon_{k}^{2s}}{2}\int_{\mathbb{R}^{3}}\hat{u}_{ik}^{2}dx=-\frac{3\mu_{ik}\varepsilon_{k}^{2s}}{2}\int_{\mathbb{R}^{3}}\hat{u}_{ik}^{2}dx,
			\end{split}
		\end{equation}
		\begin{equation}\label{equ4-33}
			\begin{split}
				&\int_{\mathbb{R}^{3}}\hat{u}_{ik}^{p-1}[(x-x_{1k})\cdot\nabla\hat{u}_{ik}]dx=\lim_{R\rightarrow\infty}\int_{ B_{R}(0)}\hat{u}_{ik}^{p-1}[(x-x_{1k})\cdot\nabla\hat{u}_{ik}]dx\\
				=&\frac{1}{p}\lim_{R\rightarrow\infty}\int_{\partial B_{R}(0)}(x-x_{1k})\cdot\nu\hat{u}_{ik}^{p}dS-\frac{3}{p}\int_{\mathbb{R}^{3}}\hat{u}_{ik}^{p}dx=-\frac{3}{p}\int_{\mathbb{R}^{3}}\hat{u}_{ik}^{p}dx,
			\end{split}
		\end{equation}
		and
		\begin{equation}\label{equ4-34}
			\begin{split}
				&\frac{\varepsilon_{k}^{4s-3}}{a^{\ast}}\int_{\mathbb{R}^{3}}\int_{\mathbb{R}^{3}}\frac{\hat{u}_{ik}^{2}(y)}{|x-y|^{3-2s}}\hat{u}_{ik}(x)[(x-x_{1k})\cdot\nabla\hat{u}_{ik}]dydx\\
				=&\frac{\varepsilon_{k}^{4s-3}}{a^{\ast}}\lim_{R\rightarrow\infty}\int_{ B_{R}(0)}\int_{\mathbb{R}^{3}}\frac{\hat{u}_{ik}^{2}(y)}{|x-y|^{3-2s}}\hat{u}_{ik}(x)[(x-x_{1k})\cdot\nabla\hat{u}_{ik}]dydx\\
				=&\frac{\varepsilon_{k}^{4s-3}}{2a^{\ast}}\lim_{ R\rightarrow\infty}\int_{ \partial B_{R}(0)}\int_{\mathbb{R}^{3}}\frac{\hat{u}_{ik}^{2}(y)}{|x-y|^{3-2s}}(x-x_{1k})\cdot\nu\hat{u}_{ik}^{2}(x)dydS\\
				&-\frac{3\varepsilon_{k}^{4s-3}}{2a^{\ast}}\int_{\mathbb{R}^{3}}\int_{\mathbb{R}^{3}}\frac{\hat{u}_{ik}^{2}(y)\hat{u}_{ik}^{2}(x)}{|x-y|^{3-2s}}dydx\\
				&+\frac{(3-2s)\varepsilon_{k}^{4s-3}}{2a^{\ast}}\int_{\mathbb{R}^{3}}\int_{\mathbb{R}^{3}}\frac{(x-x_{1k})\cdot(x-y)}{|x-y|^{5-2s}}\hat{u}^{2}_{ik}(x)\hat{u}^{2}_{ik}(y)dydx\\
				=&-\frac{(3+2s)\varepsilon_{k}^{4s-3}}{4a^{\ast}}\int_{\mathbb{R}^{3}}\int_{\mathbb{R}^{3}}\frac{\hat{u}_{ik}^{2}(y)\hat{u}_{ik}^{2}(x)}{|x-y|^{3-2s}}dydx,
			\end{split}
		\end{equation}
		where \eqref{equ4-31}-\eqref{equ4-34} hold due to the facts that
		\begin{equation*}
			\left| \int_{\partial B_{R}(0)}V(\varepsilon_{k}x+x_{1k})(x-x_{1k})\cdot\nu\hat{u}_{ik}^{2}dS\right|
			\leq CR^{-6-4s}R^{2}\rightarrow0\quad \hbox{as $R\rightarrow\infty$},
		\end{equation*}
		\begin{equation*}
			\left| \int_{\partial B_{R}(0)}(x-x_{1k})\cdot\nu\hat{u}_{ik}^{2}dS\right|
			\leq CR^{-6-4s}R^{2}\rightarrow0\quad \hbox{as $R\rightarrow\infty$},
		\end{equation*}
		\begin{equation*}
			\left| \int_{\partial B_{R}(0)}(x-x_{1k})\cdot\nu\hat{u}_{ik}^{p}dS\right| \leq CR^{-p(3+2s)}R^{2}\rightarrow0\quad \hbox{as $R\rightarrow\infty$},
		\end{equation*}
		\begin{align*} 
			&	\left| \int_{ \partial B_{R}(0)}\int_{\mathbb{R}^{3}}\frac{\hat{u}_{ik}^{2}(y)}{|x-y|^{3-2s}}(x-x_{1k})\cdot\nu\hat{u}_{ik}^{2}(x)dydS\right| \\
			\leq&\left| \int_{\mathbb{R}^{3}}\frac{\hat{u}_{ik}^{2}(y)}{|x-y|^{3-2s}}dy\int_{ \partial B_{R}(0)}(x-x_{1k})\cdot\nu\hat{u}_{ik}^{2}(x)dS\right| \leq CR^{-6-4s}R^{2}\rightarrow0\quad \hbox{as $R\rightarrow\infty$},
		\end{align*}
		and
		\begin{align*} 
			\int_{\mathbb{R}^{3}}\int_{\mathbb{R}^{3}}\frac{(x-x_{1k})\cdot(x-y)}{|x-y|^{5-2s}}\hat{u}^{2}_{ik}(x)\hat{u}^{2}_{ik}(y)dydx
			=\frac{1}{2}\int_{\mathbb{R}^{3}}\int_{\mathbb{R}^{3}}\frac{\hat{u}^{2}_{ik}(x)\hat{u}^{2}_{ik}(y)}{|x-y|^{3-2s}}dydx.
		\end{align*}
		It then follows from \eqref{equ4-29}-\eqref{equ4-34} that
		\begin{align*} 
			\left( \frac{3}{p}+\frac{2s-3}{2}\right) \int_{\mathbb{R}^{3}}\hat{u}_{ik}^{p}dx=&sa^{\ast}+s\varepsilon_{k}^{2s}\int_{\mathbb{R}^{3}}V(\varepsilon_{k}x+x_{1k})\hat{u}_{ik}^{2}dx\\
			&+\frac{1}{2}\varepsilon_{k}^{2s}\int_{\mathbb{R}^{3}}\left[ (x-x_{1k})\cdot \nabla V(\varepsilon_{k}x+x_{1k})\right] \hat{u}_{ik}^{2}dx\\
			&-\frac{(6s-3)\varepsilon_{k}^{4s-3}}{4a^{\ast}}\int_{\mathbb{R}^{3}}\int_{\mathbb{R}^{3}}\frac{\hat{u}_{ik}^{2}(y)\hat{u}_{ik}^{2}(x)}{|x-y|^{3-2s}}dydx \quad \hbox{as $k\rightarrow\infty$},
		\end{align*}
		where we have used the facts that $\mu_{ik}\varepsilon^{2s}_{k}\rightarrow-1$ and $\hat{u}_{ik}(x)\rightarrow Q(x)$ in $L^{\infty}(\mathbb{R}^{3})$
		as $k\rightarrow\infty$. Thus we have
		\begin{equation*}
			\begin{split}
				&\frac{6+(2s-3)p}{2}\int_{\mathbb{R}^{3}}\left[\theta \hat{u}_{1k}+(1-\theta)\hat{u}_{2k}\right]^{p-1}\hat{\eta}_{k}dx\\
				=&s\varepsilon_{k}^{2s}\int_{\mathbb{R}^{3}}V(\varepsilon_{k}x+x_{1k})(\hat{u}_{1k}+\hat{u}_{2k})\hat{\eta}_{k}dx\\
				&+\frac{1}{2}\varepsilon_{k}^{2s}\int_{\mathbb{R}^{3}}\left[ (x-x_{1k})\cdot \nabla V(\varepsilon_{k}x+x_{1k})\right](\hat{u}_{1k}+\hat{u}_{2k})\hat{\eta}_{k}dx\\
				&-\frac{(6s-3)\varepsilon_{k}^{4s-3}}{4a^{\ast}}\int_{\mathbb{R}^{3}}\int_{\mathbb{R}^{3}}\frac{\hat{u}_{1k}^{2}(x)}{|x-y|^{3-2s}}(\hat{u}_{1k}(y)+\hat{u}_{2k}(y))\hat{\eta}_{k}(y)dydx\\
				&-\frac{(6s-3)\varepsilon_{k}^{4s-3}}{4a^{\ast}}\int_{\mathbb{R}^{3}}\int_{\mathbb{R}^{3}}\frac{\hat{u}_{2k}^{2}(y)}{|x-y|^{3-2s}}(\hat{u}_{1k}(x)+\hat{u}_{2k}(x))\hat{\eta}_{k}(x)dydx:=A_{k}(x).
			\end{split}
		\end{equation*}
		Using polynomial decay \eqref{equ4-16} and \eqref{equ5-1}, similar arguments of Lemma \ref{lem4-1} yield that
		\begin{equation*}
			A_{k}(x)=O(\varepsilon_{k}^{4s-3})\quad \hbox{as $k\rightarrow\infty$}.
		\end{equation*}
		By the fact that $\hat{u}_{ik}(x)\rightarrow Q(x)$ in $L^{\infty}(\mathbb{R}^{3})$ as $k\rightarrow\infty$, we then deduce from \eqref{equ4-16}, \eqref{equ4-22} and \eqref{equ5-1} that
		\begin{equation*}
			\begin{split}
				0&=\int_{\mathbb{R}^{3}}Q^{p-1}\hat{\eta}_{0}dx\\
				&=\int_{\mathbb{R}^{3}}Q^{p-1}\left[ b_{0}\left(Q+\frac{p-2}{2s}x\cdot\nabla Q\right)+\sum_{i=1}^{3}b_{i}\frac{\partial Q}{\partial x_{i}}\right] dx\\
				&=b_{0}\int_{\mathbb{R}^{3}}Q^{p}dx+\frac{b_{0}(p-2)}{2sp}\int_{\mathbb{R}^{3}}x\cdot\nabla Q^{p}dx+\sum_{i=1}^{3}\frac{b_{i}}{p}\int_{\mathbb{R}^{3}}\frac{\partial Q^{p}}{\partial x_{i}}dx\\
				&=b_{0}\int_{\mathbb{R}^{3}}Q^{p}dx-\frac{3b_{0}(p-2)}{2sp}\int_{\mathbb{R}^{3}}Q^{p}dx\\
				&=\frac{2sp-3p+6}{2sp}b_{0}\int_{\mathbb{R}^{3}}Q^{p}dx,
			\end{split}
		\end{equation*}
		which implies that $b_{0}=0$ due to $2<p<\frac{4s}{3}+2<\frac{6}{3-2s}$. Substituting $b_{0}=0$ into \eqref{equ4-23}, one derive that
		\begin{equation*}
			\sum^{3}_{i=1}
			b_{i}\int_{\mathbb{R}^{3}}\frac{\partial^{2}V_{0}(x+y_{0})}{\partial x_{i}\partial x_{j}}Q^{2}dx=0.
		\end{equation*}
		Since $y_{0}$ is the unique and non-degenerate critical point of
		$H(y)$, then we can obtain that $b_{1}=b_{2}=b_{3}=0$, and Lemma \ref{lem4-4} is proved.
	\end{proof}

	Using the Lemmas established above, we now proceed to complete the proof of Theorem \ref{the3}.
	\begin{proof}[Proof of Theorem \ref{the3}] Let \(z_k\) denote a point such that $|\hat{\eta}_k(z_k)| = 1$. Lemma \ref{lem4-2} shows that $|z_k| \leq C$.  Therefore, it follows from \eqref{equ4-21} that $\hat{\eta}_{0}\not\equiv0$ on $\mathbb{R}^{3}$. However, Lemma \ref{lem4-4} yields that $\hat{\eta}_0 \equiv 0$, which contradicts our assumption that $u_{1k}\not\equiv u_{2k}$. This completes the proof of Theorem \ref{the3}.\end{proof}

	\appendix
	
	\section{Basic estimates} 
	In this appendix, we present the detailed proofs of several results used in the proof of Theorem \ref{the3}. The main difficulty in establishing these results stems from the nonlocality of the fractional Laplacian.
	\begin{lemma}\label{lem5-1}
		Let $\hat{u}_{ik}$ be as in \eqref{equ4-4}, then there exists a constant $C>0$, independent of $k$, such that
		\begin{equation}\label{equ5-1}
			\hat{u}_{ik}(x)+	|\nabla\hat{u}_{ik}(x)|\leq\frac{C}{1+|x|^{3+2s}}\quad \hbox{as $k\rightarrow\infty$}.
		\end{equation}
	\end{lemma}
	\begin{proof}
		Combining the Nash-Moser iteration method with the properties of the Bessel kernel, similar to \eqref{equ3-30}-\eqref{equ3-31}, one can conclude that
		\begin{equation}\label{equ5-2}
			\hat{u}_{ik}(x)\leq\frac{C}{1+|x|^{3+2s}}\quad \hbox{as $k\rightarrow\infty$}.
		\end{equation}
		Define $\bar{u}_{ik}(z):=\hat{u}_{ik}(z+x)$, then \eqref{equ4-7} shows that
		\begin{equation*}
			\begin{split}
				&(-\Delta)^{s}\bar{u}_{ik}(z)+\varepsilon_{k}^{2s}V(\varepsilon_{k}(z+x)+x_{1k})\bar{u}_{ik}(z)\\
				=&\varepsilon_{k}^{2s}\mu_{ik}\bar{u}_{ik}(z)+\frac{\varepsilon_{k}^{4s-3}}{a^{\ast}}\left(\int_{\mathbb{R}^{3}}\frac{|\bar{u}_{ik}(w)|^{2}}{|z-w|^{3-2s}}dw\right)\bar{u}_{ik}(z)+\bar{u}_{ik}^{p-1}(z) \quad \hbox{in $\mathbb{R}^3$}.
			\end{split}
		\end{equation*}
		Since $V(x)$ is a bounded potential, it then follows from \eqref{equ4-12} and \cite[Theorem 12.2.4]{C} that
		\begin{equation*}
			\|\bar{u}_{ik}(z)\|_{C^{2s}(B_{1}(0))}\leq C	\|\bar{u}_{ik}(z)\|_{L^{\infty}(B_{3}(0))}+\int_{|z|\geq3}\frac{|\bar{u}_{ik}(z)|}{(1+|z|)^{3+2s}}dz\quad \hbox{as $k\rightarrow\infty$}.
		\end{equation*}
		This further implies that
		\begin{equation}\label{equ5-3}
			\|\hat{u}_{ik}(y)\|_{C^{2s}(B_{1}(x))}\leq C	\|\hat{u}_{ik}(y)\|_{L^{\infty}(B_{3}(x))}+\int_{|z|\geq3}\frac{|\hat{u}_{ik}(x+z)|}{(1+|z|)^{3+2s}}dz\quad \hbox{as $k\rightarrow\infty$},
		\end{equation}
		where we have used the facts that
		\begin{equation*}
			\begin{split}
				\|\bar{u}_{ik}(z)\|_{C^{2s}(B_{1}(0))}=	\|\hat{u}_{ik}(y)\|_{C^{2s}(B_{1}(x))}\quad \hbox{and}\quad 		\|\bar{u}_{ik}(z)\|_{L^{\infty}(B_{3}(0))}=\|\hat{u}_{ik}(y)\|_{L^{\infty}(B_{3}(x))}.
			\end{split}
		\end{equation*}		
		For any fixed sufficiently large $R>0$, we observe that $|y|\geq|x|-|y-x|\geq|x|-3$ for $y\in B_{3}(x)$, using \eqref{equ5-2}, one derive that
		\begin{equation}\label{equ5-4}
			\begin{split}
				\|\hat{u}_{ik}(y)\|_{L^{\infty}(B_{3}(x))}&\leq\frac{C}{1+|y|^{3+2s}}\leq\frac{C}{1+(|x|-3)^{3+2s}}\\
				&\leq\frac{C}{\frac{1}{2}+\frac{1}{2}|x|^{3+2s}}\leq\frac{C}{1+|x|^{3+2s}}\quad \hbox{for $|x|\geq R$ as $k\rightarrow\infty$}.
			\end{split}
		\end{equation}
		Applying \eqref{equ5-2} again, it holds that
		\begin{equation}\label{equ5-5}
			\begin{split}
				&	\int_{|z|\geq3}\frac{|\hat{u}_{ik}(x+z)|}{(1+|z|)^{3+2s}}dz\leq C\int_{\mathbb{R}^{3}}\frac{|\hat{u}_{ik}(x+z)|}{(1+|z|)^{3+2s}}dz
				\leq C\int_{\mathbb{R}^{3}}\frac{1}{1+|x+z|^{3+2s}}\frac{1}{(1+|z|)^{3+2s}}dz\\
				=&C\int_{B^{c}_{\frac{|x|}{2}}(x)}\frac{1}{1+|x+z|^{3+2s}}\frac{1}{(1+|z|)^{3+2s}}dz+C\int_{B_{\frac{|x|}{2}}(x)}\frac{1}{1+|x+z|^{3+2s}}\frac{1}{(1+|z|)^{3+2s}}dz.
			\end{split}
		\end{equation}
		By \cite[Lemma A.3]{GYX}, one has
		\begin{equation}\label{equ5-6}
			\begin{split}
				&	\left| \int_{B^{c}_{\frac{|x|}{2}}(x)}\frac{1}{1+|x+z|^{3+2s}}\frac{1}{(1+|z|)^{3+2s}}dz\right| \\
				\leq& C\int_{B^{c}_{\frac{|x|}{2}}(x)}\frac{1}{|x+z|^{3+2s}}\frac{1}{(1+|z|)^{3+2s}}dz\\
				\leq& C\left( \frac{1}{(1+|x|)^{3+2s}}+\frac{1}{(1+|x|)^{3+2s}}\frac{1}{|x|^{2s}}\right)\\
				\leq&\frac{C}{(1+|x|)^{3+2s}}\leq\frac{C}{1+|x|^{3+2s}}\quad \hbox{for $|x|\geq R$ as $k\rightarrow\infty$}.
			\end{split}
		\end{equation}
		Furthermore, a direct calculation implies that		
		\begin{equation}\label{equ5-7}
			\begin{split}
				&	\left| \int_{B_{\frac{|x|}{2}}(x)}\frac{1}{1+|x+z|^{3+2s}}\frac{1}{(1+|z|)^{3+2s}}dz\right| \\
				\leq& C\int_{B_{\frac{|x|}{2}}(x)}\frac{1}{(1+|z|)^{3+2s}}\frac{1}{1+|x+z|^{3+2s}}dz\\
				\leq&\frac{C}{\left( 1+\frac{|x|}{2}\right)^{3+2s} }\int_{B_{\frac{|x|}{2}}(x)}\frac{1}{1+|x+z|^{3+2s}}dz\\
				\leq&\frac{C}{1+|x|^{3+2s}}\int_{\mathbb{R}^{3}}\frac{1}{1+|y|^{3+2s}}dy\leq\frac{C}{1+|x|^{3+2s}}\quad \hbox{as $k\rightarrow\infty$},
			\end{split}
		\end{equation}
		where we have used the fact that $|z|\geq|x|-|x-z|\geq|x|-\frac{|x|}{2}=\frac{|x|}{2}$ for $z\in B_{\frac{|x|}{2}}(x)$. Combining \eqref{equ5-3}-\eqref{equ5-7}, we observe that
		\begin{equation}\label{equ5-8}
			|\nabla\hat{u}_{ik}(x)|\leq\|\hat{u}_{ik}(y)\|_{C^{2s}(B_{1}(x))}\leq\frac{C}{1+|x|^{3+2s}}\quad \hbox{for $|x|\geq R>3$ as $k\rightarrow\infty$}.
		\end{equation}
		Moreover, we also have
		\begin{equation}\label{equ5-9}
			|\nabla\hat{u}_{ik}(x)|\leq\frac{1+R^{3+2s}}{1+|x|^{3+2s}}	\|\nabla\hat{u}_{ik}\|_{L^{\infty}(\mathbb{R}^{3})}\leq\frac{C}{1+|x|^{3+2s}}\quad \hbox{for $|x|\leq R$ as $k\rightarrow\infty$}.
		\end{equation}
		Thus, \eqref{equ5-8} and \eqref{equ5-9} imply \eqref{equ5-1}, and the proof of Lemma \ref{lem5-1} is completed. 
	\end{proof}

	To establish the decay estimate of $\hat{\eta}_{k}(x)$, we develop the following comparison principle.	
	
	\begin{lemma}\label{lem5-2}
		Suppose there exists a constant $C_{0}>0$ such that $|g(x)|\leq\frac{C_{0}}{|x|^{3+2s}}$, and let $R>0$ be sufficiently large so that  $\inf\limits_{|x|>R}\rho(x):=\rho_{0}>0$. If $u\in L^{\infty}(\mathbb{R}^{3})$ is a weak solution to
		\begin{equation}\label{equ5-10}
			\left\{
			\begin{array}{ll}
				(-\Delta)^{s}u+\rho(x)u=g(x) \quad \hbox{in $\mathbb{R}^{3}$},\\
				\lim_{|x|\rightarrow\infty}u(x)=0.
			\end{array}
			\right.
		\end{equation}
		Then the following decay estimate holds
		\begin{equation}\label{equ5-11}
			|u(x)|\leq\frac{C}{1+|x|^{3+2s}}\quad \hbox{for $x\in \mathbb{R}^{3}$}.
		\end{equation}
	\end{lemma}
	\begin{proof}
		For $s\in (0, 1)$ and $\lambda>0$, let
		\begin{equation*}
			\varphi_{\lambda}(x):=
			\left\{
			\begin{array}{ll}
				\lambda^{-3-2s} &\quad \hbox{$|x|<\lambda$},\\
				|x|^{-3-2s}&\quad \hbox{$|x|\geq\lambda$}.
			\end{array}
			\right.
		\end{equation*}
		We first claim that for any $\lambda>0$, there exist $R_{0}>0$ and $C_{1}>0$ such that
		\begin{equation}\label{equ5-12}
			-C_{1}\lambda^{-2s}\varphi_{\lambda}(x)\leq(-\Delta)^{s}\varphi_{\lambda}(x)\leq-C_{1}^{-1}\lambda^{-2s}\varphi_{\lambda}(x)\quad \hbox{for $x\in B^{c}_{R_{0}}(0)$}.
		\end{equation}
		In fact, assume that $\lambda>0$ and $x\in \mathbb{R}^{3}$ satisfies $|x|>3\lambda$, one has
		\begin{equation}\label{equ5-13}
			(-\Delta)^{s}\varphi_{\lambda}(x)=-\frac{C_{s}}{2}\int_{\mathbb{R}^{3}}\frac{\varphi_{\lambda}(x+y)+\varphi_{\lambda}(x-y)-2\varphi_{\lambda}(x)}{|y|^{3+2s}}dy.
		\end{equation}		
		For $y\in B_{\frac{|x|}{3}}(0)$, we have that $|x\pm y|\geq\lambda$. By Taylor expansion, we obtain that
		\begin{equation}\label{equ5-14}
			\begin{split}
				&	\left| \int_{B_{\frac{|x|}{3}}(0)}\frac{\varphi_{\lambda}(x+y)+\varphi_{\lambda}(x-y)-2\varphi_{\lambda}(x)}{|y|^{3+2s}}dy\right| \\
				=&|x|^{-3-4s}\left| \int_{B_{\frac{1}{3}}(0)}\frac{\left| y+\frac{x}{|x|}\right|^{-3-2s}+\left| y-\frac{x}{|x|}\right|^{-3-2s}-2}{|y|^{3+2s}}dy\right| \\
				\leq&C_{2}|x|^{-3-4s}\int_{B_{\frac{1}{3}}(0)}\frac{|y|^{2}}{|y|^{3+2s}}dy\leq C_{3}|x|^{-3-4s}.
			\end{split}
		\end{equation}
		For $y\in B_{\frac{|x|}{3}}(x)\backslash B_{\lambda}(x)$, we see that $|x+y|\geq|x-y|\geq\lambda$. Moreover, we also have $\left| y+\frac{x}{|x|}\right|\geq\left| y-\frac{x}{|x|}\right|$ for $y\in B_{\frac{1}{3}}\left( \frac{x}{|x|}\right)\backslash B_{\frac{\lambda}{|x|}}\left( \frac{x}{|x|}\right)$ and $|y|\geq\frac{2}{3}$ for $y\in B_{\frac{1}{3}}\left( \frac{x}{|x|}\right)$. Then we obtain that
		\begin{equation}\label{equ5-15}
			\begin{split}
				&\int_{B_{\frac{|x|}{3}}(x)\backslash B_{\lambda}(x)}\frac{\varphi_{\lambda}(x+y)+\varphi_{\lambda}(x-y)-2\varphi_{\lambda}(x)}{|y|^{3+2s}}dy \\
				=&\int_{B_{\frac{|x|}{3}}(x)\backslash B_{\lambda}(x)}\frac{|x+y|^{-3-2s}+|x-y|^{-3-2s}-2|x|^{-3-2s}}{|y|^{3+2s}}\\
				=&|x|^{-3-4s}\int_{B_{\frac{1}{3}}\left( \frac{x}{|x|}\right)\backslash B_{\frac{\lambda}{|x|}}\left( \frac{x}{|x|}\right)}\frac{\left| y+\frac{x}{|x|}\right|^{-3-2s}+\left| y-\frac{x}{|x|}\right|^{-3-2s}-2}{|y|^{3+2s}}dy\\
				\leq&C_{4}|x|^{-3-4s}\int_{B_{\frac{1}{3}}\left( \frac{x}{|x|}\right)\backslash B_{\frac{\lambda}{|x|}}\left( \frac{x}{|x|}\right)}\left|y-\frac{x}{|x|} \right|^{-3-2s} dy\leq C_{5}\lambda^{-2s}|x|^{-3-2s}.
			\end{split}
		\end{equation}
		Moreover, note that $|y|\leq\frac{4}{3}$ for $y\in B_{\frac{1}{3}}\left( \frac{x}{|x|}\right)$, we have
		\begin{equation}\label{equ5-16}
			\begin{split}
				&\int_{B_{\frac{|x|}{3}}(x)\backslash B_{\lambda}(x)}\frac{\varphi_{\lambda}(x+y)+\varphi_{\lambda}(x-y)-2\varphi_{\lambda}(x)}{|y|^{3+2s}}dy\\
				\geq&|x|^{-3-4s}\left( \int_{B_{\frac{1}{3}}\left( \frac{x}{|x|}\right)\backslash B_{\frac{\lambda}{|x|}}\left( \frac{x}{|x|}\right)}\frac{\left|y-\frac{x}{|x|} \right|^{-3-2s}}{|y|^{3+2s}}dy-\int_{B_{\frac{1}{3}}\left( \frac{x}{|x|}\right)}\frac{2}{|y|^{3+2s}}dy\right) \\
				\geq& C_{6}|x|^{-3-4s}\int_{B_{\frac{1}{3}}\left( \frac{x}{|x|}\right)\backslash B_{\frac{\lambda}{|x|}}\left( \frac{x}{|x|}\right)}\left|y-\frac{x}{|x|} \right|^{-3-2s}dy-C_{7}|x|^{-3-4s}\\
				\geq&C_{8}\lambda^{-2s}|x|^{-3-2s}-C_{9}|x|^{-3-4s}.
			\end{split}
		\end{equation}
		For $y\in B_{\lambda}(x)$, we can see that $|x+y|>\lambda>|x-y|$ and $|y|\geq|x|-\lambda\geq\frac{2|x|}{3}$. Then we have
		\begin{equation}\label{equ5-17}
			\begin{split}
				&\int_{ B_{\lambda}(x)}\frac{\varphi_{\lambda}(x+y)+\varphi_{\lambda}(x-y)-2\varphi_{\lambda}(x)}{|y|^{3+2s}}dy\\
				=&\int_{ B_{\lambda}(x)}\frac{|x+y|^{-3-2s}+\lambda^{-3-2s}-2|x|^{-3-2s}}{|y|^{3+2s}}dy\\
				\leq&2\int_{ B_{\lambda}(x)}\frac{\lambda^{-3-2s}}{|y|^{3+2s}}dy\leq C_{10}\lambda^{-2s}|x|^{-3-2s},
			\end{split}
		\end{equation}
		and
		\begin{equation}\label{equ5-18}
			\begin{split}
				&\int_{ B_{\lambda}(x)}\frac{|x+y|^{-3-2s}+\lambda^{-3-2s}-2|x|^{-3-2s}}{|y|^{3+2s}}dy\\
				\geq&\int_{ B_{\lambda}(x)}\frac{-2|x|^{-3-2s}}{|y|^{3+2s}}dy\geq-C_{11}\lambda^{3}(|x|-\lambda)^{-3-2s}|x|^{-3-2s}\geq-C_{12}|x|^{-3-4s}.
			\end{split}
		\end{equation}
		The same method allows us to estimate the integral domains in $B_{\frac{|x|}{3}}(-x)\backslash B_{\lambda}(-x)$ and $B_{\lambda}(-x)$. For $y\in \mathbb{R}^{3}\backslash \left[B_{\frac{|x|}{3}}(0)\cup B_{\frac{|x|}{3}}(x)\cup B_{\frac{|x|}{3}}(-x)\right]$, we have $|x\pm y|\geq\frac{|x|}{3}$ for $|x|>3\lambda$. Thus
		\begin{equation}\label{equ5-19}
			\begin{split}
				&\left|\int_{\mathbb{R}^{3}\backslash \left[B_{\frac{|x|}{3}}(0)\cup B_{\frac{|x|}{3}}(x)\cup B_{\frac{|x|}{3}}(-x)\right]}\frac{|x+y|^{-3-2s}+|x-y|^{-3-2s}-2|x|^{-3-2s}}{|y|^{3+2s}}dy\right|\\
				\leq&C_{13}|x|^{-3-2s}\int_{\mathbb{R}^{3}\backslash B_{\frac{|x|}{3}}(0)}\frac{1}{|y|^{3+2s}}dy\leq C_{14}|x|^{-3-4s}.
			\end{split}
		\end{equation}
		Combining \eqref{equ5-14}-\eqref{equ5-19}, it then follows from the fact $|x|>3\lambda$ that
		\begin{equation*}
			C_{15}^{-1}\lambda^{-2s}|x|^{-3-2s}-C_{15}|x|^{-3-4s}\leq\int_{\mathbb{R}^{3}}\frac{\varphi_{\lambda}(x+y)+\varphi_{\lambda}(x-y)-2\varphi_{\lambda}(x)}{|y|^{3+2s}}dy\leq C_{15}\lambda^{-2s}|x|^{-3-2s}.
		\end{equation*}
		Let $R_{0}>3\lambda$ such that $C_{15}^{-1}\lambda^{-2s}-C_{15}|x|^{-2s}\geq\frac{1}{2}C_{15}^{-1}\lambda^{-2s}$ for $|x|\geq R_{0}$, then \eqref{equ5-13} shows that \eqref{equ5-12} holds, and claim is proved. We now prove \eqref{equ5-11}. Since $u\in L^{\infty}(\mathbb{R}^{3})$, then there exists $M>0$ such that $|u(x)|\leq M$. Take $R\geq R_{0}>3\lambda$ large enough, then $\rho(x)\geq\rho_{0}>0$ for $|x|\geq R$. Choose $\lambda>\left( \frac{C_{1}}{\rho_{0}}\right)^{\frac{1}{2s}}$, where $C_{1}>0$ is given by \eqref{equ5-12}. Define 
		\begin{equation*}
			v(x):=\frac{C_{16}}{|x|^{3+2s}}\quad \mathrm {with} \quad C_{16}\geq\max \left\lbrace\frac{C_{0}}{\rho_{0}-C_{1}\lambda^{-2s}},  R^{3+2s}M \right\rbrace.
		\end{equation*}
		It then follows from \eqref{equ5-12} that
		\begin{equation}\label{equ5-20}
			(-\Delta)^{s}v+\rho(x)v\geq\frac{-C_{1}C_{16}\lambda^{-2s}}{|x|^{3+2s}}+\frac{\rho_{0}C_{16}}{|x|^{3+2s}}\geq\frac{C_{0}}{|x|^{3+2s}}\geq|g(x)|\quad \hbox{for $|x|\geq R$}.
		\end{equation}
		Together with \eqref{equ5-10}, this yields that
		\begin{equation}\label{equ5-21}
			(-\Delta)^{s}\dot{w}(x)+\rho(x)\dot{w}(x)\leq0 \quad \hbox{for $|x|\geq R$},
		\end{equation}
		where $\dot{w}(x):=|u(x)|-v(x)$. Note that
		\begin{equation*}
			v(x)=\frac{C_{16}}{|x|^{3+2s}}\geq\frac{C_{16}}{R^{3+2s}}\geq M\geq|u(x)|\quad \hbox{for $|x|\leq R$},
		\end{equation*}
		then we deduce that $\dot{w}(x)\leq0$ in $B_{R}(0)$. We now claim that $\dot{w}(x)\leq0$ in $B^{c}_{R}(0)$. Otherwise, there exists some global maximum point $x_{1}\in B^{c}_{R}(0)$ satisfying $\dot{w}(x_{1})>0$. It then follows from \eqref{equ5-21} that
		\begin{equation*}
			(-\Delta)^{s}\dot{w}(x_{1})\leq-\rho(x_{1})\dot{w}(x_{1})\leq-\rho_{0}\dot{w}(x_{1})<0.
		\end{equation*}
		However, we have
		\begin{equation*}
			(-\Delta)^{s}\dot{w}(x_{1})=\int_{\mathbb{R}^{3}}\frac{\dot{w}(x_{1})-\dot{w}(y)}{|x-y|^{3+2s}}dy\geq0.
		\end{equation*}
		This is a contradiction. Hence, we infer that $|u(x)|\leq v(x)$ in $\mathbb{R}^{3}$, which further implies that \eqref{equ5-11} holds. The proof of Lemma \ref{lem5-2} is completed.
	\end{proof}
	\begin{lemma}\label{lem5-3}
		Let $\hat{u}_{ik}$ be as in \eqref{equ4-4}, there holds that
		\begin{equation*}
			\int_{\mathbb{R}^{3}}(-\Delta)^{s} \hat{u}_{ik}\frac{\partial\hat{u}_{ik}}{\partial x_{j}}dx=0.
		\end{equation*}
	\end{lemma}
	\begin{proof}
		We employ the $s$-harmonic extension technique following the approach in \cite{CC}. Deonte
		\begin{equation*}
			\tilde{\hat{u}}_{ik}(x,t)
			=\int_{\mathbb{R}^{3}}
			\mathcal{P}_{s}(x-z, t)
			\hat{u}_{ik}(z)dz \quad\hbox{$(x, t)\in\mathbb{R}^{4}_{+}$ and $i=1, 2$},
		\end{equation*} 
		of solves
		\begin{equation*}
			\begin{cases}
				&	-{\rm div}(t^{1-2s}
				\nabla \tilde{\hat{u}}_{ik})
				=0\quad\hbox{in $\mathbb{R}_{+}^{4}$},\\
				&\tilde{\hat{u}}_{ik}(x,0)
				=\hat{u}_{ik}(x)
				\quad\hbox{on $\mathbb{R}^{3}$},\\
				&-\lim\limits_{t\rightarrow0}t^{1-2s}\partial_{t}\tilde{\hat{u}}_{ik}(x,t)=(-\Delta)^{s}\hat{u}_{ik}\quad\hbox{on $\mathbb{R}^{3}$},
			\end{cases}
		\end{equation*}
		where $i=1, 2$, and
		\begin{equation*}
			\mathbb{R}^{4}_{+}
			:=\{(x, t): x\in\mathbb{R}^{3}, t>0\},
			~~
			\mathcal{P}_{s}:=
			\beta(s)
			\frac{t^{2s}}{(|x|^{2}+t^{2})^{\frac{3+2s}{2}}},
		\end{equation*}
		$\beta(s)$ is a constant such that 
		$\int_{\mathbb{R}^3}
		\mathcal{P}_{s}(x, 1)
		\mathrm{d}x=1$. For $R>1$, set
		\begin{equation*}
			\begin{aligned}
				\Omega_{R, \frac{1}{R}}^{+}
				:=&\left\{(x, t)\in \mathbb{R}^{3}\times\left[ \frac{1}{R}, \infty\right):\: |x|^{2}+t^{2}\leq R^{2}
				\right\}\\
				\partial\Omega_{R, \frac{1}{R}}^{1}
				:=&\left\{(x, t)\in \mathbb{R}^{3}\times\left\lbrace t=\frac{1}{R} \right\rbrace :\: |x|^{2}\leq R^{2}-\frac{1}{R^{2}}
				\right\}\\
				\partial\Omega_{R, \frac{1}{R}}^{2}
				:=&\left\{(x, t)\in \mathbb{R}^{3}\times\left[ \frac{1}{R}, \infty\right):\: |x|^{2}+t^{2}= R^{2}
				\right\}.
			\end{aligned}
		\end{equation*}		
		Clearly, one has
		\begin{equation*}
			\partial\Omega_{R, \frac{1}{R}}^{+}=\partial\Omega_{R, \frac{1}{R}}^{1}\cup\partial\Omega_{R, \frac{1}{R}}^{2}\quad \hbox{and}\quad 
			\nu=
			\left\{
			\begin{array}{ll}
				(0, 0, 0, -1) &\quad \hbox{on $\partial\Omega_{R, \frac{1}{R}}^{1}$},\\
				\frac{(x, t)}{R}&\quad \hbox{on $\partial\Omega_{R, \frac{1}{R}}^{2}$},
			\end{array}
			\right.
		\end{equation*}
		where $\nu$ denotes the unit outward normal vector of $\partial\Omega_{R, \frac{1}{R}}^{+}$. Direct calculation shows that
		\begin{equation*}
			{\rm div}\left( t^{1-2s}\nabla \tilde{\hat{u}}_{ik}\frac{\partial \tilde{\hat{u}}_{ik}}{\partial x_{j}}\right) ={\rm div}\left( t^{1-2s}\nabla \tilde{\hat{u}}_{ik}\right) \frac{\partial \tilde{\hat{u}}_{ik}}{\partial x_{j}}+t^{1-2s}\nabla \tilde{\hat{u}}_{ik}\cdot\nabla\left( \frac{\partial \tilde{\hat{u}}_{ik}}{\partial x_{j}}\right).
		\end{equation*}
		Applying the divergence theorem, it follows that		
		\begin{equation}\label{equ5-22}
			\int_{\Omega_{R, \frac{1}{R}}^{+}}t^{1-2s}\nabla \tilde{\hat{u}}_{ik}\cdot\nabla\left( \frac{\partial \tilde{\hat{u}}_{ik}}{\partial x_{j}}\right)dxdt=\int_{\partial\Omega_{R, \frac{1}{R}}^{+}}t^{1-2s}\frac{\partial \tilde{\hat{u}}_{ik}}{\partial\nu}\frac{\partial \tilde{\hat{u}}_{ik}}{\partial x_{j}}dS,
		\end{equation}			
		where we have used the fact that ${\rm div}(t^{1-2s}
		\nabla \tilde{\hat{u}}_{ik})
		=0$. Similar to the proof of \cite[Proposition 4.1]{CCC}, we know that $\tilde{\hat{u}}_{ik}\in C^{1}(\mathbb{R}_{+}^{4})$. Simple calculation yields that
		\begin{equation*}
			\nabla \tilde{\hat{u}}_{ik}\cdot\nabla\left( \frac{\partial \tilde{\hat{u}}_{ik}}{\partial x_{j}}\right)=\frac{1}{2}\frac{\partial \left( \left| \nabla\tilde{\hat{u}}_{ik}\right| ^{2}\right) }{\partial x_{j}},
		\end{equation*}
		from which and \eqref{equ5-22} we deduce that
		\begin{equation}\label{equ5-23}
			\frac{1}{2}\int_{\Omega_{R, \frac{1}{R}}^{+}}t^{1-2s}\frac{\partial \left( \left| \nabla\tilde{\hat{u}}_{ik}\right| ^{2}\right) }{\partial x_{j}}dxdt=\int_{\partial\Omega_{R, \frac{1}{R}}^{+}}t^{1-2s}\frac{\partial \tilde{\hat{u}}_{ik}}{\partial\nu}\frac{\partial \tilde{\hat{u}}_{ik}}{\partial x_{j}}dS.
		\end{equation}			
		Integrating by parts, one derive that
		\begin{equation}\label{equ5-24}
			\begin{split}
				\int_{\Omega_{R, \frac{1}{R}}^{+}}t^{1-2s}\frac{\partial \left( \left| \nabla\tilde{\hat{u}}_{ik}\right| ^{2}\right) }{\partial x_{j}}dxdt&=\int_{\partial\Omega_{R, \frac{1}{R}}^{+}}t^{1-2s}\left|\nabla\tilde{\hat{u}}_{ik} \right| ^{2}\nu_{j}dS-\int_{\Omega_{R, \frac{1}{R}}^{+}}\left| \nabla\tilde{\hat{u}}_{ik}\right| ^{2}\partial_{x_{j}}(t^{1-2s})dxdt\\
				&=\int_{\partial\Omega_{R, \frac{1}{R}}^{2}}t^{1-2s}\left|\nabla\tilde{\hat{u}}_{ik} \right| ^{2}\nu_{j}dS,
			\end{split}
		\end{equation}
		where we have used the facts that $\partial_{x_{j}}(t^{1-2s})=0$
		and $\nu_{j}=0$ on $\partial\Omega_{R, \frac{1}{R}}^{1}$. Note that
		\begin{equation*}
			\left|(x, t)\right|=R,\quad \left| \nabla\tilde{\hat{u}}_{ik}(x, t)\right|\leq\frac{C}{\left|(x, t)\right|^{3+2s}},  \quad  \hbox{and}\quad t\geq\frac{1}{R}\quad \hbox{on $\partial\Omega_{R, \frac{1}{R}}^{2}$},
		\end{equation*}
		then we have
		\begin{equation}\label{equ5-25}
			\left| \int_{\partial\Omega_{R, \frac{1}{R}}^{2}}t^{1-2s}\left|\nabla\tilde{\hat{u}}_{ik} \right| ^{2}\nu_{j}dS\right| \leq CR^{2s-1}R^{-6-4s}R^{3}\rightarrow0\quad \hbox{as $R\rightarrow\infty$}.
		\end{equation}	
		We proceed to estimate the RHS of \eqref{equ5-23}. Clearly
		\begin{equation}\label{equ5-26}
			\begin{split}
				&	\int_{\partial\Omega_{R, \frac{1}{R}}^{+}}t^{1-2s}\frac{\partial \tilde{\hat{u}}_{ik}}{\partial\nu}\frac{\partial \tilde{\hat{u}}_{ik}}{\partial x_{j}}dS\\
				=&\int_{\partial\Omega_{R, \frac{1}{R}}^{1}}t^{1-2s}\frac{\partial \tilde{\hat{u}}_{ik}}{\partial\nu}\frac{\partial \tilde{\hat{u}}_{ik}}{\partial x_{j}}dS+\int_{\partial\Omega_{R, \frac{1}{R}}^{2}}t^{1-2s}\frac{\partial \tilde{\hat{u}}_{ik}}{\partial\nu}\frac{\partial \tilde{\hat{u}}_{ik}}{\partial x_{j}}dS:=I_{1}+I_{2}.
			\end{split}
		\end{equation}
		Since $\frac{\partial \tilde{\hat{u}}_{ik}}{\partial\nu}=-\frac{\partial \tilde{\hat{u}}_{ik}}{\partial t}$ on $\partial\Omega_{R, \frac{1}{R}}^{1}$, we thus obtain that
		\begin{equation*}
			I_{1}=-\left( \frac{1}{R}\right)^{1-2s}\int_{|x|\leq\sqrt{R^{2}-\frac{1}{R^{2}}}}\frac{\partial \tilde{\hat{u}}_{ik}\left( x, \frac{1}{R}\right) }{\partial t}\frac{\partial \tilde{\hat{u}}_{ik}\left( x, \frac{1}{R}\right) }{\partial x_{j}}dx.
		\end{equation*}	
		Taking $R\rightarrow\infty$, one can see that
		\begin{equation}\label{equ5-27}
			\lim_{R\rightarrow\infty}I_{1}=	\int_{\mathbb{R}^{3}}(-\Delta)^{s} \hat{u}_{ik}\frac{\partial\hat{u}_{ik}}{\partial x_{j}}dx.
		\end{equation}
		Note that
		\begin{equation*}
			\begin{aligned}
				\left| \frac{\partial \tilde{\hat{u}}_{ik}}{\partial\nu}\right| =&\left| \nabla\tilde{\hat{u}}_{ik}\cdot\nu\right| \leq\left| \nabla\tilde{\hat{u}}_{ik}\right| \leq\frac{C}{\left|(x, t)\right|^{3+2s}}\quad \hbox{on $\partial\Omega_{R, \frac{1}{R}}^{2}$},\\
				\left| \frac{\partial \tilde{\hat{u}}_{ik}}{\partial x_{j}}\right| \leq&\left| \nabla\tilde{\hat{u}}_{ik}\right| \leq\frac{C}{\left|(x, t)\right|^{3+2s}}\quad \hbox{and}\quad t\geq\frac{1}{R}\quad  \hbox{on $\partial\Omega_{R, \frac{1}{R}}^{2}$}.
			\end{aligned}
		\end{equation*}		
		We thus derive that
		\begin{equation}\label{equ5-28}
			\left| 	I_{2}\right| \leq CR^{2s-1}R^{-6-4s}R^{3}\rightarrow0\quad \hbox{as $R\rightarrow\infty$}.
		\end{equation}
		Combining \eqref{equ5-23}-\eqref{equ5-28}, we can conclude that
		\begin{equation*}
			\int_{\mathbb{R}^{3}}(-\Delta)^{s} \hat{u}_{ik}\frac{\partial\hat{u}_{ik}}{\partial x_{j}}dx=0,
		\end{equation*}	
		this therefore completes the proof of Lemma \ref{lem5-3}.
	\end{proof}
	
	\begin{lemma}\label{lem5-4}
		Assume that $x_{1k}$ is the unique global maximum point of $u_{1k}$, and $\hat{u}_{ik}$ is given in \eqref{equ4-4}, then there holds that
		\begin{equation*}
			\int_{\mathbb{R}^{3}}(-\Delta)^{s}\hat{u}_{ik}[(x-x_{1k})\cdot\nabla\hat{u}_{ik}]dx=\frac{2s-3}{2} \int_{\mathbb{R}^3} \left| (-\Delta)^{\frac{s}{2}}\hat{u}_{ik}\right| ^2dx.
		\end{equation*}
	\end{lemma}
	\begin{proof}
		Inspired by Lemma \ref{lem5-3}, for $R > 1$, we define 
		\begin{align*}
			D^+_{R, \frac{1}{R}} &:= \left\lbrace (x,t) \in \mathbb{R}^3 \times \left[\frac{1}{R} , \infty\right) : |x - x_{1k}|^2 + t^2 \le R^2\right\rbrace  \\
			\partial D^1_{R, 1/R} &:= \left\lbrace (x,t) \in \mathbb{R}^3 \times \left\lbrace t =\frac{1}{R} \right\rbrace  : |x - x_{1k}|^2 \le R^2 - 1/R^2\right\rbrace , \\
			\partial D^2_{R, 1/R} &:= \left\lbrace (x,t) \in \mathbb{R}^3 \times \left[\frac{1}{R}, \infty\right)  : |x - x_{1k}|^2 + t^2 = R^2\right\rbrace .
		\end{align*}
		In the same way, we have
		\begin{equation*}
			\partial D_{R, \frac{1}{R}}^{+}=\partial D_{R, \frac{1}{R}}^{1}\cup\partial D_{R, \frac{1}{R}}^{2}\quad \hbox{and}\quad 
			\nu=
			\left\{
			\begin{array}{ll}
				(0, 0, 0, -1) &\quad \hbox{on $\partial D_{R, \frac{1}{R}}^{1}$},\\
				\frac{(x-x_{1k}, t)}{R}&\quad \hbox{on $\partial D_{R, \frac{1}{R}}^{2}$},
			\end{array}
			\right.
		\end{equation*}
		Let $X = (x - x_{1k}, t)$ and define
		\begin{equation*}
			\mathbf{P} := t^{1-2s} \left[ (X \cdot \nabla \tilde{\hat{u}}_{ik}) \nabla \tilde{\hat{u}}_{ik} - \frac{1}{2} |\nabla \tilde{\hat{u}}_{ik}|^2 X \right] + \frac{3-2s}{2} t^{1-2s} \tilde{\hat{u}}_{ik} \nabla \tilde{\hat{u}}_{ik}.
		\end{equation*}
		Direct calculation shows that
		\begin{equation*}
			\text{div}  \mathbf{P} = \left( s - \frac{3}{2}  + \frac{3-2s}{2} \right) t^{1-2s} |\nabla \tilde{\hat{u}}_{ik}|^2 = 0,
		\end{equation*}
		where we have used the facts ${\rm div}(t^{1-2s}
		\nabla \tilde{\hat{u}}_{ik})
		=0$, and
		\begin{align*}
			\text{div} \left( t^{1-2s} \left[ (X \cdot \nabla \tilde{\hat{u}}_{ik}) \nabla \tilde{\hat{u}}_{ik}- \frac{1}{2} |\nabla \tilde{\hat{u}}_{ik}|^2 X \right] \right) &= \left(s - \frac{3}{2}\right) t^{1-2s} |\nabla \tilde{\hat{u}}_{ik}|^2, \\
			\text{div} \left( \frac{3-2s}{2} t^{1-2s} \tilde{\hat{u}}_{ik} \nabla \tilde{\hat{u}}_{ik} \right) &= \frac{3-2s}{2} t^{1-2s} |\nabla \tilde{\hat{u}}_{ik}|^2.
		\end{align*}
		By divergence Theorem, we have
		\begin{equation*}
			\int_{\partial D^+_{R, \frac{1}{R}}} \mathbf{P} \cdot \nu \, dS =  \int_{D^+_{R, \frac{1}{R}}} \text{div}\mathbf{P} \, dxdt=0.
		\end{equation*}	
		The boundary integral can be divided into two parts as follows
		\begin{equation*}
			\int_{\partial D^+_{R, \frac{1}{R}}} \mathbf{P} \cdot \nu \, dS = \int_{\partial D^1_{R, \frac{1}{R}}} \mathbf{P} \cdot \nu \, dS + \int_{\partial D^2_{R, \frac{1}{R}}} \mathbf{P} \cdot \nu \, dS.
		\end{equation*}
		Since $\nu = (0, 0, 0, -1)$ on $\partial D_{R, \frac{1}{R}}^1$, we then deduce from the definition of $\mathbf{P}$ that
		\begin{equation*}
			\mathbf{P} \cdot \nu = -t^{1-2s} \left[ (X \cdot \nabla  \tilde{\hat{u}}_{ik}) \partial_t  \tilde{\hat{u}}_{ik} - \frac{1}{2} |\nabla  \tilde{\hat{u}}_{ik}|^2 t \right] -\frac{3-2s}{2} t^{1-2s}  \tilde{\hat{u}}_{ik} \partial_t  \tilde{\hat{u}}_{ik}.
		\end{equation*}
		As $R\to\infty$, $t\to 0^+$, and  $X = (x-x_0, t) \to (x-x_0, 0)$, it is easy to check that
		\begin{equation*}
			\begin{split}
				\lim_{t \to 0^{+}} - (X \cdot \nabla \tilde{\hat{u}}_{ik}) (t^{1-2s} \partial_t \tilde{\hat{u}}_{ik}) 
				&= - ((x-x_0) \cdot \nabla_x \hat{u}_{ik}) \cdot \left( -(-\Delta)^s \hat{u}_{ik} \right) \notag \\
				&=((x-x_0) \cdot \nabla \hat{u}_{ik}) (-\Delta)^s \hat{u}_{ik}.
			\end{split}
		\end{equation*}
		Since $s < 1$, one can see that $\lim_{t \to 0^{+}} \frac{1}{2} t^{2-2s} |\nabla \tilde{\hat{u}}_{ik}|^2 = 0$. Moreover, we have 
		\begin{align*}
			\lim_{t \to 0^{+}} - \frac{3-2s}{2}  \tilde{\hat{u}}_{ik} (t^{1-2s} \partial_t \tilde{\hat{u}}_{ik}) 
			= - \frac{3-2s}{2} \hat{u}_{ik} \cdot \left( - (-\Delta)^s \hat{u}_{ik} \right)  
			=  \frac{3-2s}{2} \hat{u}_{ik} (-\Delta)^s\hat{u}_{ik}.
		\end{align*}  
		Using the above facts, we obtain
		\begin{align*}
			\lim_{R \to \infty} \int_{\partial D_{R, \frac{1}{R}}^1} \mathbf{P} \cdot \nu dS= \int_{\mathbb{R}^3} \left[\left((x-x_0) \cdot \nabla \hat{u}_{ik}\right) (-\Delta)^s \hat{u}_{ik} + \frac{3-2s}{2} \hat{u}_{ik}  (-\Delta)^s \hat{u}_{ik} \right] dx.
		\end{align*}  
		Note that
		\begin{equation*}
			|X \cdot \nabla \tilde{\hat{u}}_{ik}| \le R|\nabla \tilde{\hat{u}}_{ik}|\quad   \hbox{and}\quad t\geq\frac{1}{R}\quad \hbox{on $\partial D_{R, \frac{1}{R}}^{2}$},
		\end{equation*}	
		then using the decay property of $| \tilde{\hat{u}}_{ik}|$
		and $|\nabla  \tilde{\hat{u}}_{ik}|$, we proceed with the integral estimation
		\begin{align*}
			&	\left| \int_{\partial D_{R, \frac{1}{R}}^2} \mathbf{P} \cdot \nu \, dS \right| \\
			= & \left| \int_{\partial D_{R, \frac{1}{R}}^2} \left( t^{1-2s} \left[ \frac{(X \cdot \nabla \tilde{\hat{u}}_{ik})^2}{R} - \frac{R}{2} |\nabla \tilde{\hat{u}}_{ik}|^2 \right] + \frac{3-2s}{2R} t^{1-2s} \tilde{\hat{u}}_{ik} (X \cdot \nabla \tilde{\hat{u}}_{ik}) \right) dS \right| \\
			\le& \int_{\partial D_{R, \frac{1}{R}}^2} t^{1-2s} \left( \frac{(R^{2} |\nabla \tilde{\hat{u}}_{ik}|)^2}{R} + \frac{R}{2} |\nabla \tilde{\hat{u}}_{ik}|^2 + C |\tilde{\hat{u}}_{ik}| |\nabla \tilde{\hat{u}}_{ik}| \right) dS \\
			=& \int_{\partial D_{R, \frac{1}{R}}^2} t^{1-2s} \left( \frac{3}{2} R |\nabla\tilde{\hat{u}}_{ik}|^2 + C |\tilde{\hat{u}}_{ik}| |\nabla \tilde{\hat{u}}_{ik}| \right) dS \\
			\le&  C \left( R^{-3-2s} + R^{-4-2s} \right),
		\end{align*}
		which implies that
		\begin{equation*}
			\lim_{R \to \infty} \left| \int_{\partial D_{R, \frac{1}{R}}^2} \mathbf{P} \cdot \nu \, dS \right| = 0.
		\end{equation*} 
		Till now, one can conclude that 
		\begin{equation*}
			\int_{\mathbb{R}^{3}}(-\Delta)^{s}\hat{u}_{ik}[(x-x_{1k})\cdot\nabla\hat{u}_{ik}]dx=\frac{2s-3}{2} \int_{\mathbb{R}^3} \left| (-\Delta)^{\frac{s}{2}}\hat{u}_{ik}\right| ^2dx,
		\end{equation*}
		this therefore completes the proof of Lemma \ref{lem5-4}.	
	\end{proof}
	
	\text{\bf Acknowledgements}
	L.T. Liu is supported by the Postdoctoral Fellowship Program of CPSF under Grant Number GZB20250713, the National Natural Science Foundation of China (No. 12501230), the Natural Science Foundation of Shanxi Province (No. 202303021211056) and the Scientific and Technical Innovation Team for Young Scholars in Universities of Shandong Province (No. 2025KJG015).

	\text{\bf Data availability}
	All data, models, and code generated or used during the study appear in the submitted article.

\end{document}